\documentclass[a4paper,twoside,10pt]{amsart}
\usepackage{amsmath,amsfonts,amssymb,amsthm}
\usepackage{mathtools,mathrsfs,enumerate}
\usepackage{color,graphicx}
\usepackage[a4paper]{geometry}
\usepackage[colorlinks=true,pdfstartview={XYZ null null 1.00},pdfview=FitH,citecolor=blue,urlcolor=customgreen]{hyperref}
\usepackage{caption,subcaption}
\newtheorem{theorem}{Theorem}[section]
\newtheorem{lemma}[theorem]{Lemma}
\newtheorem{proposition}[theorem]{Proposition}
\newtheorem{condition}{Condition}[section]
\numberwithin{equation}{section}
\def\B{\text{\large$\mathtt{B}$}}
\def\im{\mathrm{i}}
\def\d{\mathrm{d}}
\def\e{\mathrm{e}}
\def\sgn{\operatorname{sgn}}

\def\Re{\operatorname{Re}}

\def\Im{\operatorname{Im}}
\definecolor{customgreen}{rgb}{0.0, 0.5, 0.0}

\author[G. Nemes]{Gerg\H{o} Nemes}
\address{Department of Physics, Tokyo Metropolitan University, 1--1 Minami-osawa, Hachioji-shi, Tokyo, Japan 192-0397}
\address{Theoretical Sciences Visiting Program (TSVP), Okinawa Institute of Science and Technology Graduate University, 1919--1 Tancha, Onna-son, Kunigami-gun Okinawa, Japan 904-0495\bigskip}

\email{nemes@tmu.ac.jp}


\begin{document}

\title[Borel summability of formal solutions]{On the Borel summability of formal solutions of certain higher-order linear ordinary differential equations}

\begin{abstract} We consider a class of $n^{\text{th}}$-order linear ordinary differential equations with a large parameter $u$. Analytic solutions of these equations can be described by (divergent) formal series in descending powers of $u$. We demonstrate that, given mild conditions on the potential functions of the equation, the formal solutions are Borel summable with respect to the parameter $u$ in large, unbounded domains of the independent variable. We establish that the formal series expansions serve as asymptotic expansions, uniform with respect to the independent variable, for the Borel re-summed exact solutions. Additionally, we show that the exact solutions can be expressed using factorial series in the parameter, and these expansions converge in half-planes, uniformly with respect to the independent variable. To illustrate our theory, we apply it to an $n^{\text{th}}$-order Airy-type equation.
\end{abstract}
\maketitle

\section{Introduction and main results}\label{section1}

In this paper, we study $n^{\text{th}}$-order ($n\ge 2$) linear ordinary differential equations of the form
\begin{equation}\label{eq1}
\bigg(  - \frac{\d^n }{\d z^n } + \sum\limits_{k = 0}^{n - 2} u^{n - k} f_k (u,z)\frac{\d^k }{\d z^k } \bigg)w(u,z) = 0,
\end{equation}
where $u$ is a real or complex parameter, $z$ lies in some domain $\mathbf{D}$, bounded or otherwise, of a Riemann surface, and the potential functions $f_k(u,z)$ are analytic functions of $z$ possessing asymptotic expansions of the form
\begin{equation}\label{eq23}
f_0 (u,z) \sim \sum\limits_{m = 0}^\infty  \frac{f_{0,m} (z)}{u^m } ,\quad f_k (u,z) \sim \sum\limits_{m = 1}^\infty  \frac{f_{k,m} (z)}{u^m } \qquad (1 \le k \le n - 2)
\end{equation}
as $u\to \infty$ uniformly with respect to $z\in \mathbf{D}$. It is assumed that the coefficient functions $f_{k,m}(z)$ are analytic functions of $z$ in the domain $\mathbf{D}$ and are independent of $u$. Furthermore, we suppose that $f_{0,0}(z)$ does not vanish in $\mathbf{D}$. Note that the term in the $(n-1)^{\text{th}}$ derivative is not present in the equation \eqref{eq1}. This term can always be removed by an appropriate change of dependent variable.

We are interested in asymptotic expansions of analytic solutions of \eqref{eq1} when the parameter $u$ becomes large. Equation \eqref{eq1} belongs to a commonly studied class of linear ordinary differential equations with a large parameter (see, e.g., \cite[Ch. 5, Sec. 1.4]{Fedoryuk1993}). In general, not all of the terms corresponding to $m=0$ in the expansions of the functions $f_k(u,z)$ (for $k\ne 0$) need to be zero. However, in those instances, the roots of the characteristic equation $\xi ^n  = f_{n - 2,0} (z)\xi ^{n - 2} + \ldots + f_{0,0} (z)$ associated with \eqref{eq1} may not possess a closed form, thus introducing difficulties in the analysis. In our case, the characteristic equation takes the simple form
\begin{equation}\label{char}
\xi ^n  = f_{0,0} (z).
\end{equation}
After choosing a specific branch of $f_{0,0}^{1/n} (z)$, we obtain $n$ unique solutions $\xi_j = \e^{2\pi \im j/n} f_{0,0}^{1/n} (z)$, where $0\le j\le n-1$. This suggests the introduction of the transformed variables $\xi$ and $W(u,\xi)$ given by
\begin{equation}\label{eq2}
\xi  = \int_{z_0 }^z f_{0,0}^{1/n} (t)\d t ,\quad W(u,\xi) = f_{0,0}^{(1-1/n)/2} (z)w(u,z).
\end{equation}
The integration path, with the potential exception of its endpoint $z_0$, must lie entirely within $\mathbf{D}$. If $z_0$ happens to be a boundary point of $\mathbf{D}$, it is assumed that the integral converges as $t$ approaches $z_0$ along the integration contour. The inclusion of the factor $f_{0,0}^{(1/n - 1)/2}(z)$ in the definition of $W(u,\xi)$ results in the $(n-1)^{\text{th}}$ derivative remaining absent in the transformed equation, thereby providing a simple asymptotic ansatz for solutions. The transformation \eqref{eq2} maps $\mathbf{D}$ on a domain $\mathbf{G}$, say. In terms of the new variables, equation 
\eqref{eq1} becomes
\begin{equation}\label{eq3}
\bigg(  - \frac{\d^n }{\d\xi ^n } + \sum\limits_{k = 0}^{n - 2} u^{n - k} \psi _k (u,\xi )\frac{\d^k}{\d\xi ^k } \bigg)W(u,\xi )=0,
\end{equation}
where the functions $\psi_k(u,\xi)$ are analytic with respect to $\xi$ in $\mathbf{G}$, and admit asymptotic expansions of the form
\begin{equation}\label{eq15}
\psi _k (u,\xi ) \sim \delta_{0,k}+\sum\limits_{m = 1}^\infty  {\frac{{\psi _{k,m} (\xi )}}{{u^m }}} \qquad (0 \le k \le n - 2),
\end{equation}
as $u\to \infty$ uniformly with respect to $\xi\in \mathbf{G}$. Here $\delta_{0,k}$ represents the Kronecker delta. The functions $\psi _k (u,\xi )$ and $\psi _{k,m} (\xi )$ may be expressed in terms of the functions $f _k (u,z )$ and $f_{k,m} (z)$, respectively (see Appendix \ref{appA}).

It can be verified through direct substitution that equation \eqref{eq3} has, for each $0\le j\le n-1$, formal solutions of the form
\begin{equation}\label{eq4}
\widehat{W}_j (u,\xi ) = \exp \Big( \e^{2\pi \im j/n} u\xi  + X_j (\xi ) \Big)\bigg( 1 + \sum\limits_{m = 1}^\infty  \frac{\mathsf{A}_{j,m} (\xi )}{u^m } \bigg),
\end{equation}
with
\begin{equation}\label{eq5}
X_j (\xi ) = \frac{1}{n}\int^\xi  \bigg( \sum\limits_{k = 0}^{n - 2} \e^{2\pi \im j(k + 1)/n} \psi _{k,1} (t) \bigg)\d t .
\end{equation}
The coefficients $\mathsf{A}_{j,m} (\xi )$, which are analytic functions of $\xi$ in $\mathbf{G}$, can be determined
recursively via
\begin{gather}\label{eq6}
\begin{split}
 & \mathsf{A}_{j,m} (\xi ) =  - \frac{1}{n}\sum\limits_{p = 2}^{\min(n,m + 1)} {\binom{n}{p}\e^{2\pi \im j(n - p + 1)/n} \sum\limits_{r = 0}^{p} {\binom{p}{r}\int^\xi  {\B_{p - r} \frac{{\d^r \mathsf{A}_{j,m - p + 1} (t)}}{{\d t^r }}\d t} } } 
\\ & + \frac{1}{n}\sum\limits_{k = 0}^{n - 2} {\sum\limits_{p = 1}^{\min(m,k)} {\binom{k}{p}\e^{2\pi \im j(k - p + 1)/n} \sum\limits_{r = 0}^{p} {\binom{p}{r}\int^\xi  {\B_{p - r} \psi _{k,1} (t)\frac{{\d^r \mathsf{A}_{j,m - p} (t)}}{{\d t^r }}\d t} } } } 
\\ & + \frac{1}{n}\sum\limits_{k = 0}^{n - 2} {\sum\limits_{q = 0}^{m - 1} {\sum\limits_{p = 0}^{\min(k,q)} {\binom{k}{p}\e^{2\pi \im j(k - p + 1)/n} \sum\limits_{r = 0}^{p} {\binom{p}{r}\int^\xi  {\B_{p - r} \psi _{k,m - q + 1} (t)\frac{{\d^r \mathsf{A}_{j,q - p} (t)}}{{\d t^r }}\d t} } } } } ,
\end{split}
\end{gather}
with the convention $\mathsf{A}_{j,0} (\xi ) = 1$. The constants of integration in \eqref{eq5} and \eqref{eq6} are arbitrary. In \eqref{eq6}, $\B_p  = \B_p  ( X'_j (\xi ),X''_j (\xi ), \ldots ,X_j^{(p)} (\xi ) ) =\exp(-X_j (\xi )) \partial_\xi^p\exp(X_j (\xi ))$ is a $p^{\text{th}}$ complete exponential Bell polynomial (refer to Appendix \ref{Bell}). For a detailed derivation, see Appendix \ref{coeffappendix}.

When $n=2$, equation \eqref{eq1} simplifies to the Schr\"odinger-type equation
\begin{equation}\label{eq7}
\bigg(  - \frac{\d^2 }{\d z^2 } +  u^2 f_0 (u,z)\bigg)w(u,z) = 0.
\end{equation}
In this particular case, the formal solutions \eqref{eq4} are commonly known as WKB solutions, named after the physicists Wentzel \cite{Wentzel1926}, Kramers \cite{Kramers1926}, and Brillouin \cite{Brillouin1926}, who independently discovered them within the context of quantum mechanics in 1926. However, for complete accuracy, it should be noted that the WKB solutions were discussed earlier by Carlini \cite{Carlini1817} in 1817, by Liouville \cite{Liouville1837} and Green \cite{Green1837} in 1837, by Strutt (Lord Rayleigh) \cite{Rayleigh1912} in 1912, and once again by Jeffreys \cite{Jeffreys1924} in 1924. For detailed historical discussion and critical overview, the interested reader is directed to the books by Dingle \cite[Ch. XIII]{Dingle1973} and Wasow \cite[Ch. I]{Wasow1985}.

Typically, the series in \eqref{eq4} do not converge, and the most that can be established is that in certain subregions of $\mathbf{G}$ they provide uniform asymptotic expansions of solutions of the differential equation \eqref{eq3}. For the Schrödinger-type equation \eqref{eq7}, Thorne \cite{Thorne1960} and Olver \cite[Ch. 10, \S9]{Olver1997} demonstrated the asymptotic nature of the expansions \eqref{eq4} by constructing explicit error bounds. Fedoryuk \cite[Ch. 5]{Fedoryuk1993} employed a similar technique to analyse the higher-order case.

In this paper we take a different approach and study the formal solutions $\widehat{W}_j (u,\xi)$ from the perspective of Borel summability. Our primary objective is to demonstrate that the expansions \eqref{eq4} are Borel summable within specific subdomains $\Gamma_j$ of $\mathbf{G}$, under the condition that the functions $\psi_k(u,\xi)$ meet certain moderate criteria. That is, we shall construct exact, analytic solutions $W_j(u, \xi)$ of the differential equation \eqref{eq3} in terms of Laplace transforms (with respect to the parameter $u$) of certain associated functions, known as the Borel transforms of the formal solutions. The formal expansions presented in \eqref{eq4} will then arise as asymptotic series of these exact solutions. Furthermore, it will be demonstrated that the exact solutions $W_j (u,\xi)$ can be represented through factorial series in the parameter $u$, and that these expansions exhibit convergence in half-planes with respect to $u$, uniformly in the independent variable $\xi$.

In recent decades, there has been a growing interest in the study of Borel summability of asymptotic expansions, and this field is closely connected to exponential asymptotics (see, for instance, \cite{OldeDaalhuis2003}). Indeed, extensive progress has been made in the summability of formal WKB solutions of the second-order equation \eqref{eq7}, with comparatively fewer general results available in the case of higher-order equations. The crucial relationship between Borel summability and the analysis of WKB solutions was initially recognised by Bender and Wu \cite{Bender1969}. Dunster et al. \cite{Dunster1993} studied differential equations of the type \eqref{eq7} in the case that $f_0(u,z)=f_{0,0}(z)+f_{0,2}(z)u^{-2}$. They established the Borel summability of WKB solutions away from Stokes curves under appropriate conditions on $f_0(u,z)$, and furthermore, they derived convergent factorial series expansions for these solutions. Further contributions in this direction can be found in the works of Giller and Milczarski \cite{Giller2001} as well as Bodine and Sch\"afke \cite{Bodine2002}. Borel summability results for the WKB solutions of equations of the type \eqref{eq7}, featuring more general potentials $f_0(u,z)$, have been demonstrated in their most recent works by Nikolaev \cite{Nikolaev2023} and the author of the present work \cite{Nemes2021}.

Simultaneously, efforts were made to develop connection formulae for WKB solutions across Stokes curves. In the seminal work by Voros \cite{Voros1983}, he showed how to analyse Borel re-summed WKB solutions and demonstrated the relationship between the singular points of the Borel transforms and the global connection formulae. Aoki et al. \cite{Aoki1991} proposed an approach to perform this singularity analysis by transforming the original equation into a canonical form near the transition point, employing specific transformation series, and investigating the properties of the resulting canonical equation, which are better understood. The idea of a transformation series, in fact, has much earlier roots and can be traced back to a paper by Cherry \cite{Cherry1950}. A rigorous treatment of these transformation series for specific types of transition points has been provided in the works of Kamimoto and Koike \cite{Kamimoto2011}, as well as by Sasaki \cite{Sasaki}. An alternative method for obtaining connection formulae, using convergent factorial series, was proposed by Dunster \cite{Dunster2001,Dunster2004}. Further developments, drawing upon \'Ecalle's theory of resurgent functions, were undertaken by Delabaere et al. \cite{Delabaere1997,Delabaere1999}.

As mentioned above, the study of Borel summability of formal solutions of higher-order equations has remained largely unexplored. Aoki et al. \cite{Aoki2001} introduced the exact steepest descent method to analyse the Borel summability of formal solutions to equation \eqref{eq1} in the particular case where the potential functions $f_k(u, z)$ are polynomials in $z$ of degree at most two and independent of $u$. Possible extension to higher-order polynomial potentials relies on the Borel summability of formal solutions to certain higher-order equations.

In 1982, Berk et al. \cite{Berk1982} studied the connection formulae across Stokes curves for formal solutions of a third-order equation given by
\[
\bigg(-\frac{\d^3}{\d z^3}-u^2 3\frac{\d}{\d z}+u^3\im z\bigg)w(u,z)=0.
\]
(Note that this equation is not within the scope of the class studied in this paper since the potential $f_1(u,z)=3$ does not exhibit an asymptotic expansion of the form specified in \eqref{eq23}.) They discovered that the Stokes curves intersected each other at non-critical points, requiring the introduction of new Stokes curves in the monodromy of solutions. These additional Stokes curves originated from what appeared to be regular points rather than transition points, later referred to as virtual turning points by Aoki et al. \cite{Aoki1994}. Notably, it was observed that the newly introduced Stokes curves only became ``active" once they passed through the crossing points of the ordinary Stokes curves. No Stokes phenomenon occurred along the segments connecting the virtual turning point and the crossing points. The explanation of this remarkable phenomenon, now known as the higher-order Stokes phenomenon, was subsequently provided by Howls et al. \cite{Howls2004,Howls2007}, employing hyperasymptotic techniques. In a higher-order Stokes phenomenon, a Stokes multiplier itself can change its value, explaining the apparent sudden birth of new Stokes curves at regular points (virtual turning points). This phenomenon was found to be present also in the case of integrals with saddles, inhomogeneous equations and partial differential equations. Nevertheless, these new Stokes curves are absent in our analysis and hence do not affect the results presented in this paper. This absence could likely be attributed to either the simplicity of the characteristic equation \eqref{char} associated with \eqref{eq1} or the restrictions on the domains $\Gamma_j$ where Borel summability is established.

Before presenting our results, we will introduce the necessary assumptions, notation, and definitions. For any $n\geq 2$, we will denote by $\mathfrak{a}_n$ the unique positive real solution of the equation
\[
(1 + \mathfrak{a}_n)^n  = 1 + 2n\mathfrak{a}_n.
\] The sequence $\mathfrak{a}_n$ will naturally arise when describing the region where the Borel transforms are analytic, particularly when estimating the coefficients $\mathsf{A}_{j,m}(\xi)$. For basic properties of this sequence, the reader is referred to Appendix \ref{aseq}. We define, for any $r>0$,
\[
B(r) = \left\{ t:\left| t \right| < r \right\} \quad \text{and} \quad U(r) = \left\{ t:\Re(t) > 0,\,  |\Im(t)| < r \right\} \cup B(r).
\]
Thus, $U(r)$ comprises all points whose distance from the positive real axis is strictly less than $r$. Subsequently, we will assume that the asymptotic expansions \eqref{eq15} of the functions $\psi_k(u,\xi)$ are Borel summable, and that their respective Borel transforms satisfy certain growth conditions. The precise conditions are as follows.

\begin{condition}\label{cond} There is a non-empty domain $\Delta \subseteq \mathbf{G}$ and a set of functions $\Psi_k(t,\xi)$, where $0\leq k\leq n-2$, that satisfy the following three conditions.

\begin{enumerate}[(i)]\itemsep0.5em
\item The functions $\Psi_k(t,\xi)$ are analytic on $U( \mathfrak{a}_nd')\times \Delta$, where $d'$ is a positive constant independent of $u$, $\xi$ and $k$.
\item There exist constants $c>0$, $\rho>0$, and $\sigma' \ge 0$, which are independent of $u$, $\xi$ and $k$, such that for all $(t,\xi) \in U(\mathfrak{a}_n d')\times \Delta$, the following inequalities hold: 
\[
|\Psi_k(t,\xi)| \le c\frac{\e^{\sigma' \Re(t)}}{1 + \left| \xi  \right|^{1 + \rho }}\qquad (0 \le k \le n - 2).
\]
\item It holds that
\begin{equation}\label{eq8}
 \psi_k (u,\xi ) = \delta_{0,k}+\int_0^{ + \infty }\e^{ - u t} \Psi_k(t,\xi )\d t \qquad (0 \le k \le n - 2),
\end{equation}
when $\Re(u) > \sigma'$ and $\xi \in \Delta$.
\end{enumerate}
\end{condition}
We shall introduce some additional definitions as follows. Write
\[
\theta _{j,\ell}  =  - \arg (\e^{2\pi \im \ell/n}  - \e^{2\pi \im j/n} ), \qquad 0\le \ell \le n-1,\quad \ell\neq j.
\]
Although these angles are initially defined modulo $2\pi$, we can always choose them in such a way that the following condition holds for each $0\leq j\leq n-1$:
\begin{equation}\label{theta1}
\max_{\ell\neq j} \theta _{j,\ell}  - \min_{\ell\neq j} \theta _{j,\ell}  = \frac{n - 2}{n}\pi .
\end{equation}
Henceforth, we will assume that this condition is met. It is worth noting that
\[
\theta _{j,\ell}  =\frac{{(2-\sgn (j - \ell))n - 2(j + \ell)}}{{2n}}\pi \bmod 2\pi.
\]
Let $d>0$. We denote, for each $0\le j\le n-1$, by $\Gamma_j(d)$ any (non-empty) subdomain of $\Delta$ that satisfies the following
two requirements:
\begin{enumerate}[(i)]
    \item The distance between each point of $\Gamma_j(d)$ and each boundary point of $\Delta$ has lower bound $d$ (which is to be chosen independently of $u$).
\item If $\xi \in \Gamma_j(d)$, then for each $0\le \ell\le n-1$, $\ell\ne j$, the infinite path $\mathscr{P}_{j,\ell}(\xi)$ parametrised by
\[
[0,+\infty) \ni t \mapsto \xi+t\e^{\im \theta _{j,\ell}}
\]
lies wholly within $\Gamma_j(d)$.
\end{enumerate}
Note that condition (ii) is equivalent to the requirement that if $\xi \in \Gamma_j(d)$, then $\Gamma_j(d)$ contains the closed, infinite sector with vertex $\xi$, opening $\frac{n - 2}{n}\pi$, and radii parametrised by
\[
[0,+\infty) \ni t \mapsto \xi+t\exp\Big(\im \min_{\ell\neq j} \theta _{j,\ell}\Big) \quad \text{and} \quad[0,+\infty) \ni t \mapsto \xi+t\exp\Big(\im \max_{\ell\neq j} \theta _{j,\ell}\Big)
\]
(if $n=2$, the sector degenerates into a half line); cf. Figure \ref{Figure1a}. Furthermore, condition (ii) implies that $\Delta$ (and hence $\mathbf{G}$) must be unbounded.

\begin{figure}[t]
\begin{subfigure}[c]{0.44\textwidth}
\centering
\includegraphics[width=\textwidth]{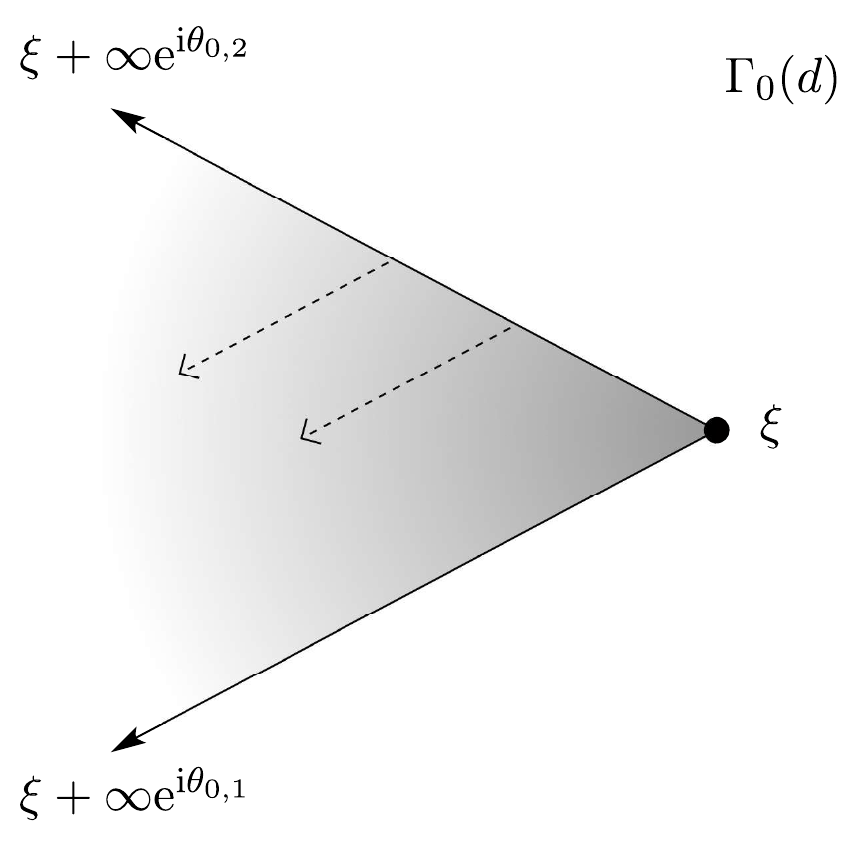}
\caption{}
\label{Figure1a}
\end{subfigure}
\hfill
\begin{subfigure}[c]{0.4675\textwidth}
\centering 
\includegraphics[width=\textwidth]{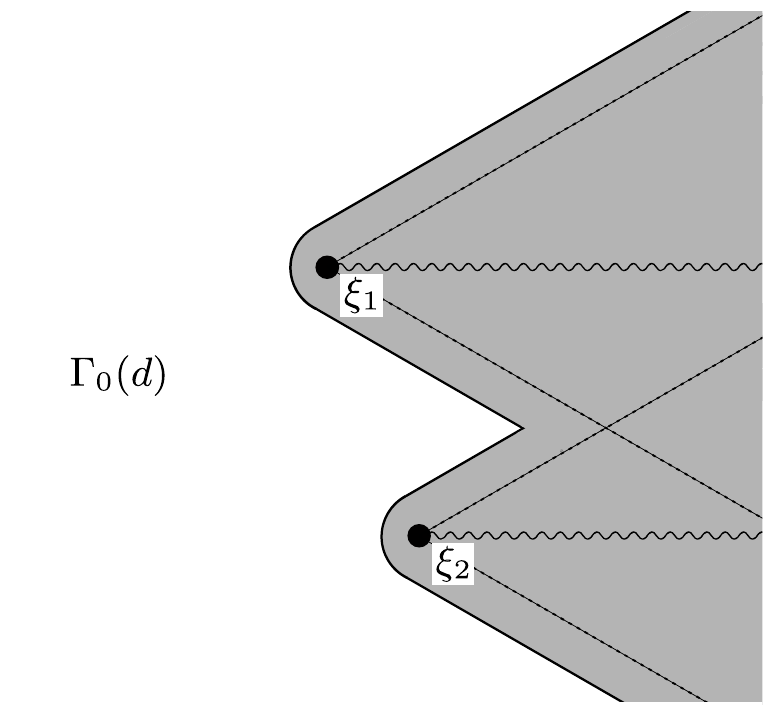}
\caption{}
\label{Figure1b}
\end{subfigure}
\caption{(a) The structure of $\Gamma_0(d)$ when $n = 3$. If $\xi \in \Gamma_0(d)$, then $\Gamma_0(d)$ must necessarily contain a sector with opening $\frac{\pi}{3}$ and vertex $\xi$. (b) Illustration of a domain $\Gamma_0(d)$ (unshaded) when $n = 3$. The points $\xi_1$ and $\xi_2$ are singularities of the functions $\psi_0(u, \xi)$ and $\psi_1(u, \xi)$. Zig-zag lines represent branch cuts, while dashed lines depict Stokes lines. The distance between Stokes lines and the boundary of $\Gamma_0(d)$ is at least $d$.}
\label{Figure1}
\end{figure}

For each $0\leq j\leq n-1$, we define the angle $\theta_j$ as follows:
\begin{equation}\label{theta2}
\theta _j = \frac{1}{2}\Big(\min_{\ell\neq j} \theta _{j,\ell} + \max_{\ell\neq j} \theta _{j,\ell}\Big).
\end{equation}
Subsequently, we define the infinite path $\mathscr{P}_j(\xi)$ using the parametrisation
\[
[0,+\infty) \ni t \mapsto \xi+t\e^{\im \theta _j}.
\]
Note that $\mathscr{P}_j(\xi)$ is entirely contained within $\Gamma_j(d)$. We seek solutions of the differential equation \eqref{eq3} of the form
\begin{equation}\label{eq17}
W_j (u,\xi ) = \exp \Big( \e^{2\pi \im j/n} u\xi  + X_j (\xi ) \Big)\big( 1 + \eta_j(u,\xi) \big),
\end{equation}
where
\begin{equation}\label{eq18}
X_j (\xi ) = -\frac{1}{n}\int_{\mathscr{P}_j(\xi)}  \bigg( \sum\limits_{k = 0}^{n - 2} \e^{2\pi \im j(k + 1)/n} \psi _{k,1} (t) \bigg)\d t ,
\end{equation}
$\Re(u)>\sigma'$, $\xi \in \Gamma_j(d)$, and the functions $\eta_j(u,\xi)$ satisfy the limit conditions
\begin{equation}\label{eq19}
\lim_{t \to  + \infty } \left[ \frac{{\d^s \eta _j (u,\zeta )}}{{\d\zeta ^s }} \right]_{\zeta  = \xi  + t\e^{\im\theta _j } }  = 0 \qquad (0\le s\le n-2).
\end{equation}
The convergence of the integral in \eqref{eq18} is assured by Condition \ref{cond} (cf. Lemmas \ref{lemma1} and \ref{lemma5}). Furthermore, we seek our solutions to exhibit asymptotic expansions in the form
\begin{equation}\label{eq22}
\eta_j(u,\xi) \sim \sum\limits_{m = 1}^\infty  \frac{\mathsf{A}_{j,m} (\xi )}{u^m },
\end{equation}
as $u\to \infty$ uniformly with respect to $\xi\in \Gamma_j(d)$. The coefficients $\mathsf{A}_{j,m} (\xi )$ must satisfy recurrence relations of the type \eqref{eq6} with a suitable choice of integration constants. Note that the conditions in \eqref{eq19} imply
\begin{equation}\label{eq20}
\lim_{t \to  + \infty } \left[ \frac{\d^s \mathsf{A}_{j,m} (\zeta )}{\d\zeta ^s } \right]_{\zeta  = \xi  + t\e^{\im\theta _j } }  = 0 \qquad (0\le s\le n-2).
\end{equation}
In Section \ref{coeffestsec}, we will show that if $\mathsf{A}_{j,0}(\xi) = 1$ and
\begin{gather}\label{eq21}
\begin{split}
 & \mathsf{A}_{j,m} (\xi ) =  \frac{1}{n}\sum\limits_{p = 2}^{\min(n,m + 1)} {\binom{n}{p}\e^{2\pi \im j(n - p + 1)/n} \sum\limits_{r = 0}^{p} {\binom{p}{r}\int_{\mathscr{P}_j(\xi)}  {\B_{p - r} \frac{{\d^r \mathsf{A}_{j,m - p + 1} (t)}}{{\d t^r }}\d t} } } 
\\ & - \frac{1}{n}\sum\limits_{k = 0}^{n - 2} {\sum\limits_{p = 1}^{\min(m,k)} {\binom{k}{p}\e^{2\pi \im j(k - p + 1)/n} \sum\limits_{r = 0}^{p} {\binom{p}{r}\int_{\mathscr{P}_j(\xi)}  {\B_{p - r} \psi _{k,1} (t)\frac{{\d^r \mathsf{A}_{j,m - p} (t)}}{{\d t^r }}\d t} } } } 
\\ & - \frac{1}{n}\sum\limits_{k = 0}^{n - 2} {\sum\limits_{q = 0}^{m - 1} {\sum\limits_{p = 0}^{\min(k,q)} {\binom{k}{p}\e^{2\pi \im j(k - p + 1)/n} \sum\limits_{r = 0}^{p} {\binom{p}{r}\int_{\mathscr{P}_j(\xi)}  {\B_{p - r} \psi _{k,m - q + 1} (t)\frac{{\d^r \mathsf{A}_{j,q - p} (t)}}{{\d t^r }}\d t} } } } }
\end{split}
\end{gather}
for $m\ge 1$ and $\xi\in \Gamma_j(d)$, then the requirements \eqref{eq20} are satisfied. The $\mathsf{A}_{j,m}(\xi)$ defined in this way will serve as the coefficients in the asymptotic expansions \eqref{eq22} of our solutions \eqref{eq17}. Throughout the rest of the paper, unless explicitly stated otherwise, $\mathsf{A}_{j,m}(\xi)$ will consistently denote the coefficients determined by the recurrence relation \eqref{eq21}.

We are now in a position to formulate our main result.

\begin{theorem}\label{thm1}
Let $0<d<d'$. If Condition \ref{cond} holds then the differential equation \eqref{eq3} admits unique solutions $W_j(u,\xi)$, $0\le j\le n-1$, of the form \eqref{eq17} with the following properties. The functions $\eta_j (u,\xi )$ are analytic in $\left\{ u : \Re(u) > \sigma' \right\}  \times  \Gamma_j (d)$, meeting the limit conditions \eqref{eq19}, and can be represented in the form
\[
\eta_j(u,\xi)=\int_0^{+\infty}\e^{-u t}F_j(t,\xi)\d t.
\]
The Borel transforms $F_j(t, \xi)$ are analytic functions in $U(\mathfrak{a}_n d)\times \Gamma_j(d)$, with convergent power series expansions
\begin{equation}\label{eq34}
F_j(t,\xi)=\sum_{m=0}^\infty \frac{\mathsf{A}_{j,m+1}(\xi)}{m!} t^m,
\end{equation}
valid for $(t, \xi) \in B(\mathfrak{a}_n d)\times \Gamma_j(d)$. Moreover, for each $\sigma > \sigma'$ and $0 < r < d$, there exists a positive constant $K = K(\sigma, r)$, independent of $u$, $t$, $\xi$, and $j$, such that
\begin{equation}\label{Fbound}
\left|F_j(t,\xi)\right| \le K \e^{\sigma \Re(t)} \sum\limits_{\ell  \ne j} \frac{1}{{(\max (1,\Re(\xi \e^{ - \im\theta _{j,\ell } } )))^\rho  }} ,
\end{equation}
for $(t, \xi) \in U(\mathfrak{a}_n r) \times \Gamma_j(d)$. Finally, for each $\sigma > \sigma'$,
\begin{equation}\label{eq35}
\eta_j(u,\xi) \sim \sum\limits_{m = 1}^\infty  \frac{\mathsf{A}_{j,m} (\xi )}{u^m }
\end{equation}
as $u\to \infty$ in the closed half-plane $\Re(u) \ge \sigma$, uniformly with respect to $\xi\in \Gamma_j(d)$.
\end{theorem}

\paragraph{\emph{Remarks.}}
\begin{enumerate}[(i)]\itemsep0.5em
\item To gain a better understanding of the main result, consider the following particular case. Suppose that the functions $\psi_k(u,\xi)$ are (potentially multivalued) analytic functions in the $\xi$-plane, except for a finite set of singularities (which could involve branch points) situated at $\xi = \xi_m$, where $m=1,2,\ldots$. For simplicity, we focus on the Borel summability of the formal solution $\widehat{W}_{0}(u,\xi)$. We introduce branch cuts extending from each singularity $\xi_m$ towards
\[
\xi_m - \infty\e^{\im \theta _0}.
\]
Subsequently, we remove each infinite (closed) sector that has its vertex at $\xi_m$, an opening angle of $\frac{n-2}{n}\pi$, and the branch cut serving as the axis of symmetry. The pre-image in the $z$-plane of each of the radii of these sectors is a Stokes curve. These curves emanate either from a zero of $f_{0,0}(z)$ or from a singularity of one of the functions $f_k(u,z)$. Suppose that the functions $\psi_k(u,\xi)$ satisfy Condition \ref{cond} within this resulting $\xi$-domain, which we denote by $\Delta_0$. Then we can set 
\[
\Gamma_0(d)=\Big\{ \xi: \xi \in \Delta_0, \, \inf_{\zeta\in\Delta_0}|\xi-\zeta|\ge d \Big\};
\]
for example, see Figure \ref{Figure1b}. Consequently, under these circumstances, Theorem \ref{thm1} ensures the Borel summability of the formal solution $\widehat{W}_{0}(u,\xi)$ within a Stokes region, away from Stokes curves. A similar assertion applies to the remaining $n-1$ formal solutions.

\item Drawing from the theory of second-order equations \cite{Giller2000,Giller2001}, we anticipate that, in quite general scenarios, the singularities of the Borel transform $F_j(t, \xi)$ within the $t$-plane will be situated at
\[
t =  - (\e^{2\pi \im\ell /n}  - \e^{2\pi \im j/n} )(\xi  - \xi _m ),\quad \ell\neq j.
\]
In line with the preceding remark, $\xi_m$ denotes a singularity of any of the functions $\psi_k(u,\xi)$. Consequently, the power series \eqref{eq34} should converge within the disc
\[
\left| t \right| < \min_{\ell  \ne j} \big| \e^{2\pi \im\ell /n}  - \e^{2\pi \im j/n} \big|\inf_{\xi  \in \Gamma_j (d)} \left| \xi  - \xi _m \right| = 2\sin \left(\frac{\pi }{n}\right)d.
\]
For $n=2$, this aligns with the disc $B(\mathfrak{a}_2 d)=B(2d)$, and for large values of $n$,
\[
2\sin \left( \frac{\pi }{n}\right) \sim \frac{2\pi}{n} \sim (5.0008191965\ldots) \times \mathfrak{a}_n 
\]
via Proposition \ref{athm}. Consequently, the domains $B(\mathfrak{a}_n d)$ and $U(\mathfrak{a}_n d)$ that our analysis suggests do not seem unrealistic.

\item Using Theorem \ref{thm1}, we can establish the Borel summability of formal solutions whose $m^{\text{th}}$ asymptotic expansion coefficients evaluate to some prescribed complex values $\alpha_{j,m}$ at certain fixed $\beta_j \in \Gamma_j (d)$. Assuming the formal series
\[
\sum\limits_{m = 1}^\infty \frac{\alpha_{j,m}}{u^m}
\]
are Borel summable in $U(\mathfrak{a}_n d)$ with some $0<d<d'$, provided $\Re (u)>\sigma'' \geq 0$, let $G_j (t)$ denote their respective Borel transforms. If $W_j (u,\xi)$ represent the exact solutions given by \eqref{eq17} and Theorem \ref{thm1}, then
\[
W_j (u,\xi,\beta_j) = W_j (u,\xi)\left(1 + \int_0^{+\infty} \e^{-ut} F_j(t,\beta_j )\d t\right)^{-1} \left(1+\int_0^{+\infty} \e^{-ut} G_j(t)\d t\right)
\]
also stand as solutions of the differential equation \eqref{eq3} for sufficiently large $u$. The right-hand side can be re-expressed as
\[
\exp \Big( \e^{2\pi \im j/n} u\xi  + X_j (\xi ) \Big)\left(1 + \int_0^{+\infty} \e^{-ut} F_j(t,\xi,\beta_j )\d t\right),
\]
where, for each $0<r<d$, the functions $F_j (t,\xi,\beta_j )$ are analytic in $(t,\xi) \in U(\mathfrak{a}_n r) \times \Gamma_j (d)$, and the Laplace transforms converge for $\Re(u) > K + \sigma'' + \sigma$, where $\sigma>\sigma'$ and $K = K(\sigma,r)$ is the corresponding constant provided in Theorem \ref{thm1} (cf. \cite[Ch. 5, Theorem 5.55]{Mitschi2016}). By invoking Watson's lemma \cite[Ch. 4, \S3]{Olver1997}, the Laplace transforms yield asymptotic expansions in inverse powers of the large parameter $u$, meeting the aforementioned requirements due to their construction.

\item In the case of second-order equations ($n=2$) with a potential function of the form $f_0(u,z) = f_{0,0}(z) + f_{0,1}(z)u^{-1} + f_{0,2}(z)u^{-2}$, Theorem \ref{thm1} is equivalent to Theorem 1.1 of the recent paper by the present author \cite{Nemes2021}. In \cite{Nemes2021}, the Borel transforms $F_j(t,\xi)$, where $j=1,2$, are constructed by establishing and solving partial differential equations. However, in the general scenario where some of the asymptotic expansions \eqref{eq23} (or equivalently \eqref{eq15}) do not terminate after a finite number of terms, the presence of the Borel transforms $\Psi_k(t,\xi)$ introduces complexities that complicate such an analysis, leading to the breakdown of the argument used in \cite{Nemes2021}. In the present work, addressing this difficulty, we adopted an alternative proof strategy relying on ordinary differential equation techniques instead of partial differential equation methods.

\end{enumerate}

In the following theorem, we present convergent alternatives to the asymptotic expansions \eqref{eq35}. The coefficients in these expansions depend on an additional (positive) parameter $\omega$ and can be computed via the formula
\[
\mathsf{B}_{j,m}(\omega,\xi)=\sum\limits_{r=0}^m\omega^{m-r}|s(m,r)|\mathsf{A}_{j,r+1}(\xi),
\]
where $m\ge 0$, and $s(m,r)$ denote the Stirling numbers of the first kind. Note that the $\mathsf{B}_{j,m}(\omega, \xi)$ are polynomials in $\omega$ of degree $m$ and are analytic functions of $\xi \in \Gamma_j(d)$.

\begin{theorem}\label{thm2}
Let $0<d<d'$ and $\omega  > \frac{\pi}{2\mathfrak{a}_n d}$. If Condition \ref{cond} holds, then the
functions $\eta _j (u,\xi )$ given in Theorem \ref{thm1} possess factorial expansions of the form
\[
\eta _j (u,\xi ) = \sum\limits_{m = 0}^\infty \frac{\mathsf{B}_{j,m} (\omega ,\xi )}{u(u + \omega ) \cdots (u + m\omega )} ,
\]
which are absolutely convergent for $\Re(u) \ge \sigma > \sigma'$, uniformly for $\xi \in \Gamma_j(d)$.
\end{theorem}

The deduction of Theorem \ref{thm2} from Theorem \ref{thm1} is completely analogous to that of the corresponding result in the specific case of $n=2$ \cite[Theorem 1.3]{Nemes2021}. Consequently, the proof is omitted.

The remaining part of the paper is organised as follows. The proof of Theorem \ref{thm1} consists of three steps. The first step is presented in Section \ref{coeffestsec}, where we establish that the coefficients $\mathsf{A}_{j,m} (\xi)$ and their derivatives satisfy Gevrey-type estimates. Moving on to the second step, in Section \ref{gevreyremainder}, we construct the functions $\eta_j(u, \xi)$ specified in Theorem \ref{thm1} in the form of truncated asymptotic expansions, each accompanied by its respective remainder term, possessing a Gevrey-type bound. The proof relies on solving specific integral equations, building upon the results established in Section \ref{coeffestsec}. The third and final step, detailed in Section \ref{Boreltransformsec}, involves the construction of the corresponding Borel transforms $F_j(t,\xi)$. This process relies on inverse Laplace transforming the functions $\eta_j(u, \xi)$ constructed in Section \ref{gevreyremainder}. We illustrate our theory in Section \ref{application} by applying it to an $n^{\text{th}}$-order Airy-type equation. The paper concludes with a discussion in Section \ref{discussion}.

\section{Gevrey-type estimates for the coefficients}\label{coeffestsec}

The primary objective of this section is to establish Gevrey-type estimates for the coefficients $\mathsf{A}_{j,m} (\xi)$ and their derivatives. Our findings can be summarised as follows.

\begin{proposition}\label{prop1}
Let $0<d<d'$. If Condition \ref{cond} is satisfied, then the following inequalities hold for each $0\le j\le n-1$, and any positive integers $m$ and $s$:
\begin{align*}
& \left| \mathsf{A}_{j,m} (\xi ) \right| \le  \frac{1}{\mathfrak{a}_n^{m - 1}}\frac{1}{d^m}\frac{(m + L+1)!}{(\max (1,\Re (\xi \e^{ - \im \theta _j } )))^{\rho } },
\\ & \left| \frac{\d^s \mathsf{A}_{j,m} (\xi )}{\d\xi ^s }\right| \le  \frac{1}{\mathfrak{a}_n^{m - 1} }\frac{1}{d^{m + s - 1}}\frac{(m + L+s )!}{1 + \left| \xi  \right|^{1 + \rho } }
\end{align*}
provided $\xi \in \Gamma_j(d)$. Here, $L$ is a suitable positive integer that is independent of $u$, $\xi$, $j$, $m$ and $s$.
\end{proposition}

Equation \eqref{eq3} can be transformed into a first-order system using standard methods. Sibuya \cite{Sibuya2000} demonstrated that the coefficients of a formal solution to such a system satisfy Gevrey-type estimates. However, in our specific scenario, we require more explicit information regarding the domain of validity and the specific constants featured in these estimates.

In order to prove Proposition \ref{prop1}, we shall establish a series of lemmas.

\begin{lemma}\label{lemma6} Let $\alpha$, $\beta$ and $\gamma$ be non-negative integers with $\beta \ge \gamma$. Then, we have
\[
\sum\limits_{q  = 0}^\alpha  \binom{\alpha}{q} (\gamma + q )!(\alpha  + \beta - \gamma - q )! = \frac{{(\alpha + \beta + 1)!}}{{(\beta + 1)\binom{\beta}{\gamma}
}}.
\]
\end{lemma}

\begin{proof} We express the given sum in terms of Pochhammer symbols and the Gauss hypergeometric function ${}_2 F_1$ evaluated at unity, and then employ the Chu--Vandermonde identity \cite[\href{http://dlmf.nist.gov/15.4.E24}{Eq. 15.4.24}]{DLMF}:
\begin{align*}
&\sum\limits_{q  = 0}^\alpha  {\binom{\alpha}{q}(\gamma  +q )!(\alpha  + \beta  - \gamma  - q )!}  = \frac{{(\alpha  + \beta )!}}{{
   \binom{ - \alpha  - \beta  + \gamma  + 1}{\gamma} }}\sum\limits_{q  = 0}^\alpha  {( - 1)^\gamma  \frac{{( - \alpha )_q  (\gamma  + 1)_q  }}{{( - \alpha  - \beta  + \gamma )_q  }}\frac{1}{{q !}}} \\ & = \frac{{(\alpha  + \beta )!}}{{ \binom{\alpha  + \beta }{\gamma} }}\sum\limits_{q  = 0}^\alpha  {\frac{{( - \alpha )_q (\gamma  + 1)_q  }}{{( - \alpha  - \beta  + \gamma )_q  }}\frac{1}{{q !}}} 
 = \gamma !(\alpha  + \beta  - \gamma )!{}_2F_1 ( - \alpha ,\gamma  + 1, - \alpha  - \beta  + \gamma ,1)
\\ &  = \gamma !(\alpha  + \beta  - \gamma )!\frac{{( - \alpha  - \beta  - 1)_\alpha  }}{{( - \alpha  - \beta  + \gamma )_\alpha  }} = \gamma !(\alpha  + \beta  - \gamma )!\frac{{(\alpha  + \beta  + 1)!(\beta  - \gamma )!}}{{(\beta  + 1)!(\alpha  + \beta  - \gamma )!}} = \frac{{(\alpha  + \beta  + 1)!}}{{(\beta  + 1)\binom{\beta}{\gamma}
}}.
\end{align*}
\end{proof}

\begin{lemma}\label{lemma7} Assume that $\alpha$ and $\beta$ are non-negative integers. Then, the following inequalities hold:
\begin{equation}\label{ineq1}
\alpha !\beta ! \le (\alpha  + \beta  - 1)! 
\end{equation}
provided that $\alpha ,\beta  \ge 1$,
\begin{equation}\label{ineq2}
\alpha !\beta ! \le (\alpha  + \beta )!
\end{equation}
provided that $\alpha ,\beta  \ge 0$. Furthermore, if $\alpha \ge 1$, we have
\begin{equation}\label{ineq3}
\sum\limits_{q = 0}^{\alpha  - 1} {(\alpha  - q)!(q + 1)!}  \le 4 \cdot \alpha !
\end{equation}
and
\begin{equation}\label{ineq4}
\sum\limits_{q = 1}^\alpha  {(\alpha  - q)!q!}  \le 2 \cdot \alpha !.
\end{equation}
Finally, if $\alpha \ge 1$ and $\beta \ge 2$, then
\begin{equation}\label{ineq5}
\sum\limits_{q = 0}^{\alpha - 1} {\frac{{(\alpha - q)!(q + \beta)!}}{{q + 1}}} \le 2\frac{{(\alpha + \beta)!}}{{\alpha + 1}}
\end{equation}
and
\begin{equation}\label{ineq6}
\sum\limits_{q = 1}^{\alpha} {\frac{{(\alpha - q)!(q + \beta)!}}{{q + 1}}} \le 3\frac{{(\alpha + \beta)!}}{{\alpha + 1}} .
\end{equation}
\end{lemma}

\begin{proof} Inequality \eqref{ineq1} is clearly satisfied when $\alpha \ge 1$ and $\beta = 1$. If $\alpha \ge 1$ and $\beta \ge 2$, then
\[
(\alpha  + \beta  - 1)! = \alpha !(\alpha  - 1 + 2) \cdots (\alpha  - 1 + \beta ) \ge \alpha !\beta !.
\]

If either $\alpha$ or $\beta$ (or both) is equal to $0$, then inequality \eqref{ineq2} certainly holds. Conversely, when both $\alpha$ and $\beta$ are greater than $0$, it is evident that inequality \eqref{ineq2} is less restrictive compared to \eqref{ineq1}.

The validity of inequality \eqref{ineq3} can be directly verified for $\alpha = 1$ and $2$, and for $\alpha \geq 3$:
\begin{multline*}
\sum\limits_{q = 0}^{\alpha  - 1} {(\alpha  - q)!(q + 1)!}  = 2 \cdot \alpha ! + (\alpha  + 1)!\sum\limits_{q = 1}^{\alpha  - 2} \frac{1}{
   \binom{\alpha  + 1}{q + 1}
} \\ \le 2 \cdot \alpha ! + (\alpha  + 1)!\sum\limits_{q = 1}^{\alpha  - 2} {\frac{2}{{(\alpha  + 1)\alpha }}}  \le 2 \cdot \alpha ! + 2 \cdot \alpha ! = 4 \cdot \alpha !.
\end{multline*}

The validity of inequality \eqref{ineq4} is apparent when $\alpha = 1$. For $\alpha \geq 2$, we can make use of inequality \eqref{ineq1} as follows:
\[
\sum\limits_{q = 1}^\alpha  {(\alpha  - q)!q!}  = \alpha ! + \sum\limits_{q = 1}^{\alpha  - 1} {(\alpha  - q)!q!}  \le \alpha ! + \sum\limits_{q = 1}^{\alpha  - 1} {(\alpha  - 1)!}  \le 2 \cdot \alpha !.
\]

The validity of inequality \eqref{ineq5} is readily apparent for $\alpha = 1$, and for $\alpha \geq 2$, we have
\begin{multline*}
 \sum\limits_{q = 0}^{\alpha - 1} {\frac{{(\alpha - q)!(q + \beta)!}}{{q + 1}}}  = \frac{{(\alpha + \beta)!}}{{\alpha(\alpha + \beta)}} + (\alpha + \beta)!\sum\limits_{q = 0}^{\alpha - 2} {\frac{1}{{(q + 1) \binom{\alpha + \beta}{q + \beta}}}} 
\\ \le \frac{{(\alpha + \beta)!}}{{\alpha(\alpha + \beta)}} + (\alpha + \beta - 2)!\sum\limits_{q = 0}^{\alpha - 2} {\frac{2}{{q + 1}}}  \le \frac{{(\alpha + \beta)!}}{{\alpha(\alpha + \beta)}} + (\alpha + \beta - 1)! \le 2\frac{{(\alpha + \beta)!}}{{\alpha + 1}}.
\end{multline*}

Inequality \eqref{ineq6} is clearly satisfied when $\alpha=1$ and $2$. If $\alpha\ge 3$, then 
\begin{align*}
\sum\limits_{q = 1}^{\alpha} {\frac{{(\alpha - q)!(q + \beta)!}}{{q + 1}}} & = \frac{{(\alpha + \beta)!}}{{\alpha + 1}} + \frac{{(\alpha + \beta - 1)!}}{\alpha} + (\alpha + \beta)!\sum\limits_{q = 1}^{\alpha - 2} {\frac{1}{{(q + 1)\binom{\alpha + \beta}{q + \beta}  }}} 
\\ & \le \frac{{(\alpha+ \beta)!}}{{\alpha + 1}} + \frac{{(\alpha + \beta)!}}{{\alpha(\alpha + \beta)}} + (\alpha +\beta - 2)!\sum\limits_{q = 1}^{\alpha - 2} {\frac{2}{{q + 1}}} 
\\ & \le \frac{{(\alpha + \beta)!}}{{\alpha + 1}} + \frac{{(\alpha + \beta)!}}{{\alpha(\alpha + \beta)}} + (\alpha + \beta - 1)! \le 3\frac{{(\alpha + \beta)!}}{{\alpha + 1}}.
\end{align*}
\end{proof}

\begin{lemma}\label{lemma1}Let $0<d<d'$ and $\sigma>\sigma'$. If Condition \ref{cond} is satisfied, then for each $0\le k\le n-2$ and any positive integer $N$, the functions $\psi _k (u,\xi)$ can be represented in the form
\[
\psi _k (u,\xi ) = \delta _{0,k}  + \sum\limits_{m = 1}^{N-1} {\frac{{\psi _{k,m} (\xi )}}{{u^m }}}  + \frac{\psi_{k,N}(u,\xi)}{u^N},
\]
where the functions $\psi _{k,m} (\xi )$ and $\psi_{k,N}(u,\xi)$ satisfy the following inequalities:
\[
\left| {\psi _{k,m} (\xi )} \right| \le c_1 \frac{1}{\mathfrak{a}_n^{m-1}}\frac{1}{{d^{m - 1} }}\frac{{(m - 1)!}}{{1 + \left| \xi  \right|^{1 + \rho } }},\quad \left| {\psi _{k,N} (u,\xi )} \right| \le c_1\frac{1}{\mathfrak{a}_n^N}\frac{1}{{d^N }}\frac{{N!}}{{1 + \left| \xi  \right|^{1 + \rho } }}
\]
provided that $\Re(u)\ge \sigma$ and $\xi \in \Delta$. Here, $c_1$ is a suitable constant that is independent of $u$, $\xi$, $k$, $m$ and $N$.
\end{lemma}

\begin{proof} We shall prove that Lemma \ref{lemma1} holds with $c_1  = c\e^{\sigma'\mathfrak{a}_n d} ((\sigma  - \sigma ')^{ - 1}  + \max (1,\mathfrak{a}_n d))$. By performing $N$ successive integration by parts in \eqref{eq8}, we obtain
\[
\psi _k (u,\xi ) = \delta _{0,k}  + \sum\limits_{m = 1}^{N-1} \frac{\psi _{k,m} (\xi )}{u^m }  + \frac{1}{u^N }\left( \psi _{k,N} (\xi ) + \int_0^{ + \infty } \e^{ - ut} \frac{\d^N \Psi _k (t,\xi )}{\d t^N }\d t \right),
\]
where 
\[
\psi _{k,m} (\xi )  = \left[ {\frac{{\d^{m - 1} \Psi _k (t,\xi )}}{{\d t^{m - 1} }}} \right]_{t = 0} .
\]
Employing Cauchy's integral formula alongside part (ii) of Condition \ref{cond}, we deduce the estimate
\begin{align*}
\left| {\frac{{\d^N \Psi _k (t,\xi )}}{{\d t^N }}} \right| & = \left| {\frac{{N!}}{{2\pi \im }}\oint_{\left| {w - t} \right| =\mathfrak{a}_n d} {\frac{{\Psi _k (w,\xi )}}{{(w - t)^{N + 1} }}\d w} } \right| \\ & \le \frac{c}{{2\pi }}\frac{1}{\mathfrak{a}_n^{N+1}}\frac{1}{{d^{N + 1} }}\frac{{N!}}{{1 + \left| \xi  \right|^{1 + \rho } }}\oint_{\left|w - t\right| =  \mathfrak{a}_nd} {\e^{\sigma '{\mathop{\rm Re}\nolimits} (w)} |\d w|}  \le c\e^{\sigma '(t+ \mathfrak{a}_nd)} \frac{1}{\mathfrak{a}_n^N}\frac{1}{{d^N }}\frac{{N!}}{{1 + \left| \xi  \right|^{1 + \rho } }}
\end{align*}
for any $t\ge 0$ and $\xi\in \Delta$. Consequently, we can derive the bounds
\[
\left| {\psi _{k,m} (\xi )} \right|  = \left| {\left[ {\frac{{\d^{m - 1} \Psi _k (t,\xi )}}{{\d t^{m - 1} }}} \right]_{t = 0} } \right|  \le c\e^{\sigma '\mathfrak{a}_n d} \frac{1}{\mathfrak{a}_n^{m-1}}\frac{1}{{d^{m - 1} }}\frac{{(m - 1)!}}{{1 + \left| \xi  \right|^{1 + \rho } }}
\]
and
\[
\left| {\int_0^{ + \infty } {\e^{ - ut} \frac{{\d^N \Psi _k (t,\xi )}}{{\d t^N }}\d t} } \right| \le \int_0^{ + \infty } {\e^{ - \sigma t} \left| {\frac{{\d^N \Psi _k (t,\xi )}}{{\d t^N }}} \right|\d t}  \le \frac{c\e^{\sigma '\mathfrak{a}_n d} }{\sigma  - \sigma '}\frac{1}{\mathfrak{a}_n^N}\frac{1}{{d^N }}\frac{{N!}}{{1 + \left| \xi  \right|^{1 + \rho } }}
\]
provided $\Re(u)\ge \sigma>\sigma'$.
\end{proof}

\begin{lemma}\label{lemma2} Let $0<d<d'$. If Condition \ref{cond} is satisfied, then the following inequalities hold for each $0\le k\le n-2$, positive integer $m$, and any non-negative integer $p$:
\[
\left| {\frac{{\d^p \psi _{k,m} (\xi )}}{{\d\xi ^p }}} \right| \le c_2 \frac{1}{\mathfrak{a}_n^{m-1}}\frac{1}{{d^{m + p - 1} }}\frac{{(m - 1)!p!}}{{1 + \left| \xi  \right|^{1 + \rho } }}
\]
provided $\xi \in \Gamma_j(d)$. Here, $c_2$ is a suitable constant that is independent of $u$, $\xi$, $k$, $m$ and $p$.
\end{lemma}

\begin{proof} We shall show that Lemma \ref{lemma2} holds with $c_2  = c_1 (1 + (1 +  d)^{1 + \rho } )$. Suppose that $\xi \in \Gamma_j(d)$ and let $p$ be a non-negative integer. By employing Cauchy's formula and using Lemma \ref{lemma1}, we deduce the following inequality:
\begin{equation}\label{ineq}
\left| {\frac{{\d^p \psi _{k,m} (\xi )}}{{\d\xi ^p }}} \right| = \left| {\frac{{p!}}{{2\pi \im}}\oint_{\left| {t - \xi } \right| =  d} {\frac{{\psi _{k,m} (t)}}{{(t - \xi )^{p + 1} }}\d t} } \right| \le c_1  \frac{1}{\mathfrak{a}_n^{m-1}}\frac{1}{{d^{m + p} }}\frac{{(m - 1)!p!}}{{2\pi }}\oint_{\left| {t - \xi } \right| = d} {\frac{{\left| {\d t} \right|}}{{1 + \left| t \right|^{1 + \rho } }}} .
\end{equation}
Now, when $\left| \xi  \right| \le 1 + d$,
\[
1 + \left| t \right|^{1 + \rho }  \ge 1 \ge \frac{1}{1 + (1 +  d)^{1 + \rho }}(1 + \left| \xi  \right|^{1 + \rho } )
\]
and when $\left| \xi  \right| \ge 1 + d$,
\[
1 + \left| t \right|^{1 + \rho }  \ge 1 + \left| \left| \xi  \right| - d\right|^{1 + \rho }  \ge 1 + \frac{1}{(1 + d)^{1 + \rho }}\left| \xi  \right|^{1 + \rho }  \ge \frac{1}{1 + (1 + d)^{1 + \rho }}(1 + \left| \xi  \right|^{1 + \rho } ).
\]
Thus, from \eqref{ineq}, we can infer that
\[
\left| {\frac{{\d^p \psi _{k,m} (\xi )}}{{\d\xi ^p }}} \right| \le c_1 (1 + (1 + d)^{1 + \rho } )\frac{1}{\mathfrak{a}_n^{m-1}} \frac{1}{{d^{m + p - 1} }}\frac{{(m - 1)!p!}}{{1 + \left| \xi  \right|^{1 + \rho } }}
\]
for any $\xi \in \Gamma_j(d)$.
\end{proof}

\begin{lemma}\label{lemma3} Suppose $0<d<d'$. If Condition \ref{cond} is met, then the following inequalities hold for each $0\le j\le n-1$, for every positive integer $p$ where $p \le n$, as well as for any non-negative integer $r$:
\[
\left| {\frac{{\d^r }}{{\d\xi ^r }}\B_p (X'_j (\xi ),X''_j (\xi ), \ldots ,X_j^{(p)} (\xi ))} \right| \le c_3 \frac{1}{{d^{p + r - 1} }}\frac{(p + r - 1)!}{1 + \left| \xi  \right|^{1 + \rho } }
\]
provided $\xi \in \Gamma_j(d)$. Here, $c_3$ is an appropriate constant that is independent of $u$, $\xi$, $j$, $p$ and $r$.
\end{lemma}

\begin{proof} We will use induction on $p$ to prove that for any positive integer $p$ and non-negative integer $r$, the inequalities
\begin{equation}\label{eq16}
\left| {\frac{{\d^r }}{{\d\xi ^r }}\B_p (X'_j (\xi ),X''_j (\xi ), \ldots ,X_j^{(p)} (\xi ))} \right| \le \frac{p}{d}\binom{p + c_2 d - 1}{p} \frac{1}{{d^{p + r - 1} }}\frac{{(p + r - 1)!}}{{1 + \left| \xi  \right|^{1 + \rho } }}
\end{equation}
hold whenever $\xi \in \Gamma_j(d)$. The statement of the lemma follows by taking $$c_3=\max_{1\le p\le n}\frac{p}{d}\binom{p + c_2 d - 1}{p}.$$ From the definition \eqref{eq18} and Lemma \ref{lemma2}, we can assert that 
\begin{equation}\label{eq10}
\left| {\frac{{\d^r }}{{\d\xi ^r }}\B_1 (X'_j (\xi ))} \right| = 
\left| X_j^{(r+1)} (\xi ) \right| \le \frac{1}{n}\sum\limits_{k = 0}^{n - 2} {\left| {\frac{{\d^r \psi _{k,1} (\xi )}}{{\d\xi ^r }}} \right|}  \le c_2 \frac{1}{{d^r }}\frac{{r!}}{{1 + \left| \xi  \right|^{1 + \rho } }},
\end{equation}
for $r\ge 0$ and $\xi \in \Gamma_j(d)$. Assume that the inequality \eqref{eq16} holds for all $r\ge 0$ and for $\B_1,\B_2,\ldots,\B_p$ with $p\ge 1$. Differentiating \eqref{Brec2} $r$ times gives
\begin{multline*}
\frac{{\d^r }}{{\d\xi ^r }}\B_{p + 1} (X'_j (\xi ),X''_j (\xi ), \ldots ,X_j^{(p + 1)} (\xi )) =  X_j^{(p + r + 1)} (\xi ) \\ + \sum\limits_{q = 0}^{p - 1} \binom{p}{q}\sum\limits_{s = 0}^r \binom{r}{s}\frac{{\d^s \B_{p - q} (X'_j (\xi ),X''_j (\xi ), \ldots ,X_j^{(p - q)} (\xi ))}}{{\d\xi ^s }}X_j^{(q + r - s + 1)} (\xi )  .
\end{multline*}
Therefore, using the induction hypothesis and Lemma \ref{lemma6},
\begin{align*}
& \left| {\frac{{\d^r }}{{\d\xi ^r }}\B_{p + 1} (X'_j (\xi ),X''_j (\xi ), \ldots ,X_j^{(p + 1)} (\xi ))} \right| \\ & \le c_2 \frac{1}{{d^{p+ r} }}\frac{{(p + r)!}}{{1 + \left| \xi  \right|^{1 + \rho } }} + c_2 \frac{1}{{d^{p + r - 1} }}\frac{1}{{1 + \left| \xi  \right|^{1 + \rho } }}\sum\limits_{q = 0}^{p - 1} \binom{p}{q}\frac{p-q}{d}\binom{p-q + c_2 d - 1}{p-q} \\ & \;\quad
 \times \sum\limits_{s = 0}^r \binom{r}{s}(p - q + s - 1)!(q + r - s)!
\\ &  = c_2 \frac{1}{{d^{p + r} }}\frac{{(p + r)!}}{{1 + \left| \xi  \right|^{1 + \rho } }} + c_2 \frac{1}{{d^{p + r } }}\frac{{(p + r)!}}{{1 + \left| \xi  \right|^{1 + \rho } }}\sum\limits_{q = 0}^{p - 1} { \binom{p-q + c_2 d - 1}{p-q}} 
\\ & = \frac{p+1}{d}\binom{p + c_2 d }{p+1} \frac{1}{{d^{p+ r} }}\frac{{(p + r)!}}{{1 + \left| \xi  \right|^{1 + \rho } }}
\end{align*}
which completes the induction step and the proof.
\end{proof}

\begin{lemma}\label{lemma4} Let $0<d<d'$. If Condition \ref{cond} is satisfied, then the following inequalities hold for each $0\le j\le n-1$, for any positive integers $m$ and $p$, $p\le n$, and any non-negative integer $r$:
\[
\left| {\frac{{\d^r }}{{\d\xi ^r }}\left( \B_p (X'_j (\xi ),X''_j (\xi ), \ldots ,X_j^{(p)} (\xi ))\psi _{k,m} (\xi ) \right)} \right| \le  c_2 c_3 \frac{1}{\mathfrak{a}_n^{m-1}}\frac{1}{{d^{m + p + r-1} }}\frac{{(m - 1)!(p + r )!}}{{1 + \left| \xi  \right|^{1 + \rho } }}
\]
provided that $\xi \in \Gamma_j(d)$. Here, $c_2$ and $c_3$ are the constants from Lemmas \ref{lemma2} and \ref{lemma3}, respectively.
\end{lemma}

\begin{proof} This is a direct consequence of Lemmas \ref{lemma6}, \ref{lemma2} and \ref{lemma3}:
\begin{align*}
&\left| {\frac{{\d^r }}{{\d\xi ^r }}\left( {\B_p (X'_j (\xi ),X''_j (\xi ), \ldots ,X_j^{(p)} (\xi ))\psi _{k,m} (\xi )} \right)} \right|
\\ & \le \sum\limits_{q = 0}^r {\binom{r}{q}\left| {\frac{{\d^q }}{{\d\xi ^q }}\B_p (X'_j (\xi ),X''_j (\xi ), \ldots ,X_j^{(p)} (\xi ))} \right|\left| {\frac{{\d^{r - q} \psi _{k,m} (\xi )}}{{\d\xi ^{r - q} }}} \right|} 
\\ &  \le c_2 c_3  \frac{1}{\mathfrak{a}_n^{m-1}}\frac{1}{{d^{m+p + r-1} }}\frac{{(m - 1)!}}{{1 + \left| \xi  \right|^{1 + \rho } }}\sum\limits_{q = 0}^r {\binom{r}{q}(p + q  - 1)!(r - q)!} 
\\ &  = c_2 c_3 \frac{1}{\mathfrak{a}_n^{m-1}}\frac{1}{{d^{ m +p+ r-1} }}\frac{{(m - 1)!(p + r )!}}{{1 + \left| \xi  \right|^{1 + \rho } }}\frac{1}{{p }}
 \le c_2 c_3 \frac{1}{\mathfrak{a}_n^{m-1}}\frac{1}{{d^{m+p + r-1} }}\frac{(m - 1)!(p + r)!}{1 + \left| \xi  \right|^{1 + \rho } }.
\end{align*}
\end{proof}

\begin{lemma}\label{lemma5} Given that $d > 0$ and $\rho>0$, the following inequalities are valid for each $0\le j\le n-1$:
\[
\int_{\mathscr{P}_j (\xi )} {\frac{{\left| \d t \right|}}{{1 + \left| t \right|^{1 + \rho } }}}  \le \frac{{2(1 + \rho )}}{\rho }\frac{1}{{(\max (1,\Re (\xi \e^{ - \im \theta _j } )))^\rho  }},
\]
provided that $\xi \in \Gamma_j(d)$.
\end{lemma}

\begin{proof} Let us assume that $\xi \in \Gamma_j(d)$. Then we have
\begin{align*}
\int_{\mathscr{P}_j (\xi )} {\frac{{\left| {\d t} \right|}}{{1 + \left| t \right|^{1 + \rho } }}}  = \int_0^{ + \infty } {\frac{{\left| {\d s} \right|}}{{1 + \left| {\xi  + s\e^{\im\theta _j } } \right|^{1 + \rho } }}} & \le \int_0^{ + \infty } {\frac{{\left| {\d s} \right|}}{{1 + \left| {\Re (\xi \e^{ - \im\theta _j } ) + s} \right|^{1 + \rho } }}}  \\ & = \int_{\Re (\xi \e^{ - \im\theta _j } )}^{ + \infty } {\frac{{\left| {\d s} \right|}}{{1 + \left| s \right|^{1 + \rho } }}} .
\end{align*}
Now, when $\Re (\xi \e^{ - \im\theta _j } ) > 1$,
\[
\int_{\Re (\xi \e^{ - \im\theta _j } )}^{ + \infty } {\frac{{\left| {\d s} \right|}}{{1 + \left| s \right|^{1 + \rho } }}}  \le \int_{\Re (\xi \e^{ - \im\theta _j } )}^{ + \infty } {\frac{{\d s}}{{s^{1 + \rho } }}}  = \frac{1}{\rho }\frac{1}{{(\Re (\xi \e^{ - \im\theta _j } ))^\rho  }} \le \frac{{2(1 + \rho )}}{\rho }\frac{1}{{(\Re (\xi \e^{ - \im\theta _j } ))^\rho  }}.
\]
Similarly, when $\Re (\xi \e^{ - \im\theta_j } ) \le 1$,
\begin{align*}
\int_{\Re (\xi \e^{ - \im \theta _j } )}^{ + \infty } {\frac{{\left| {\d s} \right|}}{{1 + \left| s \right|^{1 + \rho } }}}  & \le \int_{ - \infty }^{ + \infty } {\frac{{\left| {\d s} \right|}}{{1 + \left| s \right|^{1 + \rho } }}} = 2\int_0^{ + \infty } {\frac{{\d s}}{{1 + s^{1 + \rho } }}} \\ & = 2\Gamma \left( {\frac{\rho }{{1 + \rho }}} \right)\Gamma \left( {\frac{{2 + \rho }}{{1 + \rho }}} \right) \le \frac{{2(1 + \rho )}}{\rho }.
\end{align*}
\end{proof}

\begin{proof}[Proof of Proposition \ref{prop1}] We begin by proving that for $0\le j\le n-1$, $m,s\ge 1$, and $\xi \in \Gamma_j(d)$, the inequalities
\begin{align} &
\left| {\mathsf{A}_{j,m} (\xi )} \right| \le  C_m  \frac{1}{{\mathfrak{a}_n^{m - 1} }}\frac{1}{{d^m }}\frac{{(m + 1)!}}{{(\max (1,\Re (\xi \e^{ - \im \theta _j } )))^{\rho } }} \label{eq75}
\\ & \left| {\frac{{\d^s \mathsf{A}_{j,m} (\xi )}}{{\d\xi ^s }}} \right| \le  C_m  \frac{1}{{\mathfrak{a}_n^{m - 1} }}\frac{1}{{d^{m + s - 1} }}\frac{{(m + s )!}}{{1 + \left| \xi  \right|^{1 + \rho } }} \label{eq76}
\end{align}
hold, where
\[
C_1  =\frac{1 + \rho }{2\rho }\left( c_3 (1 + c_2 )(n - 1) + \frac{2 + 2c_1  + c_2 }{\mathfrak{a}_n } \right)
\]
and
\[
C_m  = C_{m - 1} \left( 1 + \frac{c_4}{{m + 1}} \right) \qquad (m\ge 2).
\]
The constant $c_4$ is given by
\[
c_4 =(2c_2  + 1)d + \frac{{2(1 + \rho )}}{\rho }(c_3 (d2^n  + d + 1) + 4c_1  + 2(2c_1  + 3c_2 )dn\mathfrak{a}_n  + 4c_2 c_3 (d2^{n + 1}  + 1)n\mathfrak{a}_n ),
\]
and is independent of $u$, $\xi$, $j$, $m$ and $s$. We observe that the sequence $C_m$ is strictly increasing, a property we will use consistently throughout the proof. We proceed via induction on $m$. By definition,
\begin{gather}\label{eq44}
\begin{split}
\mathsf{A}_{j,1} (\xi ) =  \frac{{n - 1}}{2}\e^{-2\pi \im j/n} \int_{\mathscr{P}_j (\xi )} {\B_2 (X'_j (t),X''_j (t))\d t} & - \frac{1}{n}\sum\limits_{k = 0}^{n - 2} \e^{2\pi \im jk/n} k\int_{\mathscr{P}_j (\xi )} {\B_1(X'_j (t))\psi _{k,1} (t)\d t} 
\\ & - \frac{1}{n}\sum\limits_{k = 0}^{n - 2} \e^{2\pi \im j(k + 1)/n}  \int_{\mathscr{P}_j (\xi )} \psi _{k,2} (t)\d t.
\end{split}
\end{gather}
Employing Lemmas \ref{lemma1}, \ref{lemma3}, \ref{lemma4} and \ref{lemma5}, together with the definition of $C_1$, we deduce
\begin{align*}
\left| \mathsf{A}_{j,1} (\xi ) \right|  &  \le (n - 1)\frac{{1 + \rho }}{\rho }c_3 \frac{1}{d}\frac{1}{(\max (1,\Re (\xi \e^{ - \im \theta _j } )))^\rho} \\ & \;\quad+ \frac{{1 + \rho }}{\rho }c_2 c_3 \frac{1}{d}\frac{1}{(\max (1,\Re (\xi \e^{ - \im \theta _j } )))^\rho}\frac{{(n - 1)(n - 2)}}{n} \\ & \;\quad+ c_1 \frac{{2(1 + \rho )}}{\rho }\frac{1}{{\mathfrak{a}_n }}\frac{1}{d}\frac{1}{(\max (1,\Re (\xi \e^{ - \im \theta _j } )))^\rho}\frac{{n - 1}}{n}
\\ &  \le \frac{{1 + \rho }}{{2\rho }}\left( {c_3 (1 + c_2 )(n - 1) + \frac{{2c_1 }}{{\mathfrak{a}_n }}} \right)\frac{1}{d}\frac{2!}{(\max (1,\Re (\xi \e^{ - \im \theta _j } )))^\rho  } \\ &  \le C_1 \frac{1}{d}\frac{2!}{(\max (1,\Re (\xi \e^{ - \im \theta _j } )))^\rho  }.
\end{align*}
Now let $s$ be an arbitrary positive integer. Differentiating \eqref{eq44} $s$ times yields
\begin{align*}
\frac{{\d^s \mathsf{A}_{j,1} (\xi )}}{{\d\xi ^s }} =  - \frac{{n - 1}}{2}\e^{ - 2\pi \im j/n} \frac{{\d^{s - 1} }}{{\d\xi ^{s - 1} }}\B_2 (X'_j (\xi ),X''_j (\xi )) & + \frac{1}{n}\sum\limits_{k = 0}^{n - 2} \e^{2\pi \im jk/n} k\frac{{\d^{s - 1} }}{{\d\xi ^{s - 1} }}\left( \B_1(X'_j (\xi ))\psi _{k,1} (\xi ) \right) 
\\ & + \frac{1}{n}\sum\limits_{k = 0}^{n - 2} {\e^{2\pi \im j(k + 1)/n} \frac{{\d^{s - 1} \psi _{k,2} (\xi )}}{{\d\xi ^{s - 1} }}} .
\end{align*}
By an application of Lemmas \ref{lemma2}, \ref{lemma3}, \ref{lemma4}, and the definition of $C_1$, we obtain
\begin{align*}
\left| \frac{\d^s \mathsf{A}_{j,1}(\xi)}{\d\xi ^s } \right| & \le \frac{{n - 1}}{2}c_3 \frac{1}{{d^s }}\frac{{s!}}{{1 + \left| \xi  \right|^{1 + \rho } }} + c_2 c_3 \frac{1}{{d^s }}\frac{{s!}}{{1 + \left| \xi  \right|^{1 + \rho } }}\frac{{(n - 1)(n - 2)}}{{2n}} + c_2 \frac{1}{\mathfrak{a}_n}\frac{1}{{d^s }}\frac{{(s - 1)!}}{{1 + \left| \xi  \right|^{1 + \rho } }}\frac{{n - 1}}{n}
\\ &  \le \left( \frac{{n - 1}}{{2(s + 1)}}c_3 (1 + c_2 ) + c_2 \frac{1}{s(s + 1)}\frac{1}{\mathfrak{a}_n} \right)\frac{1}{{d^s }}\frac{{(1 + s)!}}{{1 + \left| \xi  \right|^{1 + \rho } }} \\ & \le \frac{1}{2}\left( {\frac{1}{2}c_3 (1 + c_2 )(n - 1) + c_2\frac{1}{\mathfrak{a}_n} } \right)\frac{1}{{d^s }}\frac{{(1 + s)!}}{{1 + \left| \xi  \right|^{1 + \rho } }} \le C_1 \frac{1}{{d^s }}\frac{{(1 + s)!}}{{1 + \left| \xi  \right|^{1 + \rho } }}.
\end{align*}

Assume that \eqref{eq75} and \eqref{eq76} hold for all $s\ge 1$ and for $\mathsf{A}_{j,1}(\xi),\mathsf{A}_{j,2}(\xi),\ldots,\mathsf{A}_{j,m-1}(\xi)$ with $m\ge 2$. In the forthcoming steps, we shall make use of \eqref{eq75} and \eqref{eq76} even when $m=0$. This may be done by defining 
\[
C_0  = \frac{1}{{\mathfrak{a}_n }}<C_1,
\]while excluding the presence of the factor $(\max (1,\Re (\xi \e^{ - \im \theta _j } )))^{-\rho }$. By isolating the terms associated with $r=p$, $r=0$, $p=0$, and then performing some re-grouping, the recurrence \eqref{eq21} can be reformulated as follows:
\begin{gather}\label{largerec}
\begin{split}
& \mathsf{A}_{j,m} (\xi ) =  - \frac{1}{n}\sum\limits_{p = 2}^{\min (n,m + 1)} {\binom{n}{p}\e^{2\pi \im j(n - p + 1)/n} \frac{{\d^{p - 1} \mathsf{A}_{j,m - p + 1} (\xi )}}{{\d \xi^{p - 1} }}} 
\\ & + \frac{1}{n}\sum\limits_{p = 2}^{\min (n,m + 1)} {\binom{n}{p}\e^{2\pi \im j(n - p + 1)/n}  { \int_{\mathscr{P}_j (\xi )} {\B_{p} \mathsf{A}_{j,m - p + 1} (t)\d t} } } 
\\ & + \frac{1}{n}\sum\limits_{p = 2}^{\min (n,m + 1)} {\binom{n}{p}\e^{2\pi \im j(n - p + 1)/n} \sum\limits_{r = 1}^{p - 1} {\binom{p}{r}\int_{\mathscr{P}_j (\xi )} {\B_{p - r} \frac{{\d^r \mathsf{A}_{j,m - p + 1} (t)}}{{\d t^r }}\d t} } } 
\\ & - \frac{1}{n}\sum\limits_{k = 0}^{n - 2} {\sum\limits_{q = 0}^{m-1} {  {\e^{2\pi \im j(k + 1)/n} \int_{\mathscr{P}_j (\xi )} {\psi _{k,m - q + 1} (t)\mathsf{A}_{j,q} (t)\d t} } } } 
\\ & - \frac{1}{n}\sum\limits_{k = 0}^{n - 2} {\sum\limits_{q = 1}^m {\sum\limits_{p = 1}^{\min (k,q)} {\binom{k}{p}\e^{2\pi \im j(k - p + 1)/n} \int_{\mathscr{P}_j (\xi )} {\psi _{k,m - q + 1} (t)\frac{{\d^p \mathsf{A}_{j,q - p} (t)}}{{\d t^p }}\d t} } } } 
\\ & - \frac{1}{n}\sum\limits_{k = 0}^{n - 2} {\sum\limits_{q = 1}^m {\sum\limits_{p = 1}^{\min (k,q)} {\binom{k}{p}\e^{2\pi \im j(k - p + 1)/n}  { \int_{\mathscr{P}_j (\xi )} {\B_{p} \psi _{k,m - q + 1} (t)\mathsf{A}_{j,q - p} (t)\d t} } } } } 
\\ & - \frac{1}{n}\sum\limits_{k = 0}^{n - 2} {\sum\limits_{q = 2}^m {\sum\limits_{p = 2}^{\min (k,q)} {\binom{k}{p} \e^{2\pi \im j(k - p + 1)/n} \sum\limits_{r = 1}^{p - 1} {\binom{p}{r}\int_{\mathscr{P}_j (\xi )} {\B_{p - r} \psi _{k,m - q + 1} (t)\frac{{\d^r \mathsf{A}_{j,q - p} (t)}}{{\d t^r }}\d t} } } } }. 
\end{split}
\end{gather}
We shall estimate the absolute value of each line on the right-hand side of the equality \eqref{largerec} separately. Using the induction hypothesis along with the definition of $\mathfrak{a}_n$, we can establish an upper bound for the absolute value of the first line as follows:
\[
  C_{m - 1} d\frac{1}{{\mathfrak{a}_n^{m - 1} }}\frac{1}{{d^m }}\frac{{m!}}{{1 + \left| \xi  \right|^{1 + \rho } }}  \le  C_{m - 1}  \frac{d}{{m  +1}}\frac{1}{{\mathfrak{a}_n^{m - 1} }}\frac{1}{{d^m }}\frac{{(m  +1)!}}{{(\max (1,\Re (\xi \e^{ - \im \theta _j } )))^{\rho }}}.
\]
Similarly, employing the induction hypothesis along with Lemma \ref{lemma5}, and using \eqref{ineq1}, as well as the definition of $\mathfrak{a}_n$, we can estimate the absolute value of the second line in the following manner:
\begin{align*}
&  C_{m - 1}  \frac{{2(1 + \rho )}}{\rho }c_3\frac{1}{{\mathfrak{a}_n^{m - 1} }}\frac{1}{{d^m }}\frac{1}{(\max (1,\Re (\xi \e^{ - \im \theta _j } )))^{\rho }}\frac{1}{{n\mathfrak{a}_n }}\sum\limits_{p = 2}^{\min (n,m + 1)} {\binom{n}{p}\mathfrak{a}_n^p (p - 1)!(m - p + 2)!} \\
&  \le  C_{m - 1}  \frac{{2(1 + \rho )}}{\rho }\frac{c_3}{m+1}\frac{1}{{\mathfrak{a}_n^{m - 1} }}\frac{1}{{d^m }}\frac{{(m + 1)!}}{(\max (1,\Re (\xi \e^{ - \im \theta _j } )))^{\rho }}.
\end{align*}
The absolute value of the third line is estimated similarly, but by employing inequality \eqref{ineq2} instead of \eqref{ineq1}:
\begin{align*}
& C_{m - 1} \frac{{2(1 + \rho )}}{\rho }c_3 d\frac{1}{{\mathfrak{a}_n^{m - 1} }}\frac{1}{{d^m }}\frac{1}{(\max (1,\Re (\xi \e^{ - \im \theta _j } )))^{\rho }}\frac{1}{{n\mathfrak{a}_n }}\sum\limits_{p = 2}^{\min (n,m + 1)} \binom{n}{p}\mathfrak{a}_n^p \\ & \;\quad \times \sum\limits_{r = 1}^{p - 1} \binom{p}{r}(p - r - 1)!(m - p + r + 1)! 
\\ & \le C_{m - 1} \frac{{2(1 + \rho )}}{\rho }c_3 d\frac{1}{{\mathfrak{a}_n^{m - 1} }}\frac{1}{{d^m }}\frac{{m!}}{(\max (1,\Re (\xi \e^{ - \im \theta _j } )))^{\rho }}\frac{1}{{n\mathfrak{a}_n }}\sum\limits_{p = 2}^{\min (n,m + 1)} \binom{n}{p}\mathfrak{a}_n^p 2^p 
\\ & \le C_{m - 1} \frac{{2(1 + \rho )}}{\rho }\frac{{c_3 d 2^n }}{{m + 1}}\frac{1}{{\mathfrak{a}_n^{m - 1} }}\frac{1}{{d^m }}\frac{{(m + 1)!}}{(\max (1,\Re (\xi \e^{ - \im \theta _j } )))^{\rho }}.
\end{align*}
We estimate the absolute value of the fourth and fifth lines using the induction hypothesis, Lemma \ref{lemma5}, the definition of $\mathfrak{a}_n$, and inequalities \eqref{ineq3} (for the fourth line) and \eqref{ineq4} (for the fifth line), respectively. This yields the following bounds:
\begin{align*}
&  C_{m - 1} \frac{{2(1 + \rho )}}{\rho }c_1 \frac{1}{{\mathfrak{a}_n^{m - 1} }}\frac{1}{{d^m }}\frac{1}{(\max (1,\Re (\xi \e^{ - \im \theta _j } )))^{\rho }}\frac{{n - 1}}{n}\sum\limits_{q = 0}^{m - 1} {(m - q)!(q + 1)!} \\
& \le  C_{m - 1} \frac{{2(1 + \rho )}}{\rho }\frac{{4c_1 }}{{m + 1}}\frac{1}{{\mathfrak{a}_n^{m - 1} }}\frac{1}{{d^m }}\frac{{(m + 1)!}}{(\max (1,\Re (\xi \e^{ - \im \theta _j } )))^{\rho }}
\end{align*}
and
\begin{align*}
& C_{m - 1} \frac{{2(1 + \rho )}}{\rho }c_1d \frac{1}{{\mathfrak{a}_n^{m - 1} }}\frac{1}{{d^{m} }}\frac{1}{(\max (1,\Re (\xi \e^{ - \im \theta _j } )))^{\rho }}\frac{1}{n}\sum\limits_{k = 0}^{n - 2} {\sum\limits_{q = 1}^m {(m - q)!q!\sum\limits_{p = 1}^{\min (k,q)} {\binom{k}{p}\mathfrak{a}_n^p } } } 
\\ & \le C_{m - 1} \frac{{2(1 + \rho )}}{\rho }c_1d \frac{1}{{\mathfrak{a}_n^{m - 1} }}\frac{1}{{d^{m} }}\frac{1}{(\max (1,\Re (\xi \e^{ - \im \theta _j } )))^{\rho }}\sum\limits_{q = 1}^m {(m - q)!q!\sum\limits_{p = 1}^n {\binom{n}{p}\mathfrak{a}_n^p } } 
\\ & \le C_{m - 1} \frac{{2(1 + \rho )}}{\rho }\frac{{4c_1d n\mathfrak{a}_n }}{{m + 1}}\frac{1}{{\mathfrak{a}_n^{m - 1} }}\frac{1}{{d^{m} }}\frac{{(m + 1)!}}{(\max (1,\Re (\xi \e^{ - \im \theta _j } )))^{\rho }}.
\end{align*}
Finally, to bound the absolute value of the sixth and seventh lines, we employ the induction hypothesis, Lemma \ref{lemma5}, the definition of $\mathfrak{a}_n$, and inequalities \eqref{ineq1} and \eqref{ineq4}. This results in the following bounds:
\begin{align*}
&  C_{m - 1} \frac{{2(1 + \rho )}}{\rho }c_2 c_3 \frac{1}{{\mathfrak{a}_n^{m - 1} }}\frac{1}{{d^m }}\frac{1}{(\max (1,\Re (\xi \e^{ - \im \theta _j } )))^{\rho }}\frac{1}{n}\sum\limits_{k = 0}^{n - 2} {\sum\limits_{q = 1}^m {(m - q)!\sum\limits_{p = 1}^{\min (k,q)} {\binom{k}{p}\mathfrak{a}_n^p p!(q - p + 1)!} } } 
\\ & \le  C_{m - 1} \frac{{2(1 + \rho )}}{\rho }c_2 c_3 \frac{1}{{\mathfrak{a}_n^{m - 1} }}\frac{1}{{d^m }}\frac{1}{(\max (1,\Re (\xi \e^{ - \im \theta _j } )))^{\rho }}\frac{1}{n}\sum\limits_{k = 0}^{n - 2} {\sum\limits_{q = 1}^m {(m - q)!q!\sum\limits_{p = 1}^{\min (k,q)} {\binom{k}{p} \mathfrak{a}_n^p } } } 
\\ & \le  C_{m - 1} \frac{{2(1 + \rho )}}{\rho }\frac{{4c_2 c_3 n\mathfrak{a}_n }}{{m + 1}}\frac{1}{{\mathfrak{a}_n^{m - 1} }}\frac{1}{{d^m }}\frac{(m+1)!}{(\max (1,\Re (\xi \e^{ - \im \theta _j } )))^{\rho }}
\end{align*}
and
\begin{align*}
&  C_{m - 1} \frac{{2(1 + \rho )}}{\rho }c_2 c_3 d\frac{1}{{\mathfrak{a}_n^{m - 1} }}\frac{1}{{d^m }}\frac{1}{(\max (1,\Re (\xi \e^{ - \im \theta _j } )))^{\rho }}\frac{1}{n}\sum\limits_{k = 0}^{n - 2} \sum\limits_{q = 2}^m (m - q)!\sum\limits_{p = 1}^{\min (k,q)} \binom{k}{p}\mathfrak{a}_n^p \\ & \;\quad \times \sum\limits_{r = 1}^{p - 1} \binom{p}{r}(p - r)!(q - p + r)!  
\\ & \le C_{m - 1} \frac{{2(1 + \rho )}}{\rho }c_2 c_3 d\frac{1}{{\mathfrak{a}_n^{m - 1} }}\frac{1}{{d^m }}\frac{1}{(\max (1,\Re (\xi \e^{ - \im \theta _j } )))^{\rho }}\frac{1}{n}\sum\limits_{k = 0}^{n - 2} {\sum\limits_{q = 2}^m {(m - q)!(q - 1)!\sum\limits_{p = 1}^{\min (k,q)} {\binom{k}{p}\mathfrak{a}_n^p 2^p } } } 
\\ & \le C_{m - 1} \frac{{2(1 + \rho )}}{\rho }c_2 c_3 d2^n \frac{1}{{\mathfrak{a}_n^{m - 1} }}\frac{1}{{d^m }}\frac{1}{(\max (1,\Re (\xi \e^{ - \im \theta _j } )))^{\rho }}\sum\limits_{q = 1}^{m - 1} {(m - 1 - q)!q!\sum\limits_{p = 1}^n {\binom{n}{p}\mathfrak{a}_n^p } } 
\\ & \le C_{m - 1} \frac{{2(1 + \rho )}}{\rho }\frac{{4c_2 c_3 d2^n n\mathfrak{a}_n }}{{m(m + 1)}}\frac{1}{{\mathfrak{a}_n^{m - 1} }}\frac{1}{{d^m }}\frac{{(m + 1)!}}{(\max (1,\Re (\xi \e^{ - \im \theta _j } )))^{\rho }}.
\end{align*}
By collecting all these estimates and employing the definition of $C_m$, we conclude the induction step for $\mathsf{A}_{j,m} (\xi )$:
\begin{align*}
\left| \mathsf{A}_{j,m} (\xi ) \right| \le \; & C_{m - 1} \left( {d + \frac{{2(1 + \rho )}}{\rho }(c_3 (d2^n  + 1) + 4c_1 (1 + dn\mathfrak{a}_n ) + 4c_2 c_3 (d2^n  + 1)n\mathfrak{a}_n }) \right)\frac{1}{{m + 1}} \\ & \times \frac{1}{{\mathfrak{a}_n^{m - 1} }}\frac{1}{{d^m }} \frac{{(m + 1)!}}{(\max (1,\Re (\xi \e^{ - \im \theta _j } )))^{\rho }} \leq  C_m \frac{1}{{\mathfrak{a}_n^{m - 1} }}\frac{1}{{d^m }} \frac{{(m + 1)!}}{(\max (1,\Re (\xi \e^{ - \im \theta _j } )))^{\rho }}.
\end{align*}

After differentiating each side of \eqref{largerec} once with respect to $\xi$, it is readily seen that the absolute value of each line, except the first on the right-hand side of the resulting equation, can be estimated in the same manner as before. The only difference is that the factor $2(1 + \rho )\rho^{-1}(\max (1,\Re (\xi \e^{ - \im \theta _j } )))^{-\rho }$ can be substituted with $(1 + \left| \xi  \right|^{1 + \rho } )^{-1}$. In order to estimate the absolute value of the derivative of the first line on the right-hand side of \eqref{largerec}, we use the induction hypothesis alongside the definition of $\mathfrak{a}_n$. This yields the following bound:
\[
 C_{m - 1} \frac{1}{{\mathfrak{a}_n^{m - 1} }}\frac{1}{{d^{m} }}\frac{{(m + 1)!}}{{1 + \left| \xi  \right|^{1 + \rho } }}.
\]
Consequently, invoking the definition of $C_m$,
\begin{align*}
\left| {\frac{{\d\mathsf{A}_{j,m} (\xi )}}{{\d\xi }}} \right| & \le C_{m - 1} \left( {1 + \frac{{c_3 (d2^n  + 1) + 4c_1 (1 + dn\mathfrak{a}_n ) + 4c_2 c_3 (d2^n  + 1)n\mathfrak{a}_n }}{{m + 1}}} \right)\frac{1}{{\mathfrak{a}_n^{m - 1} }}\frac{1}{{d^{m} }}\frac{{(m + 1)!}}{{1 + \left| \xi  \right|^{1 + \rho } }} \\ & \le C_m \frac{1}{{\mathfrak{a}_n^{m - 1} }}\frac{1}{{d^{m} }}\frac{{(m + 1)!}}{{1 + \left| \xi  \right|^{1 + \rho } }}.
\end{align*}
This completes the induction step for the derivative of $\mathsf{A}_{j,m} (\xi )$. We observe that this step in the proof is where the sequence $\mathfrak{a}_n$ naturally emerges in the analysis.

Differentiating both sides of \eqref{largerec} $s\ge 2$ times with respect to $\xi$ yields
\begin{gather}\label{largerec2}
\begin{split}
& \frac{{\d^s\mathsf{A}_{j,m} (\xi )}}{{\d\xi^s }} =  - \frac{1}{n}\sum\limits_{p = 2}^{\min (n,m + 1)} {\binom{n}{p} \e^{2\pi \im j(n - p + 1)/n} \frac{{\d^{p+s - 1} \mathsf{A}_{j,m - p + 1} (\xi )}}{{\d \xi^{p+s - 1} }}} 
\\ & - \frac{1}{n}\sum\limits_{p = 2}^{\min (n,m + 1)} \binom{n}{p}\e^{2\pi \im j(n - p + 1)/n}    \frac{\d^{s-1}}{\d\xi^{s-1}}\left(\B_{p} \mathsf{A}_{j,m - p + 1} (\xi)  \right)
\\ & - \frac{1}{n}\sum\limits_{p = 2}^{\min (n,m + 1)} \binom{n}{p}\e^{2\pi \im j(n - p + 1)/n} \sum\limits_{r = 1}^{p - 1} \binom{p}{r}\frac{\d^{s-1}}{\d\xi^{s-1}}\left(\B_{p - r} \frac{{\d^r \mathsf{A}_{j,m - p + 1} (\xi)}}{{\d\xi^r }}\right)  
\\ & + \frac{1}{n}\sum\limits_{k = 0}^{n - 2} \sum\limits_{q = 0}^{m-1}  \e^{2\pi \im j(k + 1)/n} \frac{\d^{s-1}}{\d\xi^{s-1}}\left( \psi _{k,m - q + 1} (\xi)\mathsf{A}_{j,q} (\xi)\right)
\\ & + \frac{1}{n}\sum\limits_{k = 0}^{n - 2} \sum\limits_{q = 1}^m \sum\limits_{p = 1}^{\min (k,q)} \binom{k}{p}\e^{2\pi \im j(k - p + 1)/n} \frac{\d^{s-1}}{\d\xi^{s-1}}\left(\psi _{k,m - q + 1} (\xi)\frac{{\d^p \mathsf{A}_{j,q - p} (\xi)}}{{\d \xi^p }}\right)
\\ & + \frac{1}{n}\sum\limits_{k = 0}^{n - 2} \sum\limits_{q = 1}^m \sum\limits_{p = 1}^{\min (k,q)} \binom{k}{p} \e^{2\pi \im j(k - p + 1)/n} \frac{\d^{s-1}}{\d\xi^{s-1}}\left( \B_{p} \psi _{k,m - q + 1} (\xi)\mathsf{A}_{j,q - p} (\xi)\right)
\\ & + \frac{1}{n}\sum\limits_{k = 0}^{n - 2} \sum\limits_{q = 2}^m \sum\limits_{p = 2}^{\min (k,q)} \binom{k}{p} \e^{2\pi \im j(k - p + 1)/n} \sum\limits_{r = 1}^{p - 1} \binom{p}{r}\frac{\d^{s-1}}{\d\xi^{s-1}}\left(\B_{p - r} \psi _{k,m - q + 1} (\xi)\frac{{\d^r \mathsf{A}_{j,q - p} (\xi)}}{{\d \xi^r }}\right).
\end{split}
\end{gather}
Similar to the case of $\mathsf{A}_{j,m} (\xi )$, we will separately bound the absolute value of each line on the right-hand side of \eqref{largerec2}. By employing the induction hypothesis in combination with the definition of $\mathfrak{a}_n$, we can derive the following upper bound for the absolute value of the first line:
\[
C_{m - 1} \frac{1}{{\mathfrak{a}_n^{m - 1} }}\frac{1}{{d^{m + s - 1} }}\frac{{(m + s)!}}{{1 + \left| \xi  \right|^{1 + \rho } }} .
\]
To estimate the magnitude of the second line, we observe that by using Lemmas \ref{lemma6} and \ref{lemma3} in conjunction with the induction hypothesis,
\begin{align*}
& \left| {\frac{{\d^{s - 1} }}{{\d\xi ^{s - 1} }}(\B_p \mathsf{A}_{j,m - p + 1} (\xi ))} \right| \le \sum\limits_{h = 0}^{s - 1} {\binom{s-1}{h}\left| {\frac{{\d^h \B_p }}{{\d\xi ^h }}} \right|\left| {\frac{{\d^{s -h - 1} \mathsf{A}_{j,m - p + 1} (\xi )}}{{\d\xi ^{s - h - 1} }}} \right|} 
\\ & \le C_{m - p + 1} c_3 \frac{1}{{\mathfrak{a}_n^{m - p} }}\frac{1}{{d^{m + s - 2} }}\frac{1}{{1 + \left| \xi  \right|^{1 + \rho } }}\sum\limits_{h = 0}^{s - 1} {\binom{s-1}{h}(p + h - 1)!(m - p + s - h)!} 
\\ & = C_{m - p + 1} c_3 \frac{1}{{\mathfrak{a}_n^{m - p} }}\frac{1}{{d^{m + s - 2} }}\frac{{(m + s)!}}{{1 + \left| \xi  \right|^{1 + \rho } }}\frac{1}{{(m+1)\binom{m}{p-1}}} \\ & \le C_{m - p + 1} c_3 \frac{1}{{\mathfrak{a}_n^{m - p} }}\frac{1}{{d^{m + s - 2} }}\frac{{(m + s)!}}{{1 + \left| \xi  \right|^{1 + \rho } }}\frac{1}{{m+1}}.
\end{align*}
Therefore, the absolute value of the second line can be bounded as follows:
\begin{align*}
& C_{m - 1} \frac{c_3 d}{m+1}\frac{1}{{\mathfrak{a}_n^{m - 1} }}\frac{1}{{d^{m + s - 1} }}\frac{{(m + s)!}}{{1 + \left| \xi  \right|^{1 + \rho } }}\frac{1}{{n\mathfrak{a}_n }}\sum\limits_{p = 2}^{\min (n,m + 1)} {\binom{n}{p}\mathfrak{a}_n^p} 
\\ & \le C_{m - 1} \frac{{c_3 d}}{{m + 1}}\frac{1}{{\mathfrak{a}_n^{m - 1} }}\frac{1}{{d^{m + s - 1} }}\frac{{(m + s)!}}{{1 + \left| \xi  \right|^{1 + \rho } }}\frac{1}{{n\mathfrak{a}_n }}\sum\limits_{p = 2}^{n} {\binom{n}{p}\mathfrak{a}_n^p } 
\\ & \le C_{m - 1} \frac{{c_3 d}}{{m + 1}}\frac{1}{{\mathfrak{a}_n^{m - 1} }}\frac{1}{{d^{m + s - 1} }}\frac{{(m + s)!}}{{1 + \left| \xi  \right|^{1 + \rho } }},
\end{align*}
where we have used the definition of $\mathfrak{a}_n$. To estimate the absolute value of the third line, we employ Lemmas \ref{lemma6} and \ref{lemma3} alongside the induction hypothesis, yielding the following bound:
\begin{align*}
& \left| {\frac{{\d^{s - 1} }}{{\d\xi ^{s - 1} }}\left( {\B_{p - r} \frac{{\d^r \mathsf{A}_{j,m - p + 1} (\xi )}}{{\d\xi ^r }}} \right)} \right| \le \sum\limits_{h = 0}^{s - 1} {\binom{s-1}{h}\left| {\frac{{\d^h \B_{p - r} }}{{\d\xi ^h }}} \right|\left| {\frac{{\d^{r + s - h - 1} \mathsf{A}_{j,m - p + 1} (\xi )}}{{\d\xi ^{r + s - h - 1} }}} \right|} 
\\ & \le C_{m - p + 1} c_3 \frac{1}{{\mathfrak{a}_n^{m - p} }}\frac{1}{{d^{m + s - 2} }}\frac{1}{{1 + \left| \xi  \right|^{1 + \rho } }}\sum\limits_{h = 0}^{s - 1} \binom{s-1}{h}(p - r + h - 1)! (m - p + r + s - h)!
\\ & = C_{m - p + 1} c_3 \frac{1}{{\mathfrak{a}_n^{m - p} }}\frac{1}{{d^{m + s - 2} }}\frac{{(m + s)!}}{{1 + \left| \xi  \right|^{1 + \rho } }}\frac{1}{{(m + 1)\binom{m}{p - r - 1}}}
\\ & \le C_{m - p + 1} c_3 \frac{1}{{\mathfrak{a}_n^{m - p} }}\frac{1}{{d^{m + s - 2} }}\frac{{(m + s)!}}{{1 + \left| \xi  \right|^{1 + \rho } }}\frac{1}{{m + 1}}.
\end{align*}
Hence, by invoking the definition of $\mathfrak{a}_n$, the absolute value of the third line can be estimated as follows:
\begin{align*}
& C_{m - 1} \frac{{c_3 d}}{{m + 1}}\frac{1}{{\mathfrak{a}_n^{m - 1} }}\frac{1}{{d^{m + s - 1} }}\frac{{(m + s)!}}{{1 + \left| \xi  \right|^{1 + \rho } }}\frac{1}{{n\mathfrak{a}_n }}\sum\limits_{p = 2}^{\min (n,m + 1)} {\binom{n}{p}\mathfrak{a}_n^p \sum\limits_{r = 1}^{p - 1} {\binom{p}{r}} } 
\\ & \le C_{m - 1} \frac{{c_3 d}}{{m + 1}}\frac{1}{{\mathfrak{a}_n^{m - 1} }}\frac{1}{{d^{m + s - 1} }}\frac{{(m + s)!}}{{1 + \left| \xi  \right|^{1 + \rho } }}\frac{1}{{n\mathfrak{a}_n }}\sum\limits_{p = 2}^n {\binom{n}{p}\mathfrak{a}_n^p 2^p } 
\\ & \le C_{m - 1} \frac{{c_3 d2^n }}{{m + 1}}\frac{1}{{\mathfrak{a}_n^{m - 1} }}\frac{1}{{d^{m + s - 1} }}\frac{{(m + s)!}}{{1 + \left| \xi  \right|^{1 + \rho } }}.
\end{align*}
Next, using Lemmas \ref{lemma6} and \ref{lemma2}, along with the induction hypothesis, we obtain
\begin{align*}
& \left| {\frac{{\d^{s - 1} }}{{\d\xi ^{s - 1} }}(\psi _{k,m - q + 1}(\xi) \mathsf{A}_{j,q} (\xi ))} \right| \le \sum\limits_{h = 0}^{s - 1} {\binom{s-1}{h}\left| {\frac{{\d^h \psi _{k,m - q + 1}(\xi) }}{{\d\xi ^h }}} \right|\left| {\frac{{\d^{s - h - 1} \mathsf{A}_{j,q} (\xi )}}{{\d\xi ^{s - h - 1} }}} \right|} 
\\ &  \le C_q c_2 \frac{1}{{a_n^{m - 1} }}\frac{1}{{d^{m + s - 2} }}\frac{{(m - q)!}}{{1 + \left| \xi  \right|^{1 + \rho } }}\sum\limits_{h = 0}^{s - 1} {\binom{s-1}{h}h!(q + s - h - 1)!} 
\\ & = C_q c_2 \frac{1}{{a_n^{m - 1} }}\frac{1}{{d^{m + s - 2} }}\frac{{(m - q)!}}{{1 + \left| \xi  \right|^{1 + \rho } }}\frac{{(q + s)!}}{{q + 1}}.
\end{align*}
Thus, using inequality \eqref{ineq5}, we can bound the absolute value of the fourth line as follows:
\begin{align*}
& C_{m - 1} c_2d \frac{1}{{\mathfrak{a}_n^{m - 1} }}\frac{1}{{d^{m + s - 1} }}\frac{1}{{1 + \left| \xi  \right|^{1 + \rho } }}\frac{1}{n}\sum\limits_{k = 0}^{n - 2} {\sum\limits_{q = 0}^{m - 1} {\frac{{(m - q)!(q + s)!}}{{q + 1}}} } 
\\ & \le C_{m - 1}  c_2d \frac{1}{{\mathfrak{a}_n^{m - 1} }}\frac{1}{{d^{m + s - 1} }}\frac{1}{{1 + \left| \xi  \right|^{1 + \rho } }}\sum\limits_{q = 0}^{m - 1} {\frac{{(m - q)!(q + s)!}}{{q + 1}}} \\ & \le C_{m - 1} \frac{2c_2d}{m+1} \frac{1}{{\mathfrak{a}_n^{m - 1} }}\frac{1}{{d^{m + s - 1} }}\frac{(m+s)!}{{1 + \left| \xi  \right|^{1 + \rho } }}.
\end{align*}
In a similar manner, we obtain
\begin{align*}
& \left| {\frac{{\d^{s - 1} }}{{\d\xi ^{s - 1} }}\left( {\psi _{k,m - q + 1} (\xi )\frac{{\d^p \mathsf{A}_{j,q - p} (\xi )}}{{\d\xi ^p }}} \right)} \right| \le \sum\limits_{h = 0}^{s - 1} {\binom{s-1}{h}\left| {\frac{{\d^h \psi _{k,m - q + 1} (\xi )}}{{\d\xi ^h }}} \right|\left| {\frac{{\d^{p + s - h- 1} \mathsf{A}_{j,q - p} (\xi )}}{{\d\xi ^{p + s - h - 1} }}} \right|} 
\\ & \le C_{q - p} c_2 \frac{1}{{\mathfrak{a}_n^{m - p - 1} }}\frac{1}{{d^{m + s - 2} }}\frac{{(m - q)!}}{{1 + \left| \xi  \right|^{1 + \rho } }}\sum\limits_{h = 0}^{s - 1} {\binom{s-1}{h}h!(q + s - h - 1)!} 
\\ & = C_{q - p} c_2 \frac{1}{{\mathfrak{a}_n^{m - p - 1} }}\frac{1}{{d^{m + s - 2} }}\frac{{(m - q)!}}{{1 + \left| \xi  \right|^{1 + \rho } }}\frac{{(q + s)!}}{{q + 1}}.
\end{align*}
Therefore, applying the definition of $\mathfrak{a}_n$ and inequality \eqref{ineq6}, we may estimate the absolute value of the fifth line as follows:
\begin{align*}
& C_{m - 1} c_2 d\frac{1}{{\mathfrak{a}_n^{m - 1} }}\frac{1}{{d^{m + s - 1} }}\frac{1}{{1 + \left| \xi  \right|^{1 + \rho } }}\frac{1}{n}\sum\limits_{k = 0}^{n - 2} {\sum\limits_{q = 1}^m {\frac{{(m - q)!(q + s)!}}{{q + 1}}\sum\limits_{p = 1}^{\min (k,q)} {\binom{k}{p}\mathfrak{a}_n^p } } } 
\\ & \le C_{m - 1} c_2 d\frac{1}{{\mathfrak{a}_n^{m - 1} }}\frac{1}{{d^{m + s - 1} }}\frac{1}{{1 + \left| \xi  \right|^{1 + \rho } }}\sum\limits_{q = 1}^m {\frac{{(m - q)!(q + s)!}}{{q + 1}}\sum\limits_{p = 1}^n {\binom{n}{p}\mathfrak{a}_n^p } } 
\\ &  \le C_{m - 1} 2c_2 dn\mathfrak{a}_n \frac{1}{{\mathfrak{a}_n^{m - 1} }}\frac{1}{{d^{m + s - 1} }}\frac{1}{{1 + \left| \xi  \right|^{1 + \rho } }}\sum\limits_{q = 1}^m {\frac{{(m - q)!(q + s)!}}{{q + 1}}} 
\\ &  \le C_{m - 1} \frac{{6c_2 dn\mathfrak{a}_n }}{{m + 1}}\frac{1}{{\mathfrak{a}_n^{m - 1} }}\frac{1}{{d^{m + s - 1} }}\frac{{(m + s)!}}{{1 + \left| \xi  \right|^{1 + \rho } }}.
\end{align*}
To estimate the absolute value of the sixth line, we note that by using Lemmas \ref{lemma6} and \ref{lemma4} in conjunction with the induction hypothesis,
\begin{align*}
& \left| {\frac{{\d^{s - 1} }}{{\d\xi ^{s - 1} }}\left( {\B_p \psi _{k,m - q + 1} (\xi )\mathsf{A}_{j,q - p} (\xi )} \right)} \right| \le \sum\limits_{h = 0}^{s - 1} {\binom{s-1}{h}\left| {\frac{{\d^h }}{{\d\xi ^h }}\left( {\B_p \psi _{k,m - q + 1} (\xi )} \right)} \right|\left| {\frac{{\d^{s - h - 1} \mathsf{A}_{j,q - p} (\xi )}}{{\d\xi ^{s-h - 1} }}} \right|} 
\\ & \le C_{q - p} c_2 c_3 \frac{1}{{\mathfrak{a}_n^{m - p - 1} }}\frac{1}{{d^{m + s - 2} }}\frac{{(m - q)!}}{{1 + \left| \xi  \right|^{1 + \rho } }}\sum\limits_{h = 0}^{s - 1} {\binom{s-1}{h}(p + h)!(q - p + s - h - 1)!} 
\\ & = C_{q - p} c_2 c_3 \frac{1}{{\mathfrak{a}_n^{m - p - 1} }}\frac{1}{{d^{m + s - 2} }}\frac{{(m - q)!}}{{1 + \left| \xi  \right|^{1 + \rho } }}\frac{{(q + s)!}}{{(q + 1)\binom{q}{p}}} \\ & \le C_{q - p} c_2 c_3 \frac{1}{{\mathfrak{a}_n^{m - p - 1} }}\frac{1}{{d^{m + s - 2} }}\frac{{(m - q)!}}{{1 + \left| \xi  \right|^{1 + \rho } }}\frac{{(q + s)!}}{{q + 1}}.
\end{align*}
Hence, by using the definition of $\mathfrak{a}_n$ along with inequality \eqref{ineq6}, we can estimate the absolute value of the sixth line in the following manner:
\begin{align*}
& C_{m - 1} c_2 c_3 d\frac{1}{{\mathfrak{a}_n^{m - 1} }}\frac{1}{{d^{m + s - 1} }}\frac{1}{{1 + \left| \xi  \right|^{1 + \rho } }}\frac{1}{n}\sum\limits_{k = 0}^{n - 2} {\sum\limits_{q = 1}^m {\frac{{(m - q)!(q + s)!}}{{q + 1}}\sum\limits_{p = 1}^{\min (k,q)} {\binom{k}{p}\mathfrak{a}_n^p } } } 
\\ & \le C_{m - 1} c_2 c_3 d\frac{1}{{\mathfrak{a}_n^{m - 1} }}\frac{1}{{d^{m + s - 1} }}\frac{1}{{1 + \left| \xi  \right|^{1 + \rho } }} \sum\limits_{q = 1}^m {\frac{{(m - q)!(q + s)!}}{{q + 1}}\sum\limits_{p = 1}^{n} {\binom{n}{p}\mathfrak{a}_n^p } }  
\\ & \le C_{m - 1} 2c_2 c_3 dn\mathfrak{a}_n \frac{1}{{\mathfrak{a}_n^{m - 1} }}\frac{1}{{d^{m + s - 1} }}\frac{1}{{1 + \left| \xi  \right|^{1 + \rho } }}\sum\limits_{q = 1}^m {\frac{{(m - q)!(q + s)!}}{{q + 1}}} \\ & \le C_{m - 1} \frac{6c_2 c_3 dn\mathfrak{a}_n}{m+1} \frac{1}{{\mathfrak{a}_n^{m - 1} }}\frac{1}{{d^{m + s - 1} }}\frac{(m+s)!}{{1 + \left| \xi  \right|^{1 + \rho } }}.
\end{align*}
Finally, following a similar approach as in the previous case, we obtain
\begin{align*}
& \left| {\frac{{\d^{s - 1} }}{{\d\xi ^{s - 1} }}\left( {\B_{p - r} \psi _{k,m - q + 1} (\xi )\frac{{\d^r \mathsf{A}_{j,q - p} (\xi )}}{{\d\xi ^r }}} \right)} \right| \\ & \le \sum\limits_{h = 0}^{s - 1} {\binom{s-q}{h}\left| {\frac{{\d^h }}{{\d\xi ^h }}\left( {\B_{p - r} \psi _{k,m - q + 1} (\xi )} \right)} \right|\left| {\frac{{\d^{r + s - h - 1} \mathsf{A}_{j,q - p} (\xi )}}{{\d\xi ^{r + s - h - 1} }}} \right|} 
\\ & \le C_{q-p} c_2 c_3 \frac{1}{{\mathfrak{a}_n^{m - p - 1} }}\frac{1}{{d^{m + s - 2} }}\frac{{(m - q)!}}{{1 + \left| \xi  \right|^{1 + \rho } }}\sum\limits_{h = 0}^{s - 1} {\binom{s-q}{h}(p - r + h)!(q - p + r + s - h - 1)!} 
\\ & = C_{q-p} c_2 c_3 \frac{1}{{\mathfrak{a}_n^{m - p - 1} }}\frac{1}{{d^{m + s - 2} }}\frac{{(m - q)!}}{{1 + \left| \xi  \right|^{1 + \rho } }}\frac{{(q + s)!}}{{(q + 1)\binom{q}{p-r}}} \\ & \le C_{q-p} c_2 c_3 \frac{1}{{\mathfrak{a}_n^{m - p - 1} }}\frac{1}{{d^{m + s - 2} }}\frac{{(m - q)!}}{{1 + \left| \xi  \right|^{1 + \rho } }}\frac{{(q + s)!}}{{q + 1}} .
\end{align*}
Thus, by applying the definition of $\mathfrak{a}_n$ along with inequality \eqref{ineq6}, we can estimate the absolute value of the seventh line as follows:
\begin{align*}
& C_{m - 1} c_2 c_3 d\frac{1}{{\mathfrak{a}_n^{m - 1} }}\frac{1}{{d^{m + s - 1} }}\frac{1}{{1 + \left| \xi  \right|^{1 + \rho } }}\frac{1}{n}\sum\limits_{k = 0}^{n - 2} {\sum\limits_{q = 2}^m {\frac{{(m - q)!(q + s)!}}{{q + 1}}\sum\limits_{p = 2}^{\min (k,q)} {\binom{k}{p}\mathfrak{a}_n^p \sum\limits_{r = 1}^{p - 1} {\binom{p}{r}} } } } 
\\ & \le C_{m - 1} c_2 c_3 d\frac{1}{{\mathfrak{a}_n^{m - 1} }}\frac{1}{{d^{m + s - 1} }}\frac{1}{{1 + \left| \xi  \right|^{1 + \rho } }}\frac{1}{n}\sum\limits_{k = 0}^{n - 2} {\sum\limits_{q = 2}^m {\frac{{(m - q)!(q + s)!}}{{q + 1}}\sum\limits_{p = 2}^{\min (k,q)} {\binom{k}{p}\mathfrak{a}_n^p 2^p } } }
\\ &  \le C_{m - 1} c_2 c_3 d2^n \frac{1}{{\mathfrak{a}_n^{m - 1} }}\frac{1}{{d^{m + s - 1} }}\frac{1}{{1 + \left| \xi  \right|^{1 + \rho } }}\sum\limits_{q = 1}^m \frac{{(m - q)!(q + s)!}}{{q + 1}}\sum\limits_{p = 2}^n {\binom{n}{p}\mathfrak{a}_n^p }  
\\ & = C_{m - 1} c_2 c_3 d2^n n\mathfrak{a}_n \frac{1}{{\mathfrak{a}_n^{m - 1} }}\frac{1}{{d^{m + s - 1} }}\frac{1}{{1 + \left| \xi  \right|^{1 + \rho } }} \sum\limits_{q = 1}^m \frac{{(m - q)!(q + s)!}}{{q + 1}}
\\ & \le C_{m - 1} \frac{{3c_2 c_3 d2^n n\mathfrak{a}_n }}{{m + 1}}\frac{1}{{\mathfrak{a}_n^{m - 1} }}\frac{1}{{d^{m + s - 1} }}\frac{{(m + s)!}}{{1 + \left| \xi  \right|^{1 + \rho } }}.
\end{align*}
By combining all of these estimates and employing the definition of $C_m$, we conclude the induction
step for the higher derivatives of $\mathsf{A}_{j,m} (\xi )$:
\begin{align*}
& \left|\frac{{\d^s\mathsf{A}_{j,m} (\xi )}}{{\d\xi^s }}\right| \\ & \le C_{m - 1} \left( {1 + \frac{{c_3 d(2^n  + 1) + 2c_2 d + 3(2 + c_3 (2^n  + 2))c_2 dn\mathfrak{a}_n }}{{m + 1}}} \right)\frac{1}{{\mathfrak{a}_n^{m - 1} }}\frac{1}{{d^{m + s - 1} }}\frac{{(m + s)!}}{{1 + \left| \xi  \right|^{1 + \rho } }} \\ & \le C_m \frac{1}{{\mathfrak{a}_n^{m - 1} }}\frac{1}{{d^{m + s - 1} }}\frac{{(m + s)!}}{{1 + \left| \xi  \right|^{1 + \rho } }}.
\end{align*}
This completes the proof of inequalities \eqref{eq75} and \eqref{eq76}.

To finish the proof of Proposition \ref{prop1}, we note that
\begin{align*}
C_m  = C_1 \prod\limits_{r = 3}^{m + 1} {\left( {1 + \frac{{c_4 }}{r}} \right)} & \le C_1 \exp\bigg( c_4 \sum\limits_{r = 3}^{m + 1} \frac{1}{r}\bigg) \\ & \le C_1 \exp (c_4 \log (m + 1)) = C_1 (m + 1)^{c_4 }  \le (m + 1)^L 
\end{align*}
with a suitable positive integer $L$ independent of $u$, $\xi$, $j$, $m$ and $s$. When $m=1$, the empty product and sum are interpreted as $1$ and $0$, respectively. Then
\begin{align*}
& (m + 1)!(m + 1)^L  \le (m + 1)!(m + 2) \cdots (m + L + 1) = (m + L + 1)!,
\\ & (m + s)!(m + 1)^L  \le (m + s)!(m + 1 + s) \cdots (m + L + s) = (m + L + s)!
\end{align*}
for any positive integers $m$ and $s$.
\end{proof}

\section{Gevrey-type estimates for the remainder terms}\label{gevreyremainder}

In this section, we construct the functions $\eta_j(u,\xi)$ specified in Theorem \ref{thm1} in the form of truncated asymptotic expansions, each accompanied by its respective remainder term, possessing a Gevrey-type bound. The precise statement is provided below.

\begin{proposition}\label{prop2}
Let $0<d<d'$. If Condition \ref{cond} holds then the differential equation \eqref{eq3} admits unique solutions $W_j(u,\xi)$, $0\le j\le n-1$, of the form \eqref{eq17} with the following properties. The functions $\eta_j (u,\xi )$ are analytic in $\left\{ u : \Re(u) > \sigma' \right\}  \times  \Gamma_j (d)$, meeting the limit conditions \eqref{eq19}, and for any positive integer $N$, they can be represented in the form
\begin{equation}\label{eq30}
\eta_j (u,\xi )  = \sum\limits_{m = 1}^{N - 1} \frac{\mathsf{A}_{j,m} (\xi )}{u^m}  + R_{j,N} (u,\xi ).
\end{equation}
The remainder terms $R_{j,N} (u,\xi)$ possess the property that, for each $\sigma > \sigma'$, there exists a constant $c_5$, independent of $u$, $\xi$, $j$ and $N$, such that
\begin{equation}\label{eq29}
\left|R_{j,N} (u,\xi )\right| \le c_5\frac{1}{\mathfrak{a}_n^N}\frac{(N+n+L)!}{d^N }\frac{1}{\left| u \right|^N }\sum\limits_{\ell\neq j} \frac{1}{(\max (1,\Re (\xi \e^{ - \im\theta _{j,\ell} } )))^\rho  } ,
\end{equation}
under the conditions $\Re(u)\ge \sigma$ and $\xi \in \Gamma_j(d)$. Here, $L$ is the constant from Proposition \ref{prop1}.
\end{proposition}

To construct the solutions $W_j(u,\xi)$, we shall employ a technique similar to the one developed by Olver for second-order equations \cite[Ch. 10, \S3]{Olver1997}. We begin by defining, for $m\ge 0$ and $0\le j\le n-1$,
\[
\widehat{\mathsf{A}}_{j,m} (\xi ) = \exp ( X_j (\xi ) )\mathsf{A}_{j,m} (\xi ).
\]
Now, for each $0\le j\le n-1$ and $N\ge 1$, we seek solutions to the differential equation \eqref{eq3} in the form
\begin{equation}\label{eq11}
W_{j,N} (u,\xi ) = \exp \Big( {\e^{2\pi \im j/n} u\xi } \Big)\sum\limits_{m = 0}^{N - 1} \frac{\widehat{\mathsf{A}}_{j,m} (\xi )}{u^m }  + \varepsilon _{j,N} (u,\xi ),
\end{equation}
subject to the conditions $\Re(u)>\sigma'$ and $\xi \in \Gamma_j(d)$. Differentiating \eqref{eq11} $k$ times yields
\begin{gather}\label{eq12}
\begin{split}
\frac{{\d^k }}{{\d\xi ^k }}W_{j,N} (u,\xi ) =\; & u^k \exp \Big( {\e^{2\pi \im j/n} u\xi } \Big)\sum\limits_{m = 0}^{N + k - 1} {\bigg( {\sum\limits_{p = \max (0,m - N + 1)}^{\min (k,m)} {\binom{k}{p}\e^{2\pi \im j(k - p)/n} \frac{{\d^p \widehat{\mathsf{A}}_{j,m - p} (\xi )}}{{\d\xi ^p }}} } \bigg)\frac{1}{{u^m }}}  \\ & + \frac{{\d^k \varepsilon _{j,N} (u,\xi )}}{{\d\xi ^k }}.
\end{split}
\end{gather}
With the notation of Lemma \ref{lemma1}, we write for $0\le k\le n-2$, $0\le m\le N+k-1$:
\begin{equation}\label{eq13}
\psi _k (u,\xi ) = \delta _{0,k}  + \sum\limits_{q = 1}^{N + k - m} {\frac{{\psi _{k,q} (\xi )}}{{u^q }}}  + \frac{{\psi _{k,N + k - m + 1} (u,\xi )}}{{u^{N + k - m + 1} }}.
\end{equation}
Multiplying the $m^{\text{th}}$ term in \eqref{eq12} by \eqref{eq13} yields
\begin{gather}\label{eq14}
\begin{split}
& u^{n - k} \left( {\psi _k (u,\xi ) - \delta _{0,k} } \right)\frac{{\d^k }}{{\d\xi ^k }}W_{j,N} (u,\xi )
\\ & =u^n \exp \Big( {\e^{2\pi \im j/n} u\xi } \Big)\sum\limits_{m = 1}^{N + k} {\bigg( {\sum\limits_{q = 0}^{m - 1} {\psi _{k,m - q} (\xi )\sum\limits_{p = \max (0,q - N + 1)}^{\min (k,q)} {\binom{k}{p}\e^{2\pi \im j(k - p)/n} \frac{{\d^p \widehat{\mathsf{A}}_{j,q - p} (\xi )}}{{\d\xi ^p }}} } } \bigg)\frac{1}{{u^m }}} \\ & \quad +u^n\exp \Big( {\e^{2\pi \im j/n} u\xi } \Big)\frac{1}{{u^{N + k+1} }}\sum\limits_{m = 0}^{N + k - 1} \psi _{k,N + k - m + 1} (u,\xi ) \\ & \quad \times \sum\limits_{p = \max (0,m - N + 1)}^{\min (k,m)} \binom{k}{p}\e^{2\pi \im j(k - p)/n} \frac{{\d^p \widehat{\mathsf{A}}_{j,m - p} (\xi )}}{{\d\xi ^p }} 
\\ & \quad + u^n \left( {\psi _k (u,\xi ) - \delta _{0,k} } \right)\frac{{\d^k \varepsilon _{j,N} (u,\xi )}}{{u^k \d\xi ^k }}.
\end{split}
\end{gather}
If the $\mathsf{A}_{j,m} (\xi)$ satisfy the recurrence relation \eqref{eq21}, then the $\widehat{\mathsf{A}}_{j,m} (\xi)$ satisfy the following equality for all $m\ge 1$:
\begin{equation}\label{ahatrec}
\sum\limits_{p = 1}^{\min(n,m)} \binom{n}{p}\e^{2\pi \im j(n - p)/n} \frac{\d^p \widehat{\mathsf{A}}_{j,m - p} (\xi )}{\d\xi ^p }  =  \sum\limits_{k = 0}^{n - 2} \sum\limits_{q = 0}^{m - 1} \psi _{k,m - q} (\xi )\sum\limits_{p = 0}^{\min(k,q)} \binom{k}{p}\e^{2\pi \im j(k - p)/n} \frac{{\d^p \widehat{\mathsf{A}}_{j,q - p} (\xi )}}{{\d\xi ^p }} .
\end{equation}
By use of \eqref{eq14} and \eqref{ahatrec}, we may verify that if \eqref{eq11} satisfy \eqref{eq3}, then the $\varepsilon _{j,N} (u,\xi )$ satisfy the equation
\begin{gather}\label{epsilonode}
\begin{split}
\frac{{\d^n \varepsilon _{j,N} (u,\xi )}}{{\d\xi ^n }} - u^n \varepsilon _{j,N} (u,\xi ) = 
u^{n - 1} \exp \Big( \e^{2\pi \im j/n} u\xi  \Big)&\frac{1}{{u^N }}(\mathcal{P}_{j,N}(u,\xi)+\mathcal{Q}_{j,N}(u,\xi)+\mathcal{R}_{j,N}(u,\xi)) \\ &+ u^{n - 1} \sum\limits_{k = 0}^{n - 2} \psi _{k,1} (u,\xi )\frac{{\d^k \varepsilon _{j,N} (u,\xi )}}{{u^k \d\xi ^k }} ,
\end{split}
\end{gather}
where
\begin{align*}
& \mathcal{P}_{j,N}(u,\xi) = -\sum\limits_{m = 0}^{n - 2} {\bigg( {\sum\limits_{p = m + 2}^{\min (n,N + m + 1)} {\binom{n}{p}\e^{2\pi \im j(n - p)/n} \frac{{\d^p \widehat{\mathsf{A}}_{j,N + m - p + 1} (\xi )}}{{\d\xi ^p }}} } \bigg)\frac{1}{{u^m }}},
\\ & \mathcal{Q}_{j,N}(u,\xi) = \sum\limits_{k =1}^{n - 2} {\sum\limits_{m = 0}^{k - 1} {\bigg( {\sum\limits_{q = 0}^{N + m} {\psi _{k,N + m - q + 1} (\xi )\sum\limits_{p = \max (0,q - N + 1)}^{\min (k,q)} {\binom{k}{p}\e^{2\pi \im j(k - p)/n} \frac{{\d^p \widehat{\mathsf{A}}_{j,q - p} (\xi )}}{{\d\xi ^p }}} } } \bigg)\frac{1}{{u^m }}} },
\\ & \mathcal{R}_{j,N}(u,\xi) = \sum\limits_{k = 0}^{n - 2} {\frac{1}{{u^k }}\sum\limits_{m = 0}^{N + k - 1} {\psi _{k,N + k - m + 1} (u,\xi )\sum\limits_{p = \max (0,m - N + 1)}^{\min (k,m)} {\binom{k}{p}\e^{2\pi \im j(k - p)/n} \frac{{\d^p \widehat{\mathsf{A}}_{j,m - p} (\xi )}}{{\d\xi ^p }}} } }.
\end{align*}
In arriving at \eqref{epsilonode}, we have made use of the identity $u(\psi _k (u,\xi ) - \delta _{0,k} ) = \psi _{k,1} (u,\xi )$.

We are seeking solutions to the differential equation \eqref{epsilonode} that satisfy limit conditions akin to those in \eqref{eq19}. We shall establish the following lemma.

\begin{lemma}\label{lemma31}
Let $0<d<d'$. With Condition \ref{cond}, equation \eqref{epsilonode} has, for each $0\le j\le n-1$ and $N\ge 1$, a unique solution $\varepsilon _{j,N} (u,\xi )$ which is analytic in $\left\{u: \Re(u) > \sigma'\right\} \times \Gamma_j(d)$ and satisfies
\begin{equation}\label{eq26}
\lim_{t \to  + \infty } \left[ \exp \Big(  - \e^{2\pi \im j/n} u\zeta \Big)\frac{\d^s \varepsilon _{j,N} (u,\zeta )}{u^s \d\zeta ^s }\right]_{\zeta  = \xi  + t\e^{\im\theta _j } }  = 0\qquad (0 \le s \le n - 2),
\end{equation}
provided that $\Re(u)\ge \sigma$, where $\sigma$ is an arbitrary fixed real number greater than $\sigma'$, and $\xi \in \Gamma_j(d)$. Furthermore, for each $\sigma>\sigma'$, there exists a constant $c_6$, independent of $u$, $\xi$, $N$ and $s$, such that
\begin{equation}\label{epsbound}
 \left| \frac{\d^s \varepsilon _{j,N} (u,\xi )}{u^s \d\xi ^s } \right| \le \left|\exp \Big( \e^{2\pi \im j/n} u\xi  \Big) \right| c_6 \frac{1}{\mathfrak{a}_n^N }\frac{(N + n + L)!}{d^N}\frac{1}{\left| u \right|^N }\sum\limits_{\ell \ne j} \frac{1}{(\max (1,\Re (\xi \e^{ - \im\theta _{j,\ell} } )))^\rho} \end{equation}
provided $\Re(u)\ge \sigma$ and $\xi \in \Gamma_j(d)$. Here, $L$ is the constant from Proposition \ref{prop1}.
\end{lemma}

By choosing the $\varepsilon _{j,N} (u,\xi )$ in \eqref{eq11} to be those given by Lemma \ref{lemma31}, we will demonstrate that the $W_{j,N}(u,\xi)$ are solutions of the differential equation \eqref{eq3} and that they are, in fact, independent of $N$. In order to prove Lemma \ref{lemma31}, we need to establish three additional lemmas.

\begin{lemma}\label{lemma32}
Let $0 < d < d'$. If Condition \ref{cond} is satisfied, then the following inequalities hold for each $0 \leq j \leq n-1$ and any $\xi \in \Gamma_j(d)$:
\[ 
\left|\widehat{\mathsf{A}}_{j,0} (\xi)\right| \leq c_7 \frac{1}{{\mathfrak{a}_n^{- 1}}} (L + 1)!,\quad \left| {\widehat{\mathsf{A}}_{j,m} (\xi )} \right| \le c_7\frac{1}{{\mathfrak{a}_n^{m - 1} }}\frac{1}{d^m}\frac{(m + L + 1)!}{(\max (1,\Re (\xi \e^{ - \im \theta _j } )))^{\rho }},
\]
for any positive integer $m$, and
\[
\left|\frac{{\d^s \widehat{\mathsf{A}}_{j,m} (\xi)}}{{\d\xi^s}}\right| \leq c_7 \frac{1}{{\mathfrak{a}_n^{m - 1}}} \frac{1}{{d^{m + s - 1}}} \frac{{(m + L + s + 1)!}}{{1 + \left|\xi\right|^{1 + \rho}}},
\]
for any non-negative integer $m$ and positive integer $s$ with $s \leq n$. Here, $L$ is the constant from Proposition \ref{prop1}, and $c_7$ is a suitable constant that is independent of $u$, $\xi$, $j$, $m$ and $s$.
\end{lemma}

\begin{proof} We will show that Lemma \ref{lemma32} holds with
\[
c_7  = \left( {\max (1,c_3 d) + c_3  + \frac{{c_3  + 1}}{{\mathfrak{a}_n }}} \right)\e^{\frac{{2(1 + \rho )}}{\rho }c_1 } .
\]
From the definition \eqref{eq18} and Lemmas \ref{lemma1} and \ref{lemma5}, we can assert that
\begin{equation}\label{Xbound}
\left| {X_j (\xi )} \right| \le \frac{{2(1 + \rho )}}{\rho }c_1 ,
\end{equation}
for $\xi \in \Gamma_j(d)$. Consequently,
\[
\left| {\widehat{\mathsf{A}}_{j,0} (\xi )} \right| \le \e^{\frac{{2(1 + \rho )}}{\rho }c_1 }  \le \frac{1}{{\mathfrak{a}_n }}\e^{\frac{{2(1 + \rho )}}{\rho }c_1 } \frac{1}{{\mathfrak{a}_n^{ - 1} }}(L + 1)!.
\]
Similarly, by \eqref{Xbound} and Proposition \ref{prop1}, we deduce
\[
\left| {\widehat{\mathsf{A}}_{j,m} (\xi )} \right| \le \e^{\frac{{2(1 + \rho )}}{\rho }c_1 } \frac{1}{{\mathfrak{a}_n^{m - 1} }}\frac{1}{d^m}\frac{(m + L + 1)!}{(\max (1,\Re (\xi \e^{ - \im \theta _j } )))^{\rho }}
\]
for any $m\ge 1$.

For $1 \le s\le n$, employing the identity \eqref{Bprop} along with inequality \eqref{Xbound} and Lemma \ref{lemma3}, we establish
\begin{align*}
\left| {\frac{{\d^s \widehat{\mathsf{A}}_{j,0} (\xi )}}{{\d\xi ^s }}} \right| = \left| {\e^{X_j (\xi )} } \right|\left| {\B_s (X'_j (\xi ),X''_j (\xi ), \ldots ,X_j^{(t)} (\xi ))} \right| & \le c_3 \e^{\frac{{2(1 + \rho )}}{\rho }c_1 } \frac{1}{{d^{s - 1} }}\frac{{(s - 1)!}}{{1 + \left| \xi  \right|^{1 + \rho } }} \\ & \le \frac{{c_3 }}{{\mathfrak{a}_n }}\e^{\frac{{2(1 + \rho )}}{\rho }c_1 } \frac{1}{{\mathfrak{a}_n^{ - 1} }}\frac{1}{{d^{s - 1} }}\frac{{(L + s + 1)!}}{{1 + \left| \xi  \right|^{1 + \rho } }}.
\end{align*}
Lastly, considering $m\ge 1$ and $1 \le s\le n$, again using the identity \eqref{Bprop}, inequality \eqref{Xbound}, Proposition \ref{prop1}, Lemmas \ref{lemma6} and \ref{lemma3}, and inequality \eqref{ineq2}, we deduce
\begin{align*}
& \left| {\frac{{\d^s \widehat{\mathsf{A}}_{j,m} (\xi )}}{{\d\xi ^s }}} \right| \le \left| {\e^{X_j (\xi )} } \right| \sum\limits_{h = 0}^s {\binom{s}{h}\left| {\B_h (X'_j (\xi ),X''_j (\xi ), \ldots ,X_j^{(h)} (\xi ))} \right|\left| {\frac{{\d^{s - h} \mathsf{A}_{j,m} (\xi )}}{{\d\xi ^{s - h} }}} \right|}
\\ & \le \e^{\frac{{2(1 + \rho )}}{\rho }c_1 } \frac{1}{{\mathfrak{a}_n^{m - 1} }}\frac{1}{{d^{m + s - 1} }}\frac{1}{{1 + \left| \xi  \right|^{1 + \rho } }}\\ & \quad\times  \bigg( {(m + L + s)! + c_3 d\sum\limits_{h = 1}^{s - 1} {\binom{s}{h}(h - 1)!(m + L + s - h)!}  + c_3 (s - 1)!(m + L + 1)!} \bigg)
\\ & \le \e^{\frac{{2(1 + \rho )}}{\rho }c_1 } \frac{1}{{\mathfrak{a}_n^{m - 1} }}\frac{1}{{d^{m + s - 1} }}\frac{1}{{1 + \left| \xi  \right|^{1 + \rho } }}\\ & \quad\times\bigg( {\max (1,c_3 d)\sum\limits_{h = 0}^s {\binom{s}{h}h!(m + L + s - h)!}  + c_3 (s - 1)!(m + L + 1)!} \bigg)
\\ & = \e^{\frac{{2(1 + \rho )}}{\rho }c_1 } \frac{1}{{\mathfrak{a}_n^{m - 1} }}\frac{1}{{d^{m + s - 1} }}\frac{1}{{1 + \left| \xi  \right|^{1 + \rho } }}\bigg( {\max (1,c_3 d)\frac{{(m + L + s + 1)!}}{{m + L + 1}} + c_3 (s - 1)!(m + L + 1)!} \bigg)
\\ & \le (\max (1,c_3 d) + c_3 )\e^{\frac{{2(1 + \rho )}}{\rho }c_1 } \frac{1}{{\mathfrak{a}_n^{m - 1} }}\frac{1}{d^{m + s - 1}}\frac{(m + L + s + 1)!}{1 + \left| \xi  \right|^{1 + \rho } }.
\end{align*}

\end{proof}

\begin{lemma}\label{lemma33}
Let $0<d<d'$ and $\sigma > \sigma'$. If Condition \ref{cond} is satisfied, then the following inequalities hold for each $0\le j\le n-1$, and any positive integer $N$:
\[
\left| \mathcal{P}_{j,N} (u,\xi ) + \mathcal{Q}_{j,N} (u,\xi ) + \mathcal{R}_{j,N} (u,\xi ) \right| \le c_8 \frac{1}{{\mathfrak{a}_n^N }}\frac{1}{{d^N }}\frac{{(N + n + L)!}}{{1 + \left| \xi  \right|^{1 + \rho } }}
\]
provided that $\Re(u)\ge \sigma$ and $\xi \in \Gamma_j(d)$. Here, $L$ is the constant from Proposition \ref{prop1}, and $c_8$ is a suitable constant that is independent of $u$, $\xi$, $j$ and $N$.
\end{lemma}

\begin{proof} We shall prove that Lemma \ref{lemma33} holds with
\[
c_8 = \frac{3c_7 \max (d,c_1 ,c_1 \mathfrak{a}_n d)(1 + \mathfrak{a}_n )^n \max(1,d)}{\min (1,\sigma ^{n - 2} )d}\frac{(n - 1)^2 }{\min (1,(\mathfrak{a}_n d)^{n - 2} )}.
\]
From Lemma \ref{lemma32}, we can readily infer that
\begin{align*}
\left| \mathcal{P}_{j,N} (u,\xi ) \right| & \le \sum\limits_{m = 0}^{n - 2} {\bigg( {\sum\limits_{p = m + 2}^{\min (n,N + m + 1)} {\binom{n}{p}\left| {\frac{{\d^p \widehat{\mathsf{A}}_{j,N + m - p + 1} (\xi )}}{{\d\xi ^p }}} \right|} } \bigg)\frac{1}{{\left| u \right|^m }}} 
\\ & \le c_7 \frac{1}{{\mathfrak{a}_n^N }}\frac{1}{{d^N }}\frac{1}{{1 + \left| \xi  \right|^{1 + \rho } }}\sum\limits_{m = 0}^{n - 2} {\bigg( {\frac{1}{{\mathfrak{a}_n^m }}\frac{1}{{d^m }}(N + m + L + 2)!\sum\limits_{p = m + 2}^{\min (n,N + m + 1)} {\binom{n}{p}\mathfrak{a}_n^p } } \bigg)\frac{1}{{\sigma ^m }}} 
\\ & \le \frac{{c_7 }}{{\min (1,\sigma ^{n - 2} )}}\frac{1}{{\mathfrak{a}_n^N }}\frac{1}{{d^N }}\frac{{(N + n + L)!}}{{1 + \left| \xi  \right|^{1 + \rho } }}\sum\limits_{m = 0}^{n - 2} {\frac{1}{{\mathfrak{a}_n^m }}\frac{1}{{d^m }}\sum\limits_{p = m + 2}^{\min (n,N + m + 1)} {\binom{n}{p}\mathfrak{a}_n^p } } 
\\ & \le \frac{{c_7 (1 + \mathfrak{a}_n )^n }}{{\min (1,\sigma ^{n - 2} )}}\frac{1}{{\mathfrak{a}_n^N }}\frac{1}{{d^N }}\frac{{(N + n + L)!}}{{1 + \left| \xi  \right|^{1 + \rho } }}\sum\limits_{m = 0}^{n - 2} {\frac{1}{{\mathfrak{a}_n^m }}\frac{1}{{d^m }}} 
\\ & \le \frac{{c_7 (1 + \mathfrak{a}_n )^n }}{{\min (1,\sigma ^{n - 2} )}}\frac{n - 1}{\min (1,(\mathfrak{a}_n d)^{n - 2} )}\frac{1}{{\mathfrak{a}_n^N }}\frac{1}{{d^N }}\frac{{(N + n + L)!}}{1 + \left| \xi  \right|^{1 + \rho } }.
\end{align*}
Using Lemmas \ref{lemma1} and \ref{lemma32}, along with inequality \eqref{ineq2}, we can assert that
\begin{align*}
\left| \mathcal{Q}_{j,N} (u,\xi ) \right| & \le \sum\limits_{k = 1}^{n - 2} {\sum\limits_{m = 0}^{k - 1} {\bigg( {\sum\limits_{q = 0}^{N + m} {\left| {\psi _{k,N + m - q + 1} (\xi )} \right|\sum\limits_{p = \max (0,q - N + 1)}^{\min (k,q)} {\binom{k}{p}\left| {\frac{{\d^p \widehat{\mathsf{A}}_{j,q - p} (\xi )}}{{\d\xi ^p }}} \right|} } } \bigg)\frac{1}{{\left| u \right|^m }}} } \\ & \le c_1 c_7 \max(1,d)\frac{1}{{\mathfrak{a}_n^N }}\frac{1}{{d^N }}\frac{1}{{1 + \left| \xi  \right|^{1 + \rho } }}\sum\limits_{k = 1}^{n - 2} \sum\limits_{m = 0}^{k - 1} \frac{1}{{\mathfrak{a}_n^{m - 1} }}\frac{1}{{d^{m} }} \\ & \quad\times \bigg( \sum\limits_{q = 0}^{N + m} {(N + m - q)!\sum\limits_{p = \max (0,q - N + 1)}^{\min (k,q)} {\binom{k}{p}\mathfrak{a}_n^p (q  + L  + 1)!} }  \bigg)\frac{1}{\sigma ^m }  
\\ & \le \frac{{c_1 c_7\max(1,d) }}{{\min (1,\sigma ^{n - 2} )}}\frac{1}{{\mathfrak{a}_n^N }}\frac{1}{{d^N }}\frac{1}{{1 + \left| \xi  \right|^{1 + \rho } }}\sum\limits_{k = 1}^{n - 2} \sum\limits_{m = 0}^{k - 1} \frac{1}{{\mathfrak{a}_n^{m - 1} }}\frac{1}{d^{m}} (N + m + L + 1)!\\ & \quad\times \sum\limits_{q = 0}^{N + m} \sum\limits_{p = \max (0,q - N + 1)}^{\min (k,q)} \binom{k}{p}\mathfrak{a}_n^p  
\\ & \le \frac{{c_1 c_7\max(1,d) }}{{\min (1,\sigma ^{n - 2} )}}\frac{1}{{\mathfrak{a}_n^N }}\frac{1}{{d^N }}\frac{{(N + n + L - 2)!}}{{1 + \left| \xi  \right|^{1 + \rho } }}\sum\limits_{k = 1}^{n - 2} \sum\limits_{m = 0}^{k - 1} \frac{1}{{\mathfrak{a}_n^{m - 1} }}\frac{1}{{d^{m} }} \\ & \quad\times \sum\limits_{q = 0}^{N + m} \sum\limits_{p = \max (0,q - N + 1)}^{\min (k,q)} \binom{k}{p}\mathfrak{a}_n^p 
\\ & \le \frac{{c_1 c_7 (1 + \mathfrak{a}_n )^n \max(1,d)}}{{\min (1,\sigma ^{n - 2} )}}\frac{1}{{\mathfrak{a}_n^N }}\frac{1}{{d^N }}\frac{{(N + n + L - 2)!}}{{1 + \left| \xi  \right|^{1 + \rho } }}\sum\limits_{k = 1}^{n - 2} {\sum\limits_{m = 0}^{k - 1} {\frac{1}{{\mathfrak{a}_n^{m - 1} }}\frac{1}{{d^{m} }}(N + m + 1)} } 
\\ & \le \frac{{c_1 c_7 (1 + \mathfrak{a}_n )^n \mathfrak{a}_n \max(1,d)}}{{\min (1,\sigma ^{n - 2} )}}\frac{1}{{\mathfrak{a}_n^N }}\frac{1}{{d^N }}\frac{{(N + n + L)!}}{{1 + \left| \xi  \right|^{1 + \rho } }}\sum\limits_{k = 1}^{n - 2} {\frac{{k - 1}}{{\min (1,(\mathfrak{a}_n d)^{k - 1} )}}} 
\\ & \le \frac{{c_1 c_7 (1 + \mathfrak{a}_n )^n \mathfrak{a}_n \max(1,d)}}{{\min (1,\sigma ^{n - 2} )}}\frac{{(n - 1)^2 }}{{\min (1,(\mathfrak{a}_n d)^{n - 2} )}}\frac{1}{{\mathfrak{a}_n^N }}\frac{1}{{d^N }}\frac{{(N + n + L)!}}{{1 + \left| \xi  \right|^{1 + \rho } }}.
\end{align*}
Similarly, employing Lemmas \ref{lemma1} and \ref{lemma32}, together with inequality \eqref{ineq1}, we deduce
\begin{align*}
\left| \mathcal{R}_{j,N} (u,\xi ) \right| & \le \sum\limits_{k = 0}^{n - 2} {\frac{1}{{\left| u \right|^k }}\sum\limits_{m = 0}^{N + k - 1} {\left| {\psi _{k,N + k - m + 1} (u,\xi )} \right|\sum\limits_{p = \max (0,m - N + 1)}^{\min (k,m)} {\binom{k}{p}\left| {\frac{{\d^p \widehat{\mathsf{A}}_{j,m - p} (\xi )}}{{\d\xi ^p }}} \right|} } } 
\\ & \le c_1 c_7\max(1,d) \frac{1}{{\mathfrak{a}_n^N }}\frac{1}{{d^{N+1} }}\frac{1}{{1 + \left| \xi  \right|^{1 + \rho } }}\sum\limits_{k = 0}^{n - 2} \frac{1}{{\sigma ^k }}\frac{1}{{\mathfrak{a}_n^k }}\frac{1}{{d^k }}\sum\limits_{m = 0}^{N + k - 1} (N + k - m + 1)!(m + L + 1)! \\ & \quad\times \sum\limits_{p = \max (0,m - N + 1)}^{\min (k,m)} \binom{k}{p}\mathfrak{a}_n^p 
\\ & \le \frac{{c_1 c_7 (1 + \mathfrak{a}_n )^n \max(1,d)}}{{\min (1,\sigma ^{n - 2} )}}\frac{1}{{\mathfrak{a}_n^N }}\frac{1}{{d^{N+1} }}\frac{1}{{1 + \left| \xi  \right|^{1 + \rho } }}\sum\limits_{k = 0}^{n - 2} {\frac{1}{{\mathfrak{a}_n^k }}\frac{1}{{d^k }}(N + k + L + 1)!(N + k)} 
\\ & \le \frac{{c_1 c_7 (1 + \mathfrak{a}_n )^n \max(1,d)}}{{\min (1,\sigma ^{n - 2} )}}\frac{1}{{\mathfrak{a}_n^N }}\frac{1}{{d^{N+1} }}\frac{{(N + n + L)!}}{{1 + \left| \xi  \right|^{1 + \rho } }}\sum\limits_{k = 0}^{n - 2} {\frac{1}{{\mathfrak{a}_n^k }}\frac{1}{{d^k }}} 
\\ & \le \frac{{c_1 c_7 (1 + \mathfrak{a}_n )^n \max(1,d)}}{{\min (1,\sigma ^{n - 2} )d}}\frac{{n - 1}}{{\min (1,(\mathfrak{a}_n d)^{n - 2} )}}\frac{1}{{\mathfrak{a}_n^N }}\frac{1}{{d^{N} }}\frac{{(N + n + L)!}}{{1 + \left| \xi  \right|^{1 + \rho } }}.
\end{align*}
The triangle inequality now leads to the desired result.
\end{proof}

The proof of the following lemma is virtually identical to that of Lemma \ref{lemma5} and is therefore omitted.

\begin{lemma}\label{lemma34} Given that $d > 0$ and $\rho>0$, the following inequalities are valid for each $0\le j\le n-1$, $0\le \ell\le n-1$, where $\ell\ne j$:
\[
\int_{\mathscr{P}_{j,\ell} (\xi )} {\frac{{\left| \d t \right|}}{{1 + \left| t \right|^{1 + \rho } }}}  \le \frac{{2(1 + \rho )}}{\rho }\frac{1}{{(\max (1,\Re (\xi \e^{ - \im \theta _{j,\ell} } )))^\rho  }},
\]
provided that $\xi \in \Gamma_j(d)$.
\end{lemma}

\begin{proof}[Proof of Lemma \ref{lemma31}]
We begin by transforming the differential equation \eqref{epsilonode} into an integral equation for $\varepsilon _{j,N} (u,\xi )$, treating \eqref{epsilonode} as an inhomogeneous equation and applying the method of variation of parameters. A set of fundamental solutions to the corresponding homogeneous equation consists of $\exp (u\xi ),\exp ( \e^{2\pi \im/n} u\xi ), \ldots ,\exp (\e^{2\pi \im(n - 1)/n} u\xi)$. The Wronskian $\mathscr{W}$ of these functions is
\begin{align*}
&  \begin{vmatrix}
   {\exp (u\xi )} & {\exp \left( {\e^{2\pi \im /n} u\xi } \right)} &  \cdots  & {\exp \left( {\e^{2\pi \im (n - 1)/n} u\xi } \right)}  \\
   {u\exp (u\xi )} & {\e^{2\pi \im /n} u\exp \left( {\e^{2\pi \im /n} u\xi } \right)} &  \cdots  & {\e^{2\pi \im (n - 1)/n} u\exp \left( {\e^{2\pi \im (n - 1)/n} u\xi } \right)}  \\
    \vdots  &  \vdots  &  \ddots  &  \vdots   \\
   {u^{n - 1} \exp (u\xi )} & {\e^{2\pi \im (n - 1)/n} u^{n - 1} \exp \left( {\e^{2\pi \im /n} u\xi } \right)} &  \cdots  & {\e^{2\pi \im (n - 1)^2 /n} u^{n - 1} \exp \left( {\e^{2\pi \im (n - 1)/n} u\xi } \right)}  \\
\end{vmatrix}
\\ & = u^{n(n - 1)/2}\begin{vmatrix}
   1 & 1 &  \cdots  & 1  \\
   1 & {\e^{2\pi \im /n} } &  \cdots  & {\e^{2\pi \im (n - 1)/n} }  \\
    \vdots  &  \vdots  &  \ddots  &  \vdots   \\
   {1 } & {\e^{2\pi \im (n - 1)/n}  } &  \cdots  & {\e^{2\pi \im (n - 1)^2 /n}  }  \\
\end{vmatrix}  = u^{n(n - 1)/2} \prod\limits_{0 \le k < j \le n - 1} (\e^{2\pi \im j/n}  - \e^{2\pi \im k/n} ) .
\end{align*}
For each $0\le \ell \le n-1$, let $\mathscr{W}_\ell$ denote the Wronskian of the set of functions obtained by deleting $\exp( \e^{2\pi \im \ell/n} u\xi)$ from $\exp (u\xi ),\exp ( \e^{2\pi \im/n} u\xi ), \ldots ,\exp (\e^{2\pi \im(n - 1)/n} u\xi)$ and keeping the remaining functions in the same order. This Wronskian can be expressed as 
\begin{align*}
\mathscr{W}_\ell  & = u^{(n - 1)(n - 2)/2} \exp \Big( - \e^{2\pi \im \ell /n} u\xi  \Big) \prod\limits_{\substack{0 \le k < j \le n - 1\\ k,j \ne \ell}} {(\e^{2\pi \im j/n}  - \e^{2\pi \im k/n} )} 
\\ & = ( - 1)^\ell u^{(n - 1)(n - 2)/2} \exp \Big( - \e^{2\pi \im \ell/n} u\xi \Big) \cfrac{{\displaystyle\prod_{0 \le k < j \le n - 1} {(\e^{2\pi \im j/n}  - \e^{2\pi \im k/n} )} }}{{\displaystyle\prod_{\substack{0 \le j \le n - 1\\ j \ne \ell}} {(\e^{2\pi \im j/n}  - \e^{2\pi \im \ell/n} )} }}
\\ & = \frac{{( - 1)^{n + \ell + 1} }}{n}\e^{2\pi \im \ell/n}u^{(n - 1)(n - 2)/2}  \exp \Big(  - \e^{2\pi \im \ell/n} u\xi  \Big)  \prod\limits_{0 \le k < j \le n - 1} {(\e^{2\pi \im j/n}  - e^{2\pi \im k/n} )} .
\end{align*}
Thus, we find that
\[
( - 1)^{n + \ell + 1}\frac{\mathscr{W}_\ell}{\mathscr{W}} = \frac{1}{n}\frac{1}{{u^{n - 1} }}\e^{2\pi \im \ell/n} \exp\Big( { - \e^{2\pi \im \ell/n} u\xi } \Big)
\]
and employing the method of variation of parameters, a particular solution to the inhomogeneous equation \eqref{epsilonode} is expressed as
\begin{align*}
\varepsilon _{j,N} (u,\xi )   = & \; \frac{1}{n}\frac{1}{{u^N }}\exp \Big( {\e^{2\pi \im j/n} u\xi } \Big)\sum_{\ell =0}^{n-1}  \int^\xi  \e^{2\pi \im \ell/n} \exp \Big( { (\e^{2\pi \im \ell/n}  - \e^{2\pi \im j/n} )u(\xi  - t)} \Big) \\
& \;\quad\qquad\qquad\qquad\qquad\qquad\times (\mathcal{P}_{j,N} (u,t) + \mathcal{Q}_{j,N} (u,t) + \mathcal{R}_{j,N} (u,t))\d t
\\ &  + \frac{1}{n}\sum_{\ell =0}^{n-1} \int^\xi \exp \Big( {\e^{2\pi \im j/n} u(\xi  - t)} \Big) \e^{2\pi \im \ell/n} \exp \Big( {(\e^{2\pi \im \ell/n}  - \e^{2\pi \im j/n} )u(\xi  - t)} \Big) \\
& \;\;\;\quad\qquad\times \bigg( \sum\limits_{k = 0}^{n - 2} \psi _{k,1} (u,t )\frac{{\d^k \varepsilon _{j,N} (u,t)}}{u^k \d t^k }  \bigg)\d t.
\end{align*}
We now fix the constants of integration in a specific way. When $\ell \ne j$, we integrate along the path $\mathscr{P}_{j,\ell}(\xi)$. For $\ell = j$, the integral is decomposed into a sum of $n-1$ integrals, each taken along a distinct path $\mathscr{P}_{j,\ell}(\xi)$ for $0 \leq \ell \leq n-1$, $\ell \ne j$. Note that due to the orientation of these paths, the signs of the integrals must be reversed. This leads us to the expression
\begin{gather}\label{inteq}
\begin{split}
& \varepsilon _{j,N} (u,\xi ) \\ &  = - \frac{1}{n}\frac{1}{{u^N }}\exp \Big( {\e^{2\pi \im j/n} u\xi } \Big)\sum\limits_{\ell \ne j}  \int_{\mathscr{P}_{j,\ell} (\xi )} \bigg( \frac{{\e^{2\pi \im j/n} }}{{n - 1}} + \e^{2\pi \im \ell/n} \exp \Big( { (\e^{2\pi \im \ell/n}  - \e^{2\pi \im j/n} )u(\xi  - t)} \Big) \bigg)\\
& \quad\qquad\qquad\qquad\qquad\qquad\qquad\times (\mathcal{P}_{j,N} (u,t) + \mathcal{Q}_{j,N} (u,t) + \mathcal{R}_{j,N} (u,t))\d t
\\ & \quad\; - \frac{1}{n}\sum\limits_{\ell \ne j}  \int_{\mathscr{P}_{j,\ell} (\xi )} \exp \Big( {\e^{2\pi \im j/n} u(\xi  - t)} \Big)\bigg( {\frac{{\e^{2\pi \im j/n} }}{{n - 1}} + \e^{2\pi \im \ell/n} \exp \Big( {(\e^{2\pi \im \ell/n}  - \e^{2\pi \im j/n} )u(\xi  - t)} \Big)} \bigg) \\
& \quad\qquad\qquad\times \bigg( \sum\limits_{k = 0}^{n - 2} \psi _{k,1} (u,t )\frac{{\d^k \varepsilon _{j,N} (u,t)}}{u^k \d t^k }  \bigg)\d t.
\end{split}
\end{gather}
We aim to demonstrate that the integral equation \eqref{inteq} possesses a unique solution $\varepsilon _{j,N} (u,\xi )$ with the properties specified in Lemma \ref{lemma31}. To this end, we solve \eqref{inteq} via successive approximation. We introduce a sequence of functions $h_{m,j,N}(u,\xi)$, with $m=0,1,2,\ldots$, subject to the conditions $\Re(u)>\sigma'$ and $\xi \in \Gamma_j(d)$, defined by
\begin{gather}\label{h0}
\begin{split}
& h_{0,j,N} (u,\xi ) \\ &  = - \frac{1}{n}\frac{1}{{u^N }}\exp \Big( {\e^{2\pi \im j/n} u\xi } \Big)\sum\limits_{\ell \ne j}  \int_{\mathscr{P}_{j,\ell} (\xi )} \bigg( \frac{{\e^{2\pi \im j/n} }}{{n - 1}} + \e^{2\pi \im \ell/n} \exp \Big( { (\e^{2\pi \im \ell/n}  - \e^{2\pi \im j/n} )u(\xi  - t)} \Big) \bigg)\\
& \quad\qquad\qquad\qquad\qquad\qquad\qquad\times (\mathcal{P}_{j,N} (u,t) + \mathcal{Q}_{j,N} (u,t) + \mathcal{R}_{j,N} (u,t))\d t,
\end{split}
\end{gather}
and
\begin{gather}\label{hm}
\begin{split}
& h_{m+1,j,N} (u,\xi ) \\ &  = - \frac{1}{n}\sum\limits_{\ell\ne j}  \int_{\mathscr{P}_{j,\ell} (\xi )} \exp \Big( {\e^{2\pi \im j/n} u(\xi  - t)} \Big)\bigg( {\frac{{\e^{2\pi \im j/n} }}{{n - 1}} + \e^{2\pi \im \ell/n} \exp \Big( {(\e^{2\pi \im \ell/n}  - \e^{2\pi \im j/n} )u(\xi  - t)} \Big)} \bigg) \\
& \quad\qquad\qquad\times \bigg( \sum\limits_{k = 0}^{n - 2} \psi _{k,1} (u,t )\frac{{\d^k h_{m,j,N} (u,t)}}{u^k \d t^k }  \bigg)\d t
\end{split}
\end{gather}
for $m\ge 0$. Differentiating equation \eqref{h0} $s$ times with respect to $\xi$, where $0 \leq s \leq n-2$, yields
\begin{align*}
& \frac{{\d^s h_{0,j,N} (u,\xi)}}{u^s \d\xi ^s } \\ &  = - \frac{1}{n}\frac{1}{{u^N }}\exp \Big( {\e^{2\pi \im j/n} u\xi } \Big)\sum\limits_{\ell \ne j}  \int_{\mathscr{P}_{j,\ell} (\xi )} \bigg( \frac{{\e^{2\pi \im j(s+1)/n} }}{{n - 1}} + \e^{2\pi \im \ell(s+1)/n} \exp \Big(  (\e^{2\pi \im \ell/n}  - \e^{2\pi \im j/n} )u(\xi  - t) \Big) \bigg)\\
& \quad\qquad\qquad\qquad\qquad\qquad\qquad\times (\mathcal{P}_{j,N} (u,t) + \mathcal{Q}_{j,N} (u,t) + \mathcal{R}_{j,N} (u,t))\d t.
\end{align*}
Hereafter, let $\sigma$ denote an arbitrary fixed real number greater than $\sigma'$. Note that, as a consequence of the definition of the path $\mathscr{P}_{j,\ell} (\xi)$, the inequality $\Re((\e^{2\pi \im \ell/n} - \e^{2\pi \im j/n})u(\xi - t)) \leq 0$ holds as $t$ traverses $\mathscr{P}_{j,\ell} (\xi)$. Using this inequality, we see that for $\Re(u)\ge \sigma$ and $\xi \in \Gamma_j(d)$, 
\begin{equation}\label{h0bound}
\left| \frac{{\d^s h_{0,j,N} (u,\xi)}}{u^s \d\xi ^s } \right| \le \left| \exp \Big( {\e^{2\pi \im j/n} u\xi } \Big) \right|\frac{{2(1 + \rho )}}{\rho }\frac{{c_8 }}{{n - 1}}\frac{1}{{\mathfrak{a}_n^N }}\frac{{(N + n + L)!}}{{d^N }}\frac{1}{{\left| u \right|^N }},
\end{equation}
in virtue of Lemmas \ref{lemma33} and \ref{lemma34}. In a similar manner, differentiating \eqref{hm} with $m=0$ $s$ times with respect to $\xi$, where $0 \leq s \leq n-2$, gives
\begin{align*}
& \frac{{\d^s h_{1,j,N} (u,\xi)}}{u^s \d\xi ^s } \\ &  = - \frac{1}{n}\sum\limits_{\ell \ne j}  \int_{\mathscr{P}_{j,\ell} (\xi )} \exp \Big( {\e^{2\pi \im j/n} u(\xi  - t)} \Big)\bigg( {\frac{{\e^{2\pi \im j(s+1)/n} }}{{n - 1}} + \e^{2\pi \im \ell(s+1)/n} \exp \Big( {(\e^{2\pi \im \ell/n}  - \e^{2\pi \im j/n} )u(\xi  - t)} \Big)} \bigg) \\
& \quad\qquad\qquad\times \bigg( \sum\limits_{k = 0}^{n - 2} \psi _{k,1} (u,t )\frac{{\d^k h_{0,j,N} (u,t)}}{u^k \d t^k }  \bigg)\d t.
\end{align*}
By appealing to the properties of the paths $\mathscr{P}_{j,\ell} (\xi)$ and the inequalities \eqref{h0bound}, we readily establish the following bounds:
\begin{align*}
\left| \frac{{\d^s h_{1,j,N} (u,\xi)}}{u^s \d\xi ^s }\right| \le\; & \left| \exp \Big( {\e^{2\pi \im j/n} u\xi } \Big) \right|\frac{{2(1 + \rho )}}{\rho }\frac{{c_8 }}{{(n - 1)^2 }}\frac{1}{{\mathfrak{a}_n^N }}\frac{{(N + n + L)!}}{{d^N }}\frac{1}{{\left| u \right|^N }}
\\ & \times \sum\limits_{\ell \ne j} \int_{\mathscr{P}_{j,\ell} (\xi )} \bigg( {\sum\limits_{k = 0}^{n - 2} {\left| {\psi _{k,1} (u,t)} \right|} } \bigg)\left| \d t \right|,
\end{align*}
provided that $\Re(u)\ge \sigma$ and $\xi \in \Gamma_j(d)$. Continuing the argument by induction, we deduce that for any $m\geq 1$:
\begin{gather}\label{hmbound}
\begin{split}
\left| {\frac{{\d^s h_{m,j,N} (u,\xi )}}{{u^s \d\xi ^s }}} \right| \le \; & \left| \exp \Big( {\e^{2\pi \im j/n} u\xi } \Big) \right|\frac{{2(1 + \rho )}}{\rho }\frac{c_8}{{(n - 1)^2 }}\frac{1}{{\mathfrak{a}_n^N }}\frac{{(N + n + L)!}}{{d^N }}\frac{1}{{\left| u \right|^N }}\\ & \times \sum\limits_{\ell\ne j} \frac{1}{{m!}}\bigg( {\int_{\mathscr{P}_{j,\ell} (\xi )} {\bigg( {\sum\limits_{k = 0}^{n - 2} {\left| {\psi _{k,1} (u,t)} \right|} } \bigg)\left| \d t \right|} } \bigg)^m ,
\end{split}
\end{gather}
provided $\Re(u)\ge \sigma$ and $\xi \in \Gamma_j(d)$, with the key observation in the induction step being the identity
\[
\bigg(\sum\limits_{k = 0}^{n - 2} {\left| {\psi _{k,1} (u,t)} \right|} \bigg)\left| \d t \right| = -\d\bigg( {\int_{\mathscr{P}_{j,p} (t)} {\bigg( {\sum\limits_{k = 0}^{n - 2} {\left| {\psi _{k,1} (u,w)} \right|} } \bigg)\left| \d w \right|} } \bigg), \quad p\neq j.
\]
We now define
\begin{equation}\label{epsdef}
\varepsilon _{j,N} (u,\xi ) = \sum\limits_{m = 0}^\infty h_{m,j,N} (u,\xi ).
\end{equation}
It is evident from the estimates \eqref{h0bound} and \eqref{hmbound} that the series \eqref{epsdef} converges uniformly with respect to $\xi$ in $\Gamma_j(d)$ and also converges uniformly with respect to $u$ for $\Re(u)\ge \sigma$. Consequently, given that each $h_{m,j,N}(u,\xi)$ is analytic with respect to $\xi$ within $\Gamma_j(d)$ and also analytic with respect to $u$ for $\Re(u)>\sigma'$, it follows that the $\varepsilon _{j,N} (u,\xi )$ defined by the series \eqref{epsdef} is an analytic function of $\xi \in \Gamma_j(d)$ and an analytic function of $u$ for $\Re(u)>\sigma'$. Given that the function $\varepsilon_{j,N}(u,\xi)$ is continuous in $(u, \xi)$ and analytic in each of its variables, it constitutes an analytic function in $\left\{u: \Re(u) > \sigma'\right\} \times \Gamma_j(d)$. Furthermore, term-wise differentiation of the series \eqref{epsdef} is permissible, and using \eqref{h0bound} and \eqref{hmbound}, we obtain that for $0\leq s\leq n-2$,
\begin{align*}
\left| \frac{{\d^s \varepsilon _{j,N} (u,\xi )}}{{u^s \d\xi ^s }}\right| = \left| \sum\limits_{m = 0}^\infty \frac{\d^s h_{m,j,N} (u,\xi )}{{u^s \d\xi ^s }} \right| \le \; &\left| \exp \Big( {\e^{2\pi \im j/n} u\xi } \Big) \right|\frac{{2(1 + \rho )}}{\rho }\frac{{c_8 }}{{n - 1}}\frac{1}{{\mathfrak{a}_n^N }}\frac{{(N + n + L)!}}{{d^N }}\frac{1}{{\left| u \right|^N }} \\ & \times \exp \bigg( {\sum\limits_{\ell \ne j} {\int_{\mathscr{P}_{j,\ell} (\xi )} {\bigg( {\sum\limits_{k = 0}^{n - 2} {\left| {\psi _{k,1} (u,t)} \right|} } \bigg)\left| \d t \right|} } } \bigg),
\end{align*}
whenever $\Re(u)\ge \sigma$ and $\xi \in \Gamma_j(d)$. Under these assumptions, employing Lemmas \ref{lemma1} and \ref{lemma34}, we obtain
\[
\exp \bigg( \sum\limits_{\ell\ne j} \int_{\mathscr{P}_{j,\ell} (\xi )} \bigg( \sum\limits_{k = 0}^{n - 2} \left| \psi _{k,1} (u,t) \right|\bigg)\left| \d t \right|\bigg) \le \exp \bigg( \frac{2(1 + \rho )}{\rho }(n - 1)^2 c_1 \frac{1}{\mathfrak{a}_n }\frac{1}{d}\bigg).
\]
Consequently,
\begin{equation}\label{eq24}
\left| \frac{\d^s \varepsilon _{j,N} (u,\xi )}{u^s \d\xi ^s } \right|\le \left| \exp \Big( {\e^{2\pi \im j/n} u\xi } \Big) \right|
c_9 \frac{1}{{\mathfrak{a}_n^N }}\frac{{(N + n + L)!}}{{d^N }}\frac{1}{{\left| u \right|^N }}
\end{equation}
provided $\Re(u)\ge \sigma$ and $\xi \in \Gamma_j(d)$, with
\[
c_9  = \frac{{2(1 + \rho )}}{\rho }\frac{{c_7 }}{{n - 1}}\exp \bigg( {\frac{{2(1 + \rho )}}{\rho }(n - 1)^2 c_1 \frac{1}{\mathfrak{a}_n}\frac{1}{d}} \bigg).
\]
Using the properties of the paths $\mathscr{P}_{j,\ell} (\xi)$, along with Lemmas \ref{lemma33}, \ref{lemma1}, and \ref{lemma34}, as well as inequality \eqref{eq24}, we derive enhanced bounds from \eqref{inteq} as follows:
\begin{equation}\label{eq25}
 \left| {\frac{{\d^s \varepsilon _{j,N} (u,\xi )}}{{u^s \d\xi ^s }}} \right| \le \left|\exp \Big( {\e^{2\pi \im j/n} u\xi } \Big) \right| c_6 \frac{1}{{\mathfrak{a}_n^N }}\frac{{(N + n + L)!}}{{d^N }}\frac{1}{{\left| u \right|^N }}\sum\limits_{\ell \ne j} {\frac{1}{{(\max (1,\Re (\xi \e^{ - \im\theta _{j,\ell} } )))^\rho  }}} 
\end{equation}
under the conditions $\Re(u)\ge \sigma$ and $\xi \in \Gamma_j(d)$, with
\[
c_6 = \frac{{2(1 + \rho )}}{\rho }\left( {\frac{{c_8 }}{{n - 1}} + \frac{{c_1 c_9 }}{{\mathfrak{a}_N d}}} \right).
\]
From \eqref{theta1} and \eqref{theta2}, we infer that 
\begin{equation}\label{thetaineq}
|\theta_j-\theta_{j,\ell}|<\frac{\pi}{2}.
\end{equation}
This observation, when combined with \eqref{eq25}, shows that $\varepsilon _{j,N} (u,\xi )$ satisfies the limit conditions specified in \eqref{eq26}. This concludes the proof of existence.

To establish uniqueness, let us assume the existence of another analytic solution $\widetilde{\varepsilon}_{j,N} (u,\xi )$ for each $0\le j\le n-1$ and $N\ge 1$, satisfying the limit conditions \eqref{eq26}. If we define $\upsilon_{j,N} (u,\xi )=\varepsilon_{j,N} (u,\xi )-\widetilde{\varepsilon}_{j,N} (u,\xi )$ for $(u,\xi)\in\left\{u: \Re(u) > \sigma'\right\} \times \Gamma_j(d)$, then $\upsilon_{j,N} (u,\xi )$ satisfies the $n^{\text{th}}$-order linear homogeneous equation
\[
\frac{\d^n \upsilon _{j,N} (u,\xi )}{\d\xi^n } - u^n \upsilon _{j,N} (u,\xi ) = u^{n-1} \sum\limits_{k = 0}^{n - 2} \psi _{k,1} (u,\xi )\frac{\d^k \upsilon _{j,N} (u,\xi )}{u^k \d\xi^k }.
\]
Furthermore, 
\begin{equation}\label{eq28}
\lim_{t \to  + \infty } \left[ \exp \Big(  - \e^{2\pi \im j/n} u\zeta \Big)\frac{\d^s \upsilon _{j,N} (u,\zeta )}{u^s \d\zeta ^s }\right]_{\zeta  = \xi  + t\e^{\im\theta _j } }  = 0\qquad (0 \le s \le n - 2),
\end{equation}
provided that $\Re(u)\ge \sigma$ and $\xi \in \Gamma_j(d)$. Under the temporary assumption that $u$ is real, employing the method of variation of parameters shows that $\upsilon_{j,N}(u,\xi)$ satisfies the following integral equation:
\begin{gather}\label{inteq2}
\begin{split}
& \upsilon _{j,N} (u,\xi ) \\ &  =  - \frac{1}{n}  \int_{\mathscr{P}_j (\xi )} \exp \Big( {\e^{2\pi \im j/n} u(\xi  - t)} \Big)\bigg( {\frac{{\e^{2\pi \im j/n} }}{{n - 1}} + \sum\limits_{\ell \ne j}\e^{2\pi \im \ell/n} \exp \Big( {(\e^{2\pi \im \ell/n}  - \e^{2\pi \im j/n} )u(\xi  - t)} \Big)} \bigg)\\
& \quad\qquad\times \bigg( \sum\limits_{k = 0}^{n - 2} \psi _{k,1} (u,t )\frac{{\d^k \upsilon _{j,N} (u,t)}}{u^k \d t^k }  \bigg)\d t
\\ &\quad\; +\sum_{k=0}^{n-1}\lambda_{k,j,N}(u) \exp \Big( \e^{2\pi \im k /n} u\xi  \Big),
\end{split}
\end{gather}
where the $\lambda_{k,j,N}(u)$ are some functions of $u$ that are independent of $\xi$. The convergence of the integrals in \eqref{inteq2} is assured by the definition of $\mathscr{P}_j (\xi )$, \eqref{thetaineq}, \eqref{eq28}, and Lemmas \ref{lemma1} and \ref{lemma5}.
Multiplying both sides of \eqref{inteq2} by $\exp(  - \e^{2\pi \im j/n} u\xi )$, substituting $\xi$ with $\xi +w \e^{\im\theta _j }$, and letting $w\to+\infty$ shows that the functions  $\lambda_{k,j,N}(u)$ must all be identically zero. Then, by differentiating \eqref{inteq2} $s$ times with respect to $\xi$, where $0 \leq s \leq n-2$, we obtain
\begin{gather}\label{inteq3}
\begin{split}
& \frac{{\d^s \upsilon_{j,N} (u,\xi)}}{u^s \d\xi ^s } \\ &  = - \frac{1}{n}  \int_{\mathscr{P}_{j} (\xi )} \exp \Big( {\e^{2\pi \im j/n} u(\xi  - t)} \Big)\bigg( {\frac{{\e^{2\pi \im j(s+1)/n} }}{{n - 1}} +\sum\limits_{\ell \ne j} \e^{2\pi \im \ell(s+1)/n} \exp \Big( {(\e^{2\pi \im \ell/n}  - \e^{2\pi \im j/n} )u(\xi  - t)} \Big)} \bigg) \\
& \quad\qquad\times \bigg( \sum\limits_{k = 0}^{n - 2} \psi _{k,1} (u,t )\frac{{\d^k \upsilon_{j,N} (u,t)}}{u^k \d t^k }  \bigg)\d t.
\end{split}
\end{gather}
Given fixed values for $u$ and $\xi$, equation \eqref{eq28} implies that for any $t\in \mathscr{P}_j(\xi)$, the following inequality holds:
\[
\left|\exp \Big( { - \e^{2\pi \im j/n} ut} \Big)\frac{\d^s \upsilon _{j,N} (u,t)}{u^s \d t^s }\right| \le c_{10} .
\]
Here $c_{10}$ is a suitable constant that is independent of both $t$ and $s$. Successive re-substitutions on the right-hand side of \eqref{inteq3} yield
\[
\left|\frac{\d^s \upsilon _{j,N} (u,\xi )}{u^s \d\xi ^s }\right| \le \left| \exp \Big(  - \e^{2\pi \im j/n} u\xi  \Big)\right|c_{10} \frac{1}{m!}\bigg( \int_{\mathscr{P}_j (\xi )} \bigg( \sum\limits_{k = 0}^{n - 2} \left| {\psi _{k,1} (u,t)} \right| \bigg)\left| \d t \right| \bigg)^m ,
\]
where $m$ is an arbitrary non-negative integer. Letting $s=0$ and $m\to +\infty$, we see that $\upsilon _{j,N} (u,\xi )$ is zero. Consequently, we establish that $\widetilde{\varepsilon}_{j,N} (u,\xi )=\varepsilon_{j,N} (u,\xi )$ holds true for every $u\ge \sigma$ and $\xi \in \Gamma_j(d)$. The requirement for $u$ to be real can be removed through analytic continuation. With this, we conclude the proof of Lemma \ref{lemma31}.
\end{proof}

We are now in a position to prove Proposition \ref{prop2}.

\begin{proof}[Proof of Proposition \ref{prop2}]
Using \eqref{ahatrec}, it can be verified that for each $0\le j\le n-1$ and $N\ge 1$,
\[
\exp \Big( {\e^{2\pi \im j/n} u\xi } \Big)\frac{\widehat{\mathsf{A}}_{j,N} (\xi )}{u^N}  + \varepsilon _{j,N+1} (u,\xi )
\]
also satisfies the equation \eqref{epsilonode} alongside
$\varepsilon _{j,N} (u,\xi )$. Additionally, according to Lemma \ref{lemma32}, it fulfils the limit conditions \eqref{eq26}. Thus, by the uniqueness property, it follows that
\[
\varepsilon _{j,N} (u,\xi ) = \exp \Big( {\e^{2\pi \im j/n} u\xi } \Big)\frac{\widehat{\mathsf{A}}_{j,N} (\xi )}{u^N}  + \varepsilon _{j,N+1} (u,\xi ).
\]
Accordingly, by choosing the $\varepsilon _{j,N} (u,\xi )$ in \eqref{eq11} to be those given by Lemma \ref{lemma31}, we can assert that the $W_{j,N}(u,\xi)$ are solutions of the differential equation \eqref{eq3} and that these solutions are independent of $N$. Consequently, by defining the $\eta_j(u,\xi )$ as
\[
  \eta_j(u,\xi )= \exp \Big( -\e^{2\pi \im j/n} u\xi -X_j(\xi) \Big) \varepsilon _{j,1} (u,\xi ),
\]
it is straightforward to verify that they fulfil all the properties specified in Proposition \ref{prop2}, with the remainder terms $R_{j,N}(u,\xi)$ expressed as
\[
R_{j,N}(u,\xi) = \exp \Big( -\e^{2\pi \im j/n} u\xi -X_j(\xi) \Big) \varepsilon _{j,N} (u,\xi ).
\]
In particular, the estimates in \eqref{eq29} follow from \eqref{epsbound} and \eqref{Xbound}, as they imply
\[
\left| R_{j,N} (u,\xi ) \right| \le \e^{\frac{2(1 + \rho )}{\rho }c_1 } c_6 \frac{1}{\mathfrak{a}_n^N}\frac{(N + n + L)!}{d^N }\frac{1}{\left| u \right|^N }\sum\limits_{\ell  \ne j} \frac{1}{(\max (1,\Re(\xi \e^{ - \im\theta _{j,\ell } } )))^\rho  }.
\]
\end{proof}

\section{Constructing the Borel transforms}\label{Boreltransformsec}
In this section, we construct the Borel transforms $F_j(t, \xi)$ of the formal solutions \eqref{eq4} and demonstrate their properties as outlined in Theorem \ref{thm1}. We start with the simple observation that if the Borel transforms $F_j(t, \xi)$ of the formal solutions \eqref{eq4} exist, they must be the inverse Laplace transforms of the functions $\eta_j(u, \xi)$, given in Proposition \ref{prop2}, with respect to $u$. Hence, for $t \ge 0$ and $\xi \in \Gamma_j(d)$, we define the functions $F_j(t, \xi)$ as
\[
F_j(t,\xi)=\frac{1}{2\pi \im} \lim_{U\to +\infty}\int_{\sigma-\im U}^{\sigma+\im U}\e^{u t} \eta_j(u, \xi)\d u,
\]
where $\sigma$ is any real number greater than $\sigma'$. Our aim is to show that these functions extend analytically to an appropriate neighbourhood of the positive real axis with respect to $t$, exhibiting at most exponential growth. Upon substituting the truncated expansions \eqref{eq30} of the functions $\eta_j(u, \xi)$ with $N$ being replaced by $N + 2$ and integrating term-by-term, we obtain
\begin{equation}\label{eq31}
F_j(t,\xi)=\sum_{m=0}^{N}\frac{\mathsf{A}_{j,m+1}(\xi)}{m!}t^m + \frac{1}{2\pi \im} \int_{\sigma-\im \infty}^{\sigma+\im \infty}\e^{u t} R_{j,N+2}(u, \xi)\d u.
\end{equation}
Note that, due to the bounds \eqref{eq29} for $R_{j,N+2}(u, \xi)$, the contour integrals converge absolutely, and the convergence is uniform with respect to $\xi \in \Gamma_j(d)$ for any $N \ge 0$. Consequently, the $F_j(t, \xi)$ are analytic functions of $\xi$ in $\Gamma_j(d)$. Differentiating both sides of \eqref{eq31} $N$ times with respect to $t$ yields
\begin{equation}\label{eq33}
\frac{\d^N F_j(t,\xi)}{\d t^N}=\mathsf{A}_{j,N+1}(\xi)+ \frac{1}{2\pi \im} \int_{\sigma-\im \infty}^{\sigma+\im \infty}\e^{u t} u^N R_{j,N+2}(u, \xi)\d u.
\end{equation}
From Proposition \ref{prop2}, we infer
\[
\left| \mathsf{A}_{j,N + 1} (\xi ) \right| = \mathop {\lim }\limits_{u \to  + \infty } u^{N + 1} \left| {R_{j,N + 1} (u,\xi )} \right| \le c_5 \frac{1}{\mathfrak{a}_n^{N + 1} }\frac{{(N + n + L + 1)!}}{{d^{N + 1} }}\sum\limits_{\ell  \ne j} {\frac{1}{(\max (1,\Re(\xi \e^{ - \im\theta _{j,\ell } } )))^\rho  }} .
\]
By combining these estimates with \eqref{eq29}, we obtain the following bounds:
\begin{align*}
&\left|\frac{\d^N F_j(t,\xi)}{\d t^N}\right|\le|\mathsf{A}_{j,N+1}(\xi)|+ \frac{1}{2\pi} \int_{\sigma-\im \infty}^{\sigma+\im \infty}\left|\e^{u t} u^N R_{j,N+2}(u, \xi)\right||\d u|
\\ & \le \bigg( \frac{1}{{N + n + L + 2}} + \frac{1}{{\mathfrak{a}_n d}}\e^{\sigma t} \frac{1}{{2\pi }}\int_{\sigma  - \im\infty }^{\sigma  + \im\infty } {\frac{{\left| {\d u} \right|}}{{\left| u \right|^2 }}}  \bigg)c_5 \frac{1}{{\mathfrak{a}_n^{N + 1} }}\frac{{(N + n + L + 2)!}}{{d^{N + 1} }}\sum\limits_{\ell  \ne j} {\frac{1}{{(\max (1,\Re (\xi\e^{ - \im\theta _{j,\ell } } )))^\rho}}} 
\\ & \le \left( {1 + \frac{1}{{\mathfrak{a}_n d}}\frac{1}{{2\sigma }}} \right)c_5 \frac{1}{{\mathfrak{a}_n^{N + 1} }}\frac{{(N + n + L + 2)!}}{{d^{N + 1} }}\e^{\sigma t} \sum\limits_{\ell  \ne j} {\frac{1}{{(\max (1,\Re (\xi\e^{ - \im\theta _{j,\ell } } )))^\rho}}} 
\end{align*}
for any $t\ge 0$, $\xi \in \Gamma_j(d)$, and $N\ge 0$. Accordingly,
\[
\limsup_{N \to  + \infty } \left| \frac{1}{N!}\frac{\d^N F_j (t,\xi )}{\d t^N } \right|^{1/N}  \le \frac{1}{\mathfrak{a}_n d}.
\]
Therefore, according to the Cauchy--Hadamard theorem, for any $t_0\ge 0$, the Taylor series
\begin{equation}\label{eq32}
\sum\limits_{m = 0}^\infty  \frac{1}{m!}\left[\frac{\d^m F_j (t,\xi )}{\d t^m } \right]_{t = t_0 } (t - t_0 )^m 
\end{equation}
converge absolutely within the open disc $|t-t_0|<\mathfrak{a}_n d$. As a result, the power series \eqref{eq32} provide the analytic continuation of the $F_j(t, \xi)$ into the domain $U(\mathfrak{a}_n d)$. Moreover, based on the established bounds for the $N$th derivatives, we deduce the uniform convergence of series \eqref{eq32} on $\Gamma_j(d)$. Consequently, the functions $F_j(t, \xi)$ are analytic with respect to $\xi$ for every fixed $t \in U(\mathfrak{a}_n d)$. Given their continuity in $(t, \xi)$ and analyticity in each respective variable, these functions are analytic functions within $U(\mathfrak{a}_n d) \times \Gamma_j(d)$.

By taking $t=0$ and letting $\sigma \to +\infty$ in \eqref{eq33}, it follows that
\[
\left[ \frac{\d^N F_j (t,\xi )}{\d t^N } \right]_{t = 0} = \mathsf{A}_{j,N+1}(\xi),
\]
thus confirming the validity of the power series expansions in \eqref{eq34}.

To establish the exponential-type bounds \eqref{Fbound} on $|F_j (t,\xi)|$, we can proceed as follows. Assume $0 < r < d$ and $t \in U(\mathfrak{a}_n r)$. When $\Re(t) \ge 0$, the following inequalities hold:
\begin{align*}
& \left| F_j (t,\xi ) \right|  \le \sum\limits_{m = 0}^\infty  \frac{1}{{m!}}\left| \left[ \frac{\d^m F_j (t,\xi )}{\d t^m } \right]_{t = \Re(t)} \right|\left| {t - \Re (t)} \right|^m 
\\ & \le \left( {1 + \frac{1}{\mathfrak{a}_n d}\frac{1}{{2\sigma }}} \right)\frac{{c_5 }}{{\mathfrak{a}_n d}}\e^{\sigma \Re(t)} \sum\limits_{\ell  \ne j} {\frac{1}{{(\max (1,\Re(\xi \e^{ - \im\theta _{j,\ell } } )))^\rho  }}} \sum\limits_{m = 0}^\infty  {\frac{{(m + n + L + 2)!}}{{m!}}\left( {\frac{{\left| {t - \Re(t)} \right|}}{{\mathfrak{a}_n d}}} \right)^m } 
\\ & \le \left( {1 + \frac{1}{{\mathfrak{a}_n d}}\frac{1}{{2\sigma }}} \right)\frac{{c_5 }}{{\mathfrak{a}_n d}}\e^{\sigma \Re(t)} \sum\limits_{\ell  \ne j} {\frac{1}{{(\max (1,\Re(\xi \e^{ - \im\theta _{j,\ell } } )))^\rho  }}} \sum\limits_{m = 0}^\infty  {\frac{{(m + n + L + 2)!}}{{m!}}\left( {\frac{r}{d}} \right)^m } 
\\ & = \bigg( \left( {1 + \frac{1}{{\mathfrak{a}_n d}}\frac{1}{{2\sigma }}} \right)\frac{c_5 }{\mathfrak{a}_n d}\sum\limits_{m = 0}^\infty  \frac{{(m + n + L + 2)!}}{{m!}}\left( {\frac{r}{d}} \right)^m  \bigg)\e^{\sigma \Re(t)} \sum\limits_{\ell  \ne j} \frac{1}{{(\max (1,\Re(\xi \e^{ - \im\theta _{j,\ell } } )))^\rho  }} ,
\end{align*}
where the infinite series is convergent due to the constraint $0 < r < d$. For $\Re(t) < 0$, a similar series of estimates can be repeated, replacing $\Re(t)$ with $0$. Overall, it is found that with the choice
\[
K(\sigma ,r) = \bigg( \left( {1 + \frac{1}{{\mathfrak{a}_n d}}\frac{1}{{2\sigma }}} \right)\frac{{c_5 }}{{\mathfrak{a}_n d}}\sum\limits_{m = 0}^\infty \frac{{(m + n + L + 2)!}}{{m!}}\left( {\frac{r}{d}} \right)^m  \bigg)\e^{\sigma \mathfrak{a}_n r} ,
\]
the estimates \eqref{Fbound} hold for $(t, \xi) \in U(\mathfrak{a}_n r) \times \Gamma_j(d)$.

The asymptotic expansions in \eqref{eq35} follow directly from Proposition \ref{prop2}. This concludes the proof of Theorem \ref{thm1}.

\section{An application}\label{application}

In this section, we demonstrate the applicability our theory by employing it to the following $n^{\text{th}}$-order ($n\ge 2$) Airy-type equation with a large parameter $u$:
\begin{equation}\label{eq36}
\bigg(  - \frac{\d^n }{\d z^n } +  u^n z\bigg)w(u,z) = 0.
\end{equation}
When $n=2$, equation \eqref{eq36} reduces to the standard Airy equation with a large parameter (see, for instance, \cite{Takei2017}). We choose the domain $\mathbf{D}$ to be the Riemann surface associated with the complex $n^{\text{th}}$ root. Application of the transformation
\[
\xi  = \int_0^z t^{1/n} \d t =\frac{n}{n+1}z^{1+1/n},\quad W(u,\xi) = z^{(1-1/n)/2} w(u,z)
\]
brings \eqref{eq36} into the standard form \eqref{eq3}, where
\begin{gather}\label{eq37}
\begin{split}
\psi _k (u,\xi )&= \delta _{0,k}  - \frac{1}{{u^{n - k} }}\frac{(-1)^{n-k}}{{z^{(n - k)(1 + 1/n)} }}\sum\limits_{p = k}^n {\binom{n}{p}\left( {\tfrac{1}{2} - \tfrac{1}{{2n}}} \right)_{n - p} \B_{p,k} \left( {1,-\tfrac{1}{n}, \ldots ,\left( { - \tfrac{1}{n}} \right)_{p - k} } \right)} 
\\ &=\delta _{0,k}  - \frac{1}{{u^{n - k} }}\frac{1}{{\xi ^{n - k} }}\left( {\frac{-n}{{n + 1}}} \right)^{n - k} \sum\limits_{p = k}^n {\binom{n}{p}\left( {\tfrac{1}{2} - \tfrac{1}{{2n}}} \right)_{n - p} \B_{p,k} \left( {1,-\tfrac{1}{n}, \ldots , \left( { - \tfrac{1}{n}} \right)_{p - k} } \right)} 
\end{split}
\end{gather}
(cf. \eqref{eq45} and \eqref{eq46}). Here $(w)_m =\Gamma(w+m)/\Gamma(w)$ denotes the Pochhammer symbol. The domain $\mathbf{D}$ is mapped into the Riemann surface $\mathbf{G}$ of the complex $(n+1)^{\text{th}}$ root.

To apply Theorem \ref{thm1}, we have to ensure that Condition \ref{cond} is satisfied. First, we have to define a suitable subdomain $\Delta$ of $\mathbf{G}$. To this end, observe that each $\psi_k(u,\xi)$ becomes unbounded as $|\xi| \to 0$. Hence, it is natural to fix $0 < \varepsilon \ll 1$ and define
\[
\Delta = \{ \xi: \xi \in \mathbf{G},\, |\xi|> \varepsilon \}.
\]
The functions $\psi_k(u,\xi)$ are analytic and bounded throughout $\Delta$,
and Condition \ref{cond} holds true with $\rho = 1$, and any $\sigma' > 0$ and $d' > 0$. (Note that the Borel transforms $\Psi_k(t,\xi)$ are polynomial functions in $t$.)

In what follows, we focus primarily on the solution $W_0(u,\xi)$. A similar approach can be used to address the remaining $n-1$ solutions. The angles $\theta_{0,\ell}$ and $\theta_0$ are chosen as follows:
\[
\theta _{0,\ell}  =\frac{(2+\sgn (\ell))n - 2\ell}{2n}\pi,\quad \theta_0=\pi.
\]
For an arbitrary $d>0$, we specify the subdomain $\Gamma_0(d)$ of $\Delta$ as follows.
\begin{enumerate}[(i)]\itemsep0.5em
\item A branch cut is introduced along the positive real line on the principal sheet of $\mathbf{G}$, and the phase of $\xi$ is specified such that $0 < \arg \xi < 2\pi$.
\item Subsequently, we remove the sectors $0 < \arg \xi \le \frac{n-2}{2n}\pi$ and $\frac{3n+2}{2n}\pi \le \arg \xi <2\pi$.
\item Finally, the closed $d + \varepsilon$ neighbourhoods of the rays $\arg \xi =\frac{n-2}{2n}\pi$ and $\arg \xi =\frac{3n+2}{2n}\pi$ are deleted.
\end{enumerate}
We leave it to the reader to verify that $\Gamma_0(d)$ fulfils the requirements outlined in Section \ref{section1}. In the special case that $n=3$, this domain, along with its corresponding domain $D_0(d)$ on $\mathbf{D}$, is depicted in Figure \ref{Figure2}.

\begin{figure}[t]
\begin{subfigure}[c]{0.42\textwidth}
\centering
\includegraphics[width=\textwidth]{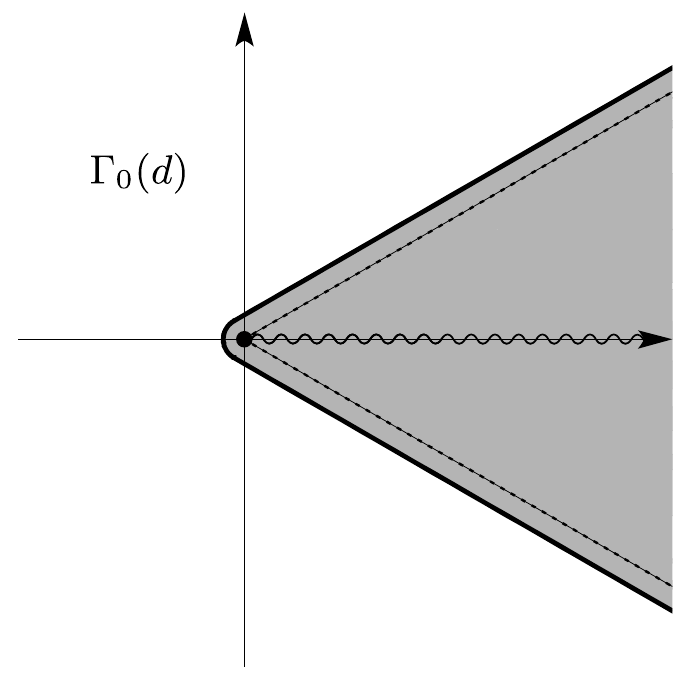}
\caption{}
\end{subfigure}
\hfill
\begin{subfigure}[c]{0.42\textwidth}
\centering 
\includegraphics[width=\textwidth]{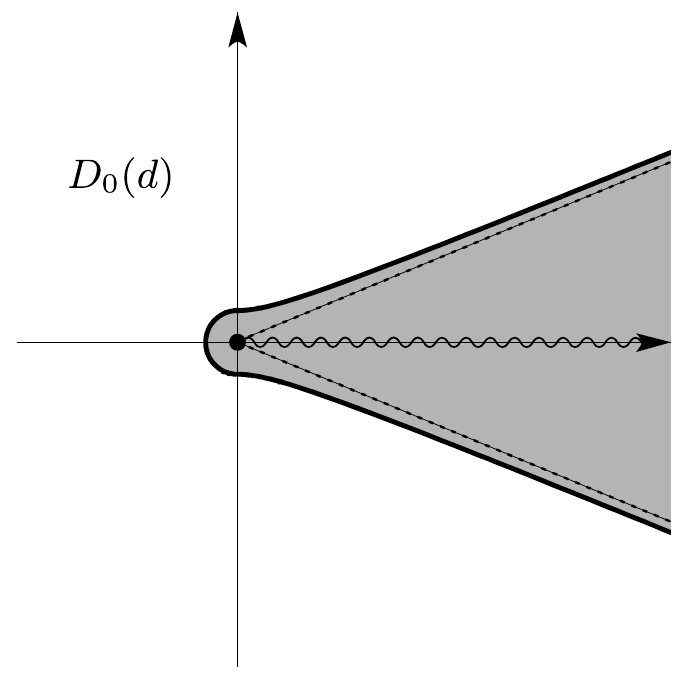}
\caption{}
\end{subfigure}
\caption{(a) Domain $\Gamma_0(d)$ (unshaded). (b) Domain $D_0(d)$ (unshaded).}
\label{Figure2}
\end{figure}

Employing Theorem \ref{thm1}, we find that the differential equation \eqref{eq3}, with the $\psi_k(u,\xi)$ as given in \eqref{eq37}, has a solution of the form
\begin{equation}\label{eq41}
W_0(u,\xi) =\e^{u\xi}\big( 1 + \eta_j(u,\xi) \big),
\end{equation}
which is analytic in $\left\{u : \Re(u) > 0\right\} \times \Gamma_0(d)$, and satisfies the limit conditions \eqref{eq19}. Furthermore, for each $\sigma > 0$,
\begin{align*}
W_0(u,\xi) & \sim \e^{u\xi} \bigg( 1 + \sum\limits_{m = 1}^\infty  \frac{\mathsf{A}_{0,m} (\xi )}{u^m } \bigg)
\\ &= \e^{u\xi} \bigg(1 + \frac{{(n - 1)(2n + 1)}}{{24(n + 1)\xi }}\frac{1}{u} + \frac{{(n - 1)(2n + 1)(2n^2  + 23n + 23)}}{{1152(n + 1)^2 \xi ^2 }}\frac{1}{{u^2 }} +  \ldots \bigg)
\end{align*}
as $u \to \infty$ in the half-plane $\Re(u) \ge \sigma$, uniformly with respect to $\xi \in \Gamma_0(d)$. This asymptotic expansion is Borel-summable, and the Borel transform $F_0(t,\xi)$ is an analytic function in $U(\mathfrak{a}_n d) \times \Gamma_0(d)$. The remaining $n-1$ solutions can be chosen as $W_j(u,\xi)=W_0(u, \xi\e^{2\pi\im j/n})$ for $\xi \in \Gamma_j(d) = \e^{-2\pi\im j/n}\Gamma_0(d)$ with $j = 1,2,\ldots,n-1$.

It is also possible to address this specific problem using the method of steepest descents for integrals \cite{Bennett2018}. We believe it is constructive to include some details of this approach and compare the results with those obtained via Theorem \ref{thm1}. We observe that the function
\begin{equation}\label{eq38}
y(x) = \sqrt {\frac{n}{2\pi }} \int_{ + \infty }^{\infty \e^{2\pi \im/(n + 1)} } \exp \left(  - \frac{t^{n + 1}}{n + 1} + xt \right)\d t 
\end{equation}
is an entire function of the complex variable $x$ and satisfies the $n^{\text{th}}$-order equation
\[
\bigg(  - \frac{\d^n }{\d x^n } + x\bigg)y(x) = 0.
\]
The prefactor $\sqrt{n/(2\pi)}$ in \eqref{eq38} is introduced for later convenience. By defining a new variable $\zeta$ using the equation
\[
\zeta  = \frac{n}{n + 1}x^{1 + 1/n},
\]
the function $y_0(\zeta)=y\big(\big(\frac{n+1}{n}\zeta\big)^{1-1/(n+1)}\big)$ becomes an analytic function of $\zeta$ on the Riemann surface associated with the complex $(n+1)^{\text{th}}$ root. Assuming $\arg \zeta = \pi$, we perform a change of integration variables in \eqref{eq38} from $t$ to $s$ via
\[
t = \left( \frac{n + 1}{n}\zeta  \right)^{1/(n + 1)} s.
\]
This leads us to the representation
\[
y_0(\zeta ) = \sqrt {\frac{n}{2\pi}} \left( \frac{n + 1}{n}\zeta\right)^{1/(n + 1)} \e^\zeta  \int_{\infty \e^{ - \pi \im/(n + 1)} }^{\infty \e^{\pi \im/(n + 1)} } \exp \left(  - \zeta \left( \frac{s^{n + 1}}{n} - \frac{n + 1}{n}s + 1 \right) \right)\d s.
\]
Using Cauchy's theorem, we deform the integration contour to traverse through the saddle point at $s = 1$. Thus,
\[
y_0(\zeta ) = \sqrt {\frac{n}{2\pi}} \left( \frac{n + 1}{n}\zeta\right)^{1/(n + 1)}   \e^\zeta  \int_{\mathscr{C}(\arg \zeta)} \exp \left(  - \zeta \left( \frac{s^{n + 1}}{n} - \frac{n + 1}{n}s + 1 \right) \right)\d s,
\]
where $\mathscr{C}(\arg \zeta)$ denotes the path of steepest descent passing through the saddle point at $s = 1$. Varying the phase of $\zeta$ results in an analytic continuation of $y_0(\zeta)$. This continuation remains valid as long as the path $\mathscr{C}(\arg \zeta)$ varies continuously with respect to $\arg \zeta$. When $n=3$, Figure \ref{Figure3} illustrates the path $\mathscr{C}(\arg \zeta)$ for different values of $\arg \zeta$. When $\arg \zeta =\frac{n-2}{2n}\pi$, the path connects to the saddle point $\e^{2\pi \im/n}$. Likewise, when $\arg \zeta =\frac{3n+2}{2n}\pi$, the path connects to the saddle point $\e^{2\pi \im(n-1)/n}$. Consequently, $\arg \zeta =\frac{n-2}{2n}\pi$ and $\arg \zeta =\frac{3n+2}{2n}\pi$ are Stokes lines for $y_0(\zeta)$. An application of the method of steepest descents then yields the asymptotic expansion
\begin{gather}\label{eq40}
\begin{split}
y_0(\zeta ) & \sim \left( \frac{n + 1}{n}\zeta\right)^{1/(n + 1)-1/2}\e^\zeta \sum_{m=0}^\infty \frac{a_m}{\zeta^m} = x^{(1/n-1)/2} \e^\zeta \sum_{m=0}^\infty \frac{a_m}{\zeta^m}
\\ & = x^{(1/n-1)/2} \e^{\zeta} \bigg(1 + \frac{{(n - 1)(2n + 1)}}{{24(n + 1) }}\frac{1}{\zeta} + \frac{{(n - 1)(2n + 1)(2n^2  + 23n + 23)}}{{1152(n + 1)^2 }}\frac{1}{{\zeta^2 }} +  \ldots \bigg),
\end{split}
\end{gather}
as $\zeta\to\infty$ in the sector $-\frac{\pi}{n}+\delta\le \arg \zeta \le \frac{2n+1}{n}\pi-\delta$, with any fixed $\delta>0$. The coefficients $a_m$ can be expressed using Perron's formula:
\[
a_m = \frac{1}{(2n + 2)^m m!}\left[\frac{\d^{2m} }{\d s^{2m}}\left(\frac{n(n + 1)}{2} \frac{(s - 1)^2 }{s^{n + 1} - (n + 1)s+ n} \right)^{m + 1/2}  \right]_{s= 1}
\]
(see, e.g., \cite[Eq. 11]{Bennett2018}).

\begin{figure}[t]
\centering
\includegraphics[width=0.4\textwidth]{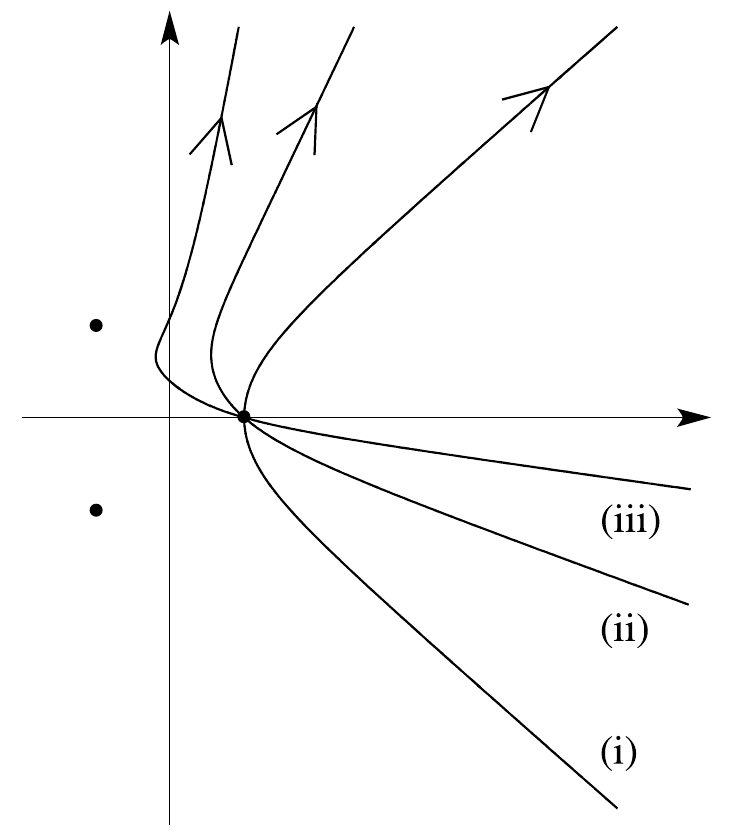}
\caption{The steepest descent contour $\mathscr{C}(\arg \zeta)$ through the saddle point $s=1$ when $n=3$ and (i)
$\arg \zeta =\pi$, (ii) $\arg \zeta =\frac{\pi}{2}$ and (iii) $\arg \zeta =\frac{\pi}{5}$. The bullet points represent the saddles $1$, $\e^{2\pi \im /3}$ and $\e^{4\pi \im /3}$.}
\label{Figure3}
\end{figure}

We can readily verify that the function $w(u,z) = y(u^{n/(n+1)} z)$ satisfies equation \eqref{eq36}. Consequently,
\[
W_0 (u,\xi ) = z^{(1-1/n)/2}w(u,z)=u^{1/2-1/(n + 1)}  z^{(1-1/n)/2}  y_0(u\xi )
\]
coincides with the solution \eqref{eq41}
obtained from Theorem \ref{thm1}. From \eqref{eq40},
\begin{equation}\label{eq39}
W_0 (u,\xi ) \sim \e^{u\xi}\sum\limits_{m = 0}^\infty \frac{a_m}{(u \xi)^m},
\end{equation}
as $u\xi \to\infty$ within the sector $-\frac{\pi}{n}+\delta\le \arg(u\xi) \le \frac{2n+1}{n}\pi-\delta$, for any fixed $\delta>0$. In particular, this implies $\mathsf{A}_{0,m} (\xi ) = a_m \xi ^{ - m}$. Therefore, for the asymptotic expansion \eqref{eq39} to remain valid as $u \to \infty$ within any half-plane $\Re(u)\ge \sigma>0$, $\xi$ must be confined within a sector $\frac{n-2}{2n}\pi+\delta\le\arg\xi\le\frac{3n+2}{2n}\pi-\delta$, with any fixed $\delta>0$, and $|\xi|$ must be bounded away from $0$. Such a set is properly contained in a $\Gamma_0(d)$, resulting from Theorem \ref{thm1} with a suitably chosen $d>0$.

\section{Discussion}\label{discussion}

We studied the Borel summability of formal solutions of certain $n^{\text{th}}$-order linear ordinary differential equations with a large parameter. We demonstrated that, given mild conditions on the potential functions of the equation, the formal solutions are Borel summable with respect to the large parameter in vast, unbounded domains of the independent variable. We established that the formal series expansions serve as asymptotic expansions, uniform with respect to the independent variable, for the Borel re-summed exact solutions. Additionally, we showed that the exact solutions can be expressed using factorial series in the parameter, and these expansions converge in half-planes, uniformly in the independent variable.

The theory presented here marks an initial step toward a global analysis of formal solutions of the differential equation \eqref{eq1}. The primary result of this paper demonstrates the Borel summability of formal solutions within a Stokes region, provided that the independent variable remains bounded away from Stokes curves emerging either from a zero of $f_{0,0}(z)$ or from a singularity of one of the functions $f_k(u,z)$. Importantly, it does not provide any information regarding connection formulae joining formal solutions across the Stokes curves.

In the following discussion, we briefly explore a potential extension of the theory to turning point problems. Turning points represent the simplest form of transition points. For illustration, consider the following third-order differential equation:
\begin{equation}\label{eq49}
\bigg(  - \frac{\d^3 }{\d z^3 } +  u^3 f(z)\bigg)w(u,z) = 0.
\end{equation}
Suppose that the potential function $f(z)$ is analytic in a domain $\mathbf{D}$ and $f(z)$ vanishes at exactly one point $z_0$, say, of $\mathbf{D}$. For the sake of simplicity, we assume that $z_0$ is a simple zero of $f(z)$. It is convenient to work in terms of the transformed variables $\zeta$ and $\mathcal{W}(u, \zeta)$ given by
\begin{equation}\label{eq50}
\frac{3}{4}\zeta^{4/3}  = \int_{z_0 }^z f^{1/3} (t)\d t ,\quad  \mathcal{W}(u,\zeta) = \zeta^{-1/3}f^{1/3}(z)w(u,z).
\end{equation}
Subsequently, equation \eqref{eq49} transforms into
\begin{equation}\label{eq51}
\bigg( -\frac{\d^3}{\d\zeta ^3 } + \phi(\zeta)\frac{\d}{\d\zeta } + u^3 \zeta  + \psi(\zeta)\bigg)\mathcal{W}(u,\zeta )=0,
\end{equation}
with
\[
 \phi (\zeta ) = \bigg( \frac{\d\widehat{f}^{-1/3} (z)}{{\d z}} \bigg)^2  -  2 \widehat{f}^{-1/3} (z) \frac{\d^2\widehat{f}^{-1/3} (z) }{\d z^2}, \quad \psi (\zeta ) = -  \widehat{f}^{-2/3} (z) \frac{{\d^3 \widehat{f}^{-1/3} (z)}}{{\d z^3 }}
\]
and $\widehat{f}(z) = \zeta ^{ - 1} f (z)$. The functions $\phi (\zeta )$ and $\psi (\zeta )$ are analytic in the corresponding $\zeta$ domain $\mathbf{H}$, say. The transformation \eqref{eq50} maps the simple turning point $z_0$ into the origin in $\mathbf{H}$, and the four Stokes curves issuing from $z_0$ are mapped into the rays $\arg \zeta  = 0, \frac{\pi }{2}\im, \pi \im,\frac{3\pi}{2}\im$. Let $Y(x)$ be any solution of the third-order Airy-type equation
\begin{equation}\label{eq53}
\bigg(  - \frac{\d^3}{\d x^3} + x\bigg)Y(x) = 0.
\end{equation}
It can be verified by direct substitution that equation \eqref{eq51} has a formal solution of the form
\begin{equation}\label{eq52}
\mathcal{W}(u,\zeta ) = Y(u^{3/4} \zeta )\sum\limits_{m = 0}^\infty  \frac{A_m (\zeta )}{u^{3m} } + u^{-9/4} Y'(u^{3/4} \zeta )\sum\limits_{m = 0}^\infty \frac{B_m (\zeta )}{u^{3m} } + u^{-3/2} Y''(u^{3/4} \zeta )\sum\limits_{m = 0}^\infty \frac{C_m (\zeta )}{u^{3m}} .
\end{equation}
The coefficients $A_m (\zeta)$, $B_m (\zeta)$ and $C_m (\zeta)$, which are analytic functions of $\zeta$ in $\mathbf{H}$, can be determined recursively via
\begin{align*}
A_0 (\zeta ) & = 1,\quad C_0 (\zeta ) = \frac{1}{3}\zeta ^{ - 2/3} \int_0^\zeta t^{ - 1/3} \phi (t)\d t ,\\
\\ B_0 (\zeta ) & =  - \frac{1}{3}\zeta ^{ - 1/3} \int_0^\zeta  t^{ - 2/3} \left( 3(tC'_0 (t))' - \phi (t)tC_0 (t) - \psi (t) \right) \d t ,
\end{align*}
and
\begin{align*}
 A_m (\zeta ) =  &- \frac{1}{3}\int^\zeta  \left( 3B''_{m-1} (t) - \phi (t)B_{m-1} (t) + C'''_{m-1} (t) - \phi (t)C'_{m-1} (t) - \psi (t)C_{m-1} (t) \right)\d t, \\
C_m (\zeta ) = & - \frac{1}{3}\zeta ^{ - 2/3} \int_0^\zeta  t^{ - 1/3} \left( 3A''_m (t) - \phi (t)A_m (t) + B'''_{m-1} (t) - \phi (t)B'_{m-1} (t) - \psi (t)B_{m-1} (t) \right)\d t,
\\ B_m (\zeta ) =  &- \frac{1}{3}\zeta ^{ - 1/3} \int_0^\zeta  t^{ - 2/3} \left( A'''_m (t) - \phi (t)A'_m (t) - \psi (t)A_m (t)+ 3(tC'_m (t))' - \phi (t)tC_m (t) \right)\d t,
\end{align*}
for $m\ge 1$. The constants of integration in the expression for $A_m (\zeta )$ can be chosen arbitrarily. We anticipate that, assuming appropriate conditions on the functions $\phi(\zeta)$ and $\psi(\zeta)$, and for sufficiently large values of $\Re(u)$, the three formal series in \eqref{eq52} are Borel-summable within a specific, well-defined subdomain of $\mathbf{H}$ that includes $\zeta=0$ and the associated Stokes rays $\arg \zeta  = 0, \frac{\pi }{2}\im, \pi \im,\frac{3\pi}{2}\im$. This could potentially be confirmed by employing an argument similar to that used in the second-order case \cite{Dunster2001}. As a result, the continuation formulae for formal solutions of the form \eqref{eq52} can be derived directly from those corresponding to the solutions of \eqref{eq53} (see, e.g., \cite{Ohyama1995}). Subsequently, by employing the transformation 
\[
\xi = \frac{3}{4}\zeta^{4/3} ,\quad  W(u,\xi)= \zeta^{1/3}\mathcal{W}(u,\zeta),
\]
connection formulae can be established for formal solutions of the form \eqref{eq4}.

Finally, extending the results of the paper to differential equations of the type \eqref{eq1}, where the terms corresponding to $m = 0$ in the expansions \eqref{eq23} of the functions $f_k(u, z)$ (for $k \neq 0$) need not be zero, would be of great interest. We anticipate that the emergence of potential new Stokes curves could pose challenges to the analysis.

\section*{Acknowledgement} The author would like to thank the referee for the careful and thorough review of the paper, which has significantly improved its presentation. This research was conducted in part during a visit to the Okinawa Institute of Science and Technology (OIST) as part of the Theoretical Sciences Visiting Program (TSVP). The author's research was supported by the JSPS KAKENHI Grants No. JP21F21020 and No. 22H01146.

\appendix

\section{The transformation of the equation}\label{appA}

In this appendix, we derive expressions that relate the functions $\psi_k(u, \xi)$ and $\psi_{k,m}(\xi)$ to their respective counterparts $f_k(u, z)$ and $f_{k,m}(z)$.
The basis for the derivation is the following lemma.

\begin{lemma}\label{lemma8}
For any $p\ge 1$, the following identity between differential operators holds:
\begin{equation}\label{eq42}
\frac{\d^p }{\d z^p } = \sum\limits_{r = 1}^p \B_{p,r} \bigg( f_{0,0}^{1/n} (z),\frac{\d f_{0,0}^{1/n} (z)}{\d z}, \ldots ,\frac{\d^{p - r} f_{0,0}^{1/n} (z)}{\d z^{p - r}}  \bigg)\frac{\d^r }{\d\xi ^r }.
\end{equation}
\end{lemma}

\begin{proof} We proceed by induction on $p$. The base case, when $p=1$, is established through the chain rule and the equality
\[
\frac{\d \xi}{\d z}=f_{0,0}^{1/n} (z).
\]
This equality is a direct consequence of \eqref{eq2}. Suppose the identity \eqref{eq42} holds for a given positive integer $p$. Using both the induction hypothesis and the base case, and applying the identity \eqref{Bdiff} and the recurrence relation \eqref{Brec1}, we deduce
\begin{align*}
\frac{\d ^{p + 1}}{\d z^{p + 1}} & =\frac{\d}{\d z}\frac{\d ^p }{\d z^p }= \sum\limits_{r = 1}^p \frac{\d \B_{p,r}}{{\d z}} \frac{{\d ^r }}{{\d \xi ^r }}  + \sum\limits_{r = 1}^p {f_{0,0}^{1/n} (z)\B_{p,r} \frac{{\d ^{r + 1} }}{{\d \xi ^{r + 1} }}}
\\ & = \sum\limits_{r = 1}^p {\bigg( {\sum\limits_{q = 1}^{p - r + 1} {\binom{p}{q}\frac{{\d ^q f_{0,0}^{1/n} (z)}}{{\d z^q }}\B_{p - q,r - 1} } } \bigg)\frac{{\d ^r }}{{\d \xi ^r }}}  + \sum\limits_{r = 1}^{p + 1} {f_{0,0}^{1/n} (z)\B_{p,r - 1} \frac{{\d ^r }}{{\d \xi ^r }}} 
\\ & = f_{0,0}^{1/n} (z)\B_{p,p} \frac{{\d ^{p + 1} }}{{\d \xi ^{p + 1} }} + \sum\limits_{r = 1}^p {\bigg( f_{0,0}^{1/n} (z)\B_{p,r - 1}  + \sum\limits_{q = 1}^{p - r + 1} {\binom{p}{q}\frac{{\d ^q f_{0,0}^{1/n} (z)}}{{\d z^q }}\B_{p - q,r - 1} } \bigg)\frac{{\d ^r }}{{\d \xi ^r }}} 
\\ & = \B_{p + 1,p + 1} \frac{{\d ^{p + 1} }}{{\d \xi ^{p + 1} }} + \sum\limits_{r = 1}^p {\bigg( \sum\limits_{q = 0}^{p - r + 1} {\binom{p}{q}\frac{{\d ^q f_{0,0}^{1/n} (z)}}{{\d z^q }}\B_{p - q,r - 1} }  \bigg)\frac{{\d ^r }}{{\d \xi ^r }}} 
\\ & = \B_{p + 1,p + 1} \frac{{\d ^{p + 1} }}{{\d \xi ^{p + 1} }} + \sum\limits_{r = 1}^p {\bigg( \sum\limits_{q = 1}^{p + 1 - r + 1} {\binom{p+1-1}{q-1}\frac{{\d ^{q - 1} f_{0,0}^{1/n} (z)}}{{\d z^{q - 1} }}\B_{p + 1 - q,r - 1} }  \bigg)\frac{{\d ^r }}{{\d \xi ^r }}} 
\\ & = \B_{p + 1,p + 1} \frac{{\d ^{p + 1} }}{{\d \xi ^{p + 1} }} + \sum\limits_{r = 1}^p {\B_{p + 1,r} \frac{{\d ^r }}{{\d \xi ^r }}}  = \sum\limits_{r = 1}^{p + 1} {\B_{p + 1,r} \frac{{\d ^r }}{{\d \xi ^r }}} .
\end{align*}
This concludes the induction step.
\end{proof}
Upon differentiating the equality $w(u,z)= f_{0,0}^{(1/n - 1)/2} (z)W(u,\xi )$ $k$ times with respect to $z$ and applying Lemma \ref{lemma8}, we obtain
\begin{align*}
\frac{\d^k w(u,z)}{\d z^k } & = \sum\limits_{p = 0}^k \binom{k}{p}\frac{\d^{k - p} f_{0,0}^{(1/n - 1)/2} (z)}{\d z^{k - p} }\frac{\d^p W(u,\xi )}{\d z^p } \\ & = \frac{\d^k f_{0,0}^{(1/n - 1)/2} (z)}{\d z^k } W(u,\xi)+\sum\limits_{p = 1}^k \binom{k}{p}\frac{\d^{k - p} f_{0,0}^{(1/n - 1)/2} (z)}{\d z^{k - p} }\sum_{r=1}^p \B_{p,r}\frac{\d^p W(u,\xi )}{\d \xi^p }.
\end{align*}
By substituting into \eqref{eq1} and re-arranging the expression, we derive the equation \eqref{eq3} with
\begin{gather}\label{eq45}
\begin{split}
\psi _k (u,\xi ) = & - \frac{{u^{k - n} }}{{f_{0,0}^{(1 + 1/n)/2} (z)}}\sum\limits_{p = k}^n {\binom{n}{p}\frac{{\d^{n - p} f_{0,0}^{(1/n - 1)/2} (z)}}{{\d z^{n - p} }}} \B_{p,k} \\ & + \frac{1}{{f_{0,0}^{(1 + 1/n)/2} (z)}}\sum\limits_{p = 0}^{n - 2} {u^{k - p} f_p (u,z)\sum\limits_{r = k}^p {\binom{p}{r}\frac{{\d^{p - r} f_{0,0}^{(1/n - 1)/2} (z)}}{{\d z^{p - r} }}\B_{r,k} } } .
\end{split}
\end{gather}
Substituting the expansions \eqref{eq23} into this expression and re-arranging the resulting equation in descending powers of $u$, we derive the following formulae for the coefficients $\psi_{k,m}(\xi)$. For $k=0$, we obtain
\[
\psi _{0,1} (\xi) = \frac{{f_{0,1} (z)}}{{f_{0,0} (z)}},
\]
\[
\psi _{0,n} (\xi ) = \frac{{f_{0,n} (z)}}{{f_{0,0} (z)}} - \frac{1}{f_{0,0}^{(1 + 1/n)/2} (z)}\bigg(\frac{{\d^n f_{0,0}^{(1/n - 1)/2} (z)}}{{\d z^n }} - \sum\limits_{p = 1}^{n - 2} {f_{p,n - p} (z)\frac{{\d ^p f_{0,0}^{(1/n - 1)/2} (z)}}{{\d z^p }}}\bigg),
\]
and
\[
\psi _{0,m} (\xi ) = 
\frac{{f_{0,m} (z)}}{{f_{0,0} (z)}} + \frac{1}{f_{0,0}^{(1 + 1/n)/2} (z)}\sum\limits_{p = 1}^{\min (n - 2,m - 1)} {f_{p,m - p} (z)\frac{{\d^p f_{0,0}^{(1/n - 1)/2} (z)}}{{\d z^p }}}
\]
for $n \ne m \ge 2$. For $1\le k\le n-2$, we find
\begin{align*} 
\psi _{k,n - k} (\xi ) = \; & \frac{1}{f_{0,0}^{(1 + 1/n)/2} (z)}\sum\limits_{p = 0}^{n - k - 2} {f_{k + p,n - k - p} (z)\sum\limits_{r = 0}^p {\binom{k+p}{k+r}\frac{{\d^{p - r} f_{0,0}^{(1/n - 1)/2} (z)}}{{\d z^{p - r} }}} \B_{k + r,k} } \\ & - \frac{1}{f_{0,0}^{(1 + 1/n)/2} (z)}\sum\limits_{r = k}^n {\binom{n}{r}\frac{{\d^{n - r} f_{0,0}^{(1/n - 1)/2} (z)}}{{\d z^{n - r} }}\B_{r,k} } ,
\end{align*}
and
\[
\psi _{k,m} (\xi ) = \frac{1}{f_{0,0}^{(1 + 1/n)/2} (z)}\sum\limits_{p = 0}^{\min (n - k - 2,m - 1)} {f_{k + p,m - p} (z)\sum\limits_{r = 0}^p {\binom{k+p}{k+r}\frac{{\d^{p - r} f_{0,0}^{(1/n - 1)/2} (z)}}{{\d z^{p - r} }}} \B_{k + r,k} } 
\]
for $n - k \ne m \ge 1$.

\section{Recurrence relation for the coefficients}\label{coeffappendix}

In this appendix, we outline the derivation of the recurrence relation \eqref{eq6} satisfied by the coefficients $\mathsf{A}_{j,m} (\xi )$. Letting $\mathsf{A}_{j,0} (\xi ) = 1$ and $\B_r  = \B_r  ( X'_j (\xi ),X''_j (\xi ), \ldots ,X_j^{(r)} (\xi ) )$, we employ \eqref{Bprop} to obtain the following expansions for the derivatives of the formal solutions \eqref{eq4}:
\begin{align*}
& \exp \Big( { - \e^{2\pi \im j/n} u\xi  - X_j (\xi )} \Big)\frac{{\d^k }}{{\d\xi ^k }}\bigg( {\exp \Big( {\e^{2\pi \im j/n} u\xi  + X_j (\xi )} \Big)\sum\limits_{m = 0}^\infty  {\frac{{\mathsf{A}_{j,m} (\xi )}}{{u^m }}} } \bigg)
\\ &
 = \sum\limits_{p = 0}^k {\binom{k}{p}\bigg( {\sum\limits_{r = 0}^p {\binom{p}{r}\e^{2\pi \im j(p - r)/n} u^{p - r} \B_r } } \bigg)\bigg( {\sum\limits_{m = 0}^\infty  {\frac{{\d^{k - p} \mathsf{A}_{j,m} (\xi )}}{{\d\xi ^{k - p} }}\frac{1}{{u^m }}} } \bigg)} 
\\ &
 = \sum\limits_{m = 0}^\infty  {\bigg( {\sum\limits_{p = 0}^k {\e^{2\pi \im jp/n} \bigg( {\sum\limits_{r = p}^k {\binom{k}{r}\binom{r}{r-p}\B_{r - p} \frac{{\d^{k - r} \mathsf{A}_{j,m} (\xi )}}{{\d\xi ^{k - r} }}} } \bigg)u^p } } \bigg)\frac{1}{{u^m }}} 
\\ &
 = u^k \sum\limits_{m = 0}^\infty  {\sum\limits_{p = 0}^k {\e^{2\pi \im j(k - p)/n} \sum\limits_{r = 0}^p {\binom{k}{r}\binom{k-r}{p-r}\B_{p - r} \frac{{\d^r \mathsf{A}_{j,m} (\xi )}}{{\d\xi ^r }}} \frac{1}{{u^{m + p} }}} } 
\\ & = u^k \sum\limits_{m = 0}^\infty  {\bigg( {\sum\limits_{p = 0}^{\min(m,k)} {\e^{2\pi \im j(k - p)/n} \sum\limits_{r = 0}^{p} {\binom{k}{r}\binom{k-r}{p-r}\B_{p - r} \frac{{\d^r \mathsf{A}_{j,m - p} (\xi )}}{{\d\xi ^r }}} } } \bigg)\frac{1}{{u^m }}} 
\\ & = u^k \sum\limits_{m = 0}^\infty  {\bigg( {\sum\limits_{p = 0}^{\min(m,k)} {\binom{k}{p}\e^{2\pi \im j(k - p)/n} \sum\limits_{r = 0}^{p} {\binom{p}{r}\B_{p - r} \frac{{\d^r \mathsf{A}_{j,m - p} (\xi )}}{{\d\xi ^r }}} } } \bigg)\frac{1}{{u^m }}} .
\end{align*}
Substituting both these expansions and those in \eqref{eq15} into equation \eqref{eq3}, performing the products of the expansions, and equating coefficients of like powers of $u$, we obtain
\begin{multline*}
\sum\limits_{p = 1}^{\min(n,m)} {\binom{n}{p}\e^{2\pi \im j(n - p)/n} \sum\limits_{r = 0}^{p} {\binom{p}{r}\B_{p - r} \frac{{\d^r \mathsf{A}_{j,m - p} (\xi )}}{{\d\xi ^r }}} }  \\ = \sum\limits_{k = 0}^{n - 2} {\sum\limits_{q = 0}^{m - 1} {\sum\limits_{p = 0}^{\min(k,q)} {\binom{k}{p}\e^{2\pi \im j(k - p)/n} \sum\limits_{r = 0}^{p} {\binom{p}{r}\B_{p - r} \psi _{k,m - q} (\xi )\frac{{\d^r \mathsf{A}_{j,q - p} (\xi )}}{{\d\xi ^r }}} } } } .
\end{multline*}
The terms involving $\mathsf{A}_{j,m} (\xi )$ cancel each other out. By solving for $\partial_\xi \mathsf{A}_{j,m-1} (\xi )$, integrating both sides, and substituting $m$ with $m+1$ throughout, we arrive at the desired recurrence relation \eqref{eq6} for $m\ge 1$.

\section{Bell polynomials}\label{Bell}

Due to the frequent occurrences of the Bell polynomials in this paper, we have compiled in this appendix their properties that are used throughout the paper. The partial exponential Bell polynomials $\B_{p,r}$ are a triangular array of polynomials defined by the generating function
\[
\frac{1}{r!}\bigg( \sum\limits_{p = 1}^\infty  x_p \frac{z^p }{p!}  \bigg)^r  = \sum\limits_{p = r}^\infty \B_{p,r} (x_1 ,x_2 , \ldots ,x_{p - r + 1} )\frac{z^p }{p!},\quad r\ge 0.
\]
They can be computed efficiently by a recurrence relation:
\begin{equation}\label{Brec1}
\B_{p + 1,r + 1} (x_1 ,x_2 , \ldots ,x_{p - r + 1} ) = \sum\limits_{q = 0}^{p - r} \binom{p}{q}x_{q + 1} \B_{p - q,r} (x_1 ,x_2 , \ldots ,x_{p - r - q + 1} ),
\end{equation}
where $\B_{0,0}=1$, $\B_{p,0}=0$ for $p\ge 1$, and $\B_{0,r}=0$ for $r\ge 1$. The partial exponential Bell polynomials satisfy the following homogeneity relation:
\begin{equation}\label{eq46}
\B_{p,r} (\alpha \beta x_1 ,\alpha \beta ^2 x_2 , \ldots ,\alpha \beta ^{p - r + 1} x_{p - r + 1} ) = \alpha ^r \beta ^p \B_{p,r} (x_1 ,x_2 , \ldots ,x_{p - r + 1} ).
\end{equation}
By applying the chain rule, one can derive the following identity:
\begin{equation}\label{Bdiff}
\frac{\d}{\d z}\B_{p,r}(x_1(z),x_2(z), \ldots, x_{p-r+1}(z)) = \sum_{q=1}^{p-r+1} \binom{p}{q} x_q'(z) \B_{p-q,r-1}(x_1(z), x_2(z),\ldots, x_{p-q-r+2}(z)).
\end{equation}

The sum
\[
\B_p(x_1 ,x_2 , \ldots ,x_p)=\sum_{r=0}^p \B_{p,r} (x_1 ,x_2 , \ldots ,x_{p - r + 1} )
\]
is called the $p^{\text{th}}$ complete exponential Bell polynomial. These polynomials can be computed recursively via
\begin{equation}\label{Brec2}
\B_{p + 1} (x_1 ,x_2 , \ldots ,x_{p + 1} ) = \sum\limits_{q = 0}^p \binom{p}{q}x_{q + 1} \B_{p - q} (x_1 ,x_2 , \ldots ,x_{p - q} ),
\end{equation}
with the initial value $\B_0=1$. The exponential Bell polynomials play a role in the celebrated Fa\`{a} di Bruno formula, which generalises the chain rule to higher derivatives. In particular, 
\begin{equation}\label{Bprop}
\frac{\d^p \exp (x(z))}{\d z^p} = \exp (x(z))\B_p (x'(z),x''(z), \ldots ,x^{(p)} (z)).
\end{equation}
For further properties of the Bell polynomials, we refer to \cite[Ch. III]{Comtet1974}.

\section{The sequence $\mathfrak{a}_n$}\label{aseq}

In this appendix, we establish certain properties of the sequence $\mathfrak{a}_n$ with $n\geq 2$, where the $n^{\text{th}}$ term is defined as the unique positive solution of the equation
\[
(1+\mathfrak{a}_n)^n=1+2n\mathfrak{a}_n.
\]
The verification that this equation possesses a distinct positive solution for each $n\ge 2$ is left to the reader.

\begin{proposition}\label{athm} Let $a=1.2564312086\ldots$ be the unique positive solution of the equation $\e^a = 1+2a$. Then for any integer $n\ge 2$, the double inequality
\begin{equation}\label{athm1}
\frac{a}{n} \le \mathfrak{a}_n \le \frac{2}{n-1}
\end{equation}
holds. Furthermore, the sequence $\mathfrak{a}_n$ satisfies the asymptotic limit
\begin{equation}\label{athm2}
\mathfrak{a}_n \sim \frac{a}{n},
\end{equation}
as $n\to+\infty$.
\end{proposition}

Here, we remark that the constant $a$ can be represented using particular branches of the Lambert $W$-function (see, for instance, \cite[\href{https://dlmf.nist.gov/4.13}{\S4.13}]{DLMF}):
\[
a =  - \tfrac{1}{2} - W_{\pm 1}\big(  - \tfrac{1}{{2\sqrt \e }} \mp 0\im\big).
\]

\begin{proof} First, we note that $1+\mathfrak{a}_n\le \e^{\mathfrak{a}_n}$. Consequently,
\[
1 + 2n\mathfrak{a}_n  \le \e^{n\mathfrak{a}_n },
\]
meaning
\[
2 \le \frac{{\e^{n\mathfrak{a}_n }  - 1}}{{n\mathfrak{a}_n }}.
\]
Since the function $x\mapsto (\e^x-1)/x$ is monotonically increasing for $x> 0$, the above inequality implies $a\le n\mathfrak{a}_n$. On the other hand, employing the binomial theorem yields
\[
(1 + \mathfrak{a}_n )^n \ge 1 + n\mathfrak{a}_n  + \frac{{n(n - 1)}}{2}\mathfrak{a}_n^2
\]
for any $n\ge 2$. Consequently,
\[
1 + n\mathfrak{a}_n  + \frac{{n(n - 1)}}{2}\mathfrak{a}_n^2  \le 1 + 2n\mathfrak{a}_n 
\]
which can be simplified to
\[
\mathfrak{a}_n  \le \frac{2}{{n - 1}}.
\]
To establish the asymptotic equality \eqref{athm2}, we make note of the fact that $(1 + x)\e^{-x} = 1 + \mathcal{O}(x^2)$ as $x \to 0$. Combining this with the double inequality \eqref{athm1}, we deduce that
\[
\lim\limits_{n \to  + \infty } ((1 + \mathfrak{a}_n )\e^{ - \mathfrak{a}_n } )^n  = 1.
\]
As the limit is positive and the exponential function exhibits monotonic growth, we ascertain that
\[
\limsup\limits_{n \to  + \infty } (1 + \mathfrak{a}_n )^n  = \limsup\limits_{n \to  + \infty } \e^{n\mathfrak{a}_n } \lim\limits_{n \to  + \infty } ((1 + \mathfrak{a}_n )\e^{ - \mathfrak{a}_n } )^n  = \exp \Big(\limsup\limits_{n \to  + \infty } (n\mathfrak{a}_n )\Big).
\]
Therefore,
\[
\exp \Big( \limsup\limits_{n \to  + \infty } (n\mathfrak{a}_n)  \Big) \le 1 + 2\limsup\limits_{n \to  + \infty } (n\mathfrak{a}_n ),
\]
which, given that $a \le \limsup_{n\to+\infty}(n\mathfrak{a}_n) \le 2$ as inferred from \eqref{athm1}, implies
\[ \frac{\exp \Big( \limsup\limits_{n \to  + \infty } (n\mathfrak{a}_n)  \Big)-1}{\limsup\limits_{n \to  + \infty } (n\mathfrak{a}_n )}\le 2.
\]
Using once more the monotonous nature of the function $x\mapsto (\e^x-1)/x$, we deduce
\[
\limsup\limits_{n \to  + \infty } (n\mathfrak{a}_n ) \le a,
\]
and with reference to \eqref{athm1}, this results in the desired limit
\[
\lim\limits_{n \to  + \infty } (n\mathfrak{a}_n ) = a.
\]
\end{proof}

\end{document}